\documentclass[11pt]{article}
\usepackage{amssymb,amsmath,epsfig}
\usepackage{color}
\usepackage{graphicx}
\usepackage{mst-stylefile}
\usepackage{harvard}
\usepackage[makeroom]{cancel}
\usepackage{multirow,array}
\usepackage{caption}
\usepackage{epstopdf}
\usepackage{dsfont}
\usepackage{mathalfa}

\usepackage[utf8]{inputenc}
\usepackage{tikz}
\usetikzlibrary{shapes,arrows}
\usepackage{verbatim}

\newcommand{\bbS}{{\mathbb S}}
\newcommand{\bbZ}{{\mathbb Z}}
\newcommand{\bbR}{{\mathbb R}}
\newcommand{\bbN}{{\mathbb N}}

\newcommand{\bbP}{{\mathbb P}}

\newcommand{\bbE}{{\mathbb E}}

\newcommand{\W}{\textbf{W}}

\newcommand{\supp}{\textnormal{supp}}

\newcommand{\imag}{{\mathbf i}}
\newcommand{\op}{\textnormal{op}}

\newcommand{\Naturals}{\mathbb{N}}
\newcommand{\Integers}{\mathbb{Z}}
\newcommand{\Reals}{\mathbb{R}}
\newcommand{\indicator}{\mathds{1}}

\renewcommand{\cite}{\citeyear}

\begin{document}

\title{On the empirical spectral distribution of large wavelet random matrices based on mixed-Gaussian fractional measurements in moderately high dimensions

\thanks{H.W.~was partially supported by ANR-18-CE45-0007 MUTATION, France. G.D.'s long term visits to ENS de Lyon were supported by the school, the CNRS and the Simons Foundation collaboration grant $\#714014$. The authors would like to thank Jean-Marc Lina, Ken McLaughlin and Govind Menon for their many helpful comments on this work. }
.}

\author{Patrice Abry \\ CNRS, ENS de Lyon, Laboratoire de Physique
\and   Gustavo Didier and Oliver Orejola\\ Mathematics Department\\ Tulane University
\and   Herwig Wendt\\ IRIT-ENSEEIHT, CNRS (UMR 5505),\\ Universit\'{e} de Toulouse, France}

\bibliographystyle{agsm}

\maketitle

\begin{abstract}
In this paper, we characterize the convergence of the (rescaled logarithmic) empirical spectral distribution of wavelet random matrices. We assume a moderately high-dimensional framework where the sample size $n$, the dimension $p(n)$ and, for a fixed integer $j$, the scale $a(n)2^j$
go to infinity in such a way that $\lim_{n \rightarrow \infty}p(n)\cdot a(n)/n = \lim_{n \rightarrow \infty} o(\sqrt{a(n)/n})= 0$. We suppose the underlying measurement process is a random scrambling of a sample of size $n$ of a growing number $p(n)$ of fractional processes. Each of the latter processes is a fractional Brownian motion conditionally on a randomly chosen Hurst exponent. We show that the (rescaled logarithmic) empirical spectral distribution of the wavelet random matrices converges weakly, in probability, to the distribution of Hurst exponents.
\end{abstract}


\section{Introduction}

A \textit{wavelet} is a unit $L^2(\bbR)$-norm function that annihilates polynomials (see \eqref{e:N_psi}). For a fixed (octave) $j \in \bbN \cup \{0\}$, a \textit{wavelet random matrix} is given by
\begin{equation}\label{e:W(a2^j)_intro}
\bbR^{p^2} \ni {\boldsymbol {\boldsymbol{\mathcal W}}_{n} \equiv {\boldsymbol{\mathcal W}}}(a(n)2^j) = \frac{1}{n_{a,j}}\sum^{n_{a,j}}_{k=1}{\mathcal D}(a(n)2^j,k)\hspace{0.5mm} {\mathcal D}(a(n)2^j,k)^*.
\end{equation}
In \eqref{e:W(a2^j)_intro}, $^*$ denotes transposition, $n_{a,j} = n/a(n)2^j$ is the number of wavelet-domain observations for a sample size $n$.  Each random vector ${\mathcal D}(a(n)2^j,k) \in \bbR^p$ is the \textit{wavelet transform} of a $p$-variate stochastic process $Y$ at
dyadic scale $a(n)2^j$ and shift $k \in \bbZ$ (see \eqref{e:continuous_wavelet_detail} and \eqref{e:disc2} for the continuous- and discrete-time definitions, respectively). The entries of $\{{\mathcal D}(a(n)2^j,k)\}_{k = 1,\hdots,n_{a,j}}$ are generally correlated. A \textit{fractal} is an object or phenomenon that displays the property of \textit{self-similarity}, in some sense, across a range of scales (Mandelbrot \cite{mandelbrot:1982}). Due to its intrinsic multiscale character and fine-tuned mathematical properties, the wavelet transform has been widely used in the study of fractals (e.g., Wornell \cite{wornell:1996}, Doukhan et al.\ \cite{doukhan:2003}, Massopust \cite{massopust:2014}). In this paper, we characterize the asymptotic behavior of the (rescaled logarithmic) empirical spectral distribution of large wavelet random matrices in a moderately high-dimensional framework where the \textit{sample size} $n$, the \textit{dimension} $p(n)$ and the \textit{scale} $a(n)2^j$ go to infinity. We assume measurements of the form
\begin{equation}\label{e:Y(ell)=P(n)X(ell)_intro}
\Reals^{p} \ni Y(t) = {\mathbf P}(n) X(t),
\end{equation}
either in continuous or in discrete time $t$. In \eqref{e:Y(ell)=P(n)X(ell)_intro}, the coordinates matrix ${\mathbf P}(n)$ is random and independent of $X(t)$. Moreover, each row of $X$ is, conditionally, an independent $H$-fractional Brownian motion, where each Hurst exponent $H$ is picked independently from a discrete probability distribution $\pi(dH)$. In the main results (Theorems \ref{t:main_theorem} and \ref{t:main_theorem_discrete}), we show that  the (rescaled logarithmic) empirical spectral distribution of ${\boldsymbol {\mathcal W}}(a(n)2^j)$ converges weakly, in probability, to a known affine transformation of the distribution $\pi(dH)$ of the Hurst exponents.\vspace{3mm}

In this paper, we combine two mathematical frameworks that are rarely considered jointly: $(i)$ high-dimensional probability theory; $(ii)$ fractal analysis. This is done by bringing together the study of large random matrices and scaling analysis in the wavelet domain.

Since the 1950s, the spectral behavior of large-dimensional random matrices has attracted considerable attention from the mathematical research community. For example, random matrices have proven to be prolific statistical mechanical models of Hamiltonian operators (e.g., Mehta and Gaudin \cite{mehta:gaudin:1960}, Dyson \cite{dyson:1962}, Ben Arous and Guionnet \cite{benarous:guionnet:1997}, Soshnikov \cite{soshnikov:1999}, Mehta \cite{mehta:2004}, Anderson et al.\ \cite{anderson:guionnet:zeitouni:2010}, Erd\H{o}s et al.\ \cite{erdos:yau:yin:2012}). They are also of great interest in combinatorics and numerical analysis (e.g., Baik et al.\ \cite{baik:deift:johansson:1999}, Baik et al.\ \cite{baik:deift:suidan:2016}, Menon and Trogdon \cite{menon:trogdon:2016}), as well as in the study of integrable systems and universality (Kuijlaars and McLaughlin \cite{kuijlaars:mclaughlin:2000}, Baik et al.\ \cite{baik:kriecherbauer:mclaughlin:miller:2007}, Deift \cite{deift:2007}, Tao and Vu \cite{tao:vu:2011}, Borodin and Petrov \cite{borodin:petrov:2014}, Deift \cite{deift:2017}). In particular, the literature on random matrices under dependence has been expanding at a fast pace (e.g., Xia et al.\ \cite{xia:qin:bai:2013}, Paul and Aue \cite{paul:aue:2014}, 
Chakrabarty et al.\ \cite{chakrabarty:hazra:sarkat:2016}, Che \cite{che:2017}, Steland and von Sachs \cite{steland:vonsachs:2017}, 
Wang et al.\ \cite{wang:aue:paul:2017}, Zhang and Wu \cite{zhang:wu:2017}, Erd\H{o}s et al.\ \cite{erdos:kruger:schroder:2019}, Merlev\`{e}de et al.\ \cite{merlevede:najim:tian:2019}, Bourguin et al.~\cite{bourguin:diez:tudor:2021}).

In turn, recall that the emergence of a fractal is typically the signature of a physical mechanism that generates \textit{scale invariance} (e.g., Mandelbrot \cite{mandelbrot:1982}, West et al.\ \cite{west:brown:enquist:1999}, Zheng et al.\ \cite{zheng:shen:wang:li:dunphy:hasan:brinker:su:2017}, He \cite{he:2018}, Shen et al.\ \cite{shen:stoev:hsing:2022}). Unlike traditional statistical mechanical systems (e.g., Reif \cite{reif:2009}), a scale-invariant system does not display a characteristic scale, namely, one that dominates its statistical behavior. Instead, the behavior of the system across scales is determined by specific parameters called \textit{scaling exponents}.  Scale invariance manifests itself in a wide range of natural and social phenomena such as in criticality (Sornette \cite{sornette:2006}), turbulence (Kolmogorov~\cite{Kolmogorovturbulence}), climate studies (Isotta et al.\ \cite{isotta:etal:2014}), dendrochronology (Bai and Taqqu~\cite{bai:taqqu:2018}) and hydrology (Benson et al.\ \cite{benson:baeumer:scheffler:2006}). Mathematically, it is a topic of central importance in Markovian settings (e.g., diffusion, lattice models, universality classes) as well as in non-Markovian ones (e.g., anomalous diffusion, long-range dependence, non-central limit theorems).

In a multidimensional framework, scaling behavior does not always appear along standard coordinate axes, and often involves multiple scaling relations. A $\bbR^r$-valued stochastic process $X$ is called \textit{operator self-similar} (o.s.s.; Laha and Rohatgi \cite{laha:rohatgi:1981}, Hudson and Mason \cite{hudson:mason:1982}) if it exhibits the scaling property
\begin{equation}\label{e:def_ss}
\{X(ct)\}_{t\in\bbR} \stackrel {\textnormal{f.d.d.}}{=} \{c^{\mathbf H} X(t)\}_{t\in\bbR}, \quad c>0.
\end{equation}
In \eqref{e:def_ss}, ${\mathbf H}$ is some (Hurst) matrix whose eigenvalues have real parts lying in the interval $(0,1)$ and $c^{\mathbf H} := \exp\{\log(c) {\mathbf H}\} = \sum^{\infty}_{k=0} \frac{(\log(c) {\mathbf H})^k}{k!}$. A canonical model for multivariate fractional systems is \textit{operator fractional Brownian motion} (ofBm), namely, a Gaussian, o.s.s., stationary-increment stochastic process (Maejima and Mason \cite{maejima:mason:1994}, Mason and Xiao \cite{mason:xiao:2002}, Didier and Pipiras \cite{didier:pipiras:2012}). In particular, ofBm is the natural multivariate generalization of the classical fractional Brownian motion (fBm; Embrechts and Maejima \cite{embrechts:maejima:2002}).

The importance of the role of multiple scaling laws in applications is now well established. For example, in econometrics, the detection of
distinct scaling laws in multivariate fractional time series is indicative of the key property of \textit{cointegration} -- namely,
the existence of meaningful and statistically useful long-run relationships among the individual series (e.g.,
Engle and Granger \cite{engle:granger:1987}, NobelPrize.org \cite{nobelprize:2003}, Hualde and Robinson \cite{hualde:robinson:2010}, Shimotsu \cite{shimotsu:2012}). From a different perspective, it has been shown that ignoring the presence of multiple scaling laws in statistical inference may lead to severe biases (the so-called \textit{dominance} and \textit{amplitude effects} -- see, for instance, Abry and Didier \cite{abry:didier:2018:dim2}).

In the study of scaling properties, \textit{eigenanalysis} has proven to be a fecund analytic framework in a number of areas, such as in the characterization of heavy tails and correlation (e.g., Meerschaert and Scheffler \cite{meerschaert:scheffler:1999,meerschaert:scheffler:2003}, Becker-Kern and Pap \cite{becker-kern:pap:2008}) and in time series modeling (e.g., Phillips and Ouliaris \cite{phillips:ouliaris:1988}, Li et al.\ \cite{li:pan:yao:2009}, Zhang et al.\ \cite{zhang:robinson:yao:2019}). In Abry and Didier \cite{abry:didier:2018:n-variate,abry:didier:2018:dim2}, \textit{wavelet eigenanalysis} is put forward in the construction of a general methodology for the statistical identification of the scaling (Hurst) structure of ofBm in low dimensions.

To the best of our knowledge, the behavior of the eigenvalues of large-dimensional wavelet random matrices was mathematically studied for the first time in Abry et al.\ \cite{abry:boniece:didier:wendt:2022,abry:boniece:didier:wendt:2023:regression}. This was done in the context of a high-dimensional signal-plus-noise model where measurements display a \textit{fixed} number $r$ of scaling laws, each driven by a deterministic (and unknown) Hurst exponent. In those papers, the scaling, convergence and joint fluctuations of the top $r$ eigenvalues of the associated wavelet random matrix were established and characterized.

As opposed to the ``strong" approach in Abry et al.\ \cite{abry:boniece:didier:wendt:2022,abry:boniece:didier:wendt:2023:regression}, where \textit{individual} eigenvalues are tracked, in this paper we put forward a \textit{weak} (or ``Wigner-like") \textit{approach}. In other words, instead of assuming a fixed number of deterministic scaling laws (Hurst exponents), we suppose that, as the dimension grows, measurements display an \textit{increasing} number of \textit{random} scaling laws. In applications, this provides a natural model for multiparameter high-dimensional fractal systems, such as those emerging in neuroscience and fMRI imaging (Li et al.\ \cite{li:pluta:shahbaba:fortin:ombao:baldi:2019}, Gotts et al.\ \cite{gotts:gilmore:martin:2020}), network traffic (Abry and Didier \cite{abry:didier:2018:n-variate}), climate science (Schmith et al.\ \cite{schmith:johansen:thejll:2012}) and high-dimensional time series (Merlev\`{e}de and Peligrad \cite{merlevede:peligrad:2016}, Chan et al.\ \cite{chan:lu:yau:2017}, Alshammri and Pan \cite{alshammri:pan:2021}). In this context, it is of primary interest to understand the \textit{bulk behavior} of the empirical spectral distribution (e.s.d.) of wavelet random matrices.

More specifically, following up on the computational study Orejola et al.\ \cite{orejola:didier:wendt:abry:2022}, for each $n \in \bbN$ we assume a random vector
\begin{equation}\label{e:H1,...,Hp(n)}
H_{1},\hdots,H_{p(n)} \in (0,1)^{p(n)}
\end{equation}
is sampled from a discrete probability measure $\pi(dH)$ on ${\mathcal B}(0,1)$. Then, measurements of the form \eqref{e:Y(ell)=P(n)X(ell)_intro} are made, where $X$ displays independent rows and the $\ell$-th row of $X$ is, conditionally on \eqref{e:H1,...,Hp(n)}, a fBm with Hurst parameter $H_\ell$ (on the recent use of this class of processes in the modeling of anomalous diffusion, see Balcerek et al.\ \cite{balcerek:burnecki:thapa:wylomanska:chechkin:2022}). In particular, the marginal distributions of the measurement process \eqref{e:Y(ell)=P(n)X(ell)_intro} are random sums of mixed-Gaussian laws.

In the main results of this paper (Theorems \ref{t:main_theorem} and \ref{t:main_theorem_discrete}), we describe the asymptotic behavior of the rescaled logarithmic e.s.d.\ of the wavelet random matrix
\begin{equation}\label{e:mathcal_Wn=cont_or_discrete}
{\boldsymbol {\mathcal W}}_n = {\mathbf W}(a(n)2^j)\textnormal{ or }{\boldsymbol {\mathcal W}}_n = \widetilde{{\mathbf W}}(a(n)2^j),
\end{equation}
corresponding to continuous- and discrete-time measurements, respectively. We consider both measurement frameworks because, in the former case, exact self-similarity-type relations hold (e.g., \eqref{e:cond_self-similarity}), which is very mathematically convenient. Once results are obtained for continuous-time measurements, then we turn to the realistic situation where measurements are made in discrete time. When taking limits, we consider a \textit{moderately high-dimensional regime} where
\begin{equation}\label{e:three-fold_lim}
\lim_{n \rightarrow \infty} \frac{p(n) \cdot a(n)}{n} = \lim_{n \rightarrow \infty} o\Big(\sqrt{\frac{a(n)}{n}}\Big) = 0
\end{equation}
(see \eqref{e:p(n),a(n)_conditions}). Namely, the sample size $n$, the dimension $p(n)$ and the scale $a(n)2^j$ go to infinity in a way that $n$ grows faster than $p(n)\cdot a(n)$. Over fixed scales (i.e., when $a(n)$ is assumed constant in \eqref{e:three-fold_lim}), moderately high-dimensional regimes have been studied in several works. These include Bai and Yin \cite{bai:yin:1988}, Jiang \cite{jiang:2008} and Wang et al.\ \cite{wang:aue:paul:2017} on Wigner, Jacobi and sample covariance matrices, respectively (on the use of condition \eqref{e:three-fold_lim} in this paper, see the discussion in Section \ref{s:conclusion}).

It is by considering the three-way limit \eqref{e:three-fold_lim}, which includes the \textit{scaling limit} $a(n) \rightarrow \infty$, that large wavelet random matrices may be used in the characterization of low-frequency behavior in a (moderately) high-dimensional framework. In fact, let $F_{{\boldsymbol {\mathcal W}}_n}(d\upsilon)$ be the rescaled logarithmic e.s.d.\ of ${\boldsymbol {\mathcal W}}_n$ as in \eqref{e:mathcal_Wn=cont_or_discrete} (see \eqref{e:empirical_specdist_log} and \eqref{e:empirical_specdist_log_discrete} for the precise definition of $F_{{\boldsymbol {\mathcal W}}_n}(d\upsilon)$ for continuous- and discrete-time measurements, respectively). In Theorems \ref{t:main_theorem} and \ref{t:main_theorem_discrete}, we show that, for any appropriate test function $g$,
\begin{equation}\label{e:main_theorem_intro}
\int_{\bbR} g(\upsilon) F_{{\boldsymbol {\mathcal W}}_n}(d\upsilon) \stackrel{\bbP}\rightarrow \int_{\bbR} g(\upsilon) F_{2\mathcal{H}+1}(d\upsilon), \quad n \rightarrow \infty.
\end{equation}
In other words, $F_{{\boldsymbol {\mathcal W}}_n}(d\upsilon)$ converges weakly, in probability, to the distribution given by the random variable $2\mathcal{H}+1$. Namely, it converges to the distribution of Hurst exponents up to a known affine transformation (see Figure \ref{fig:histogram}; see also Remark \ref{r:forces} and Figure \ref{fig:forces} on a schematic representation of the ``forces" acting on rescaled wavelet log-eigenvalues).
\begin{figure}
    \centering
    \includegraphics[scale=0.25]{"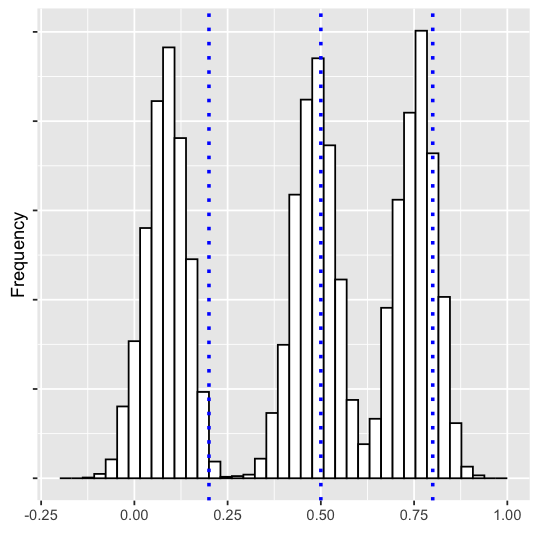"}
    \includegraphics[scale=0.25]{"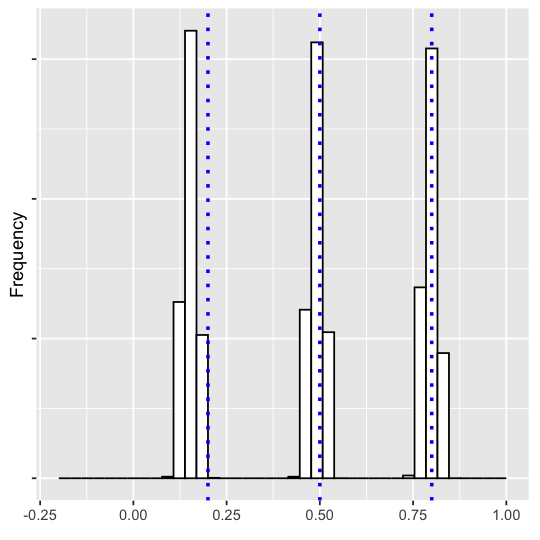"}
    \includegraphics[scale=0.25]{"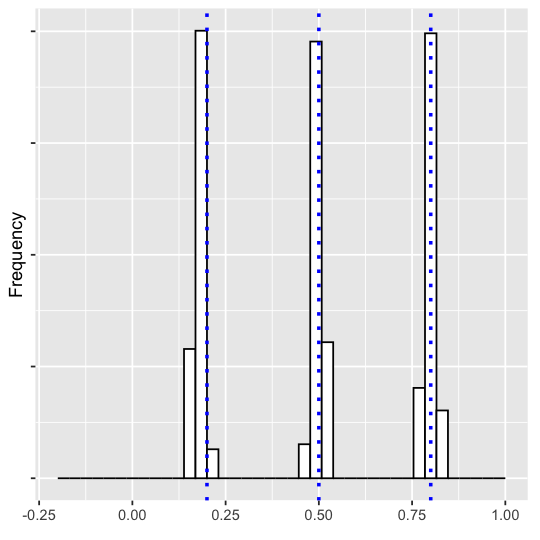"}
 \caption{ \textbf{The distribution of the (rescaled logarithmic) wavelet e.s.d.\ in the three-way limit \eqref{e:three-fold_lim}.} A Monte Carlo study displays a tri-modal distribution emerging in the rescaled logarithmic wavelet e.s.d.\ in the three-way limit \eqref{e:three-fold_lim} (\textbf{n.b.}: after applying an affine transformation, the results are shown on the same scale as that of the distribution $\pi(dH)$). In the depicted simulation study based on $1000$ realizations, $\pi(dH)$ is a discrete uniform distribution with support $\{0.2,0.5,0.8\}$. From left to right, respectively, $ (\textnormal{sample size}, \textnormal{scale},\textnormal{dimension} )  = (2^{10}, 2^4, 2^3),(2^{15}, 2^5, 2^5)$ and $(2^{18}, 2^6, 2^6)$. Wavelet log-eigenvalues weighted over multiple scales were used for enhanced (``debiased") finite-sample convergence (cf.\ Abry and Didier~\protect\cite{abry:didier:2018:n-variate} and Abry et al.\ \protect\cite{abry:boniece:didier:wendt:2022}).  }
    \label{fig:histogram}
\end{figure}

Note that \eqref{e:three-fold_lim} and \eqref{e:main_theorem_intro} stand in sharp contrast with the traditional analysis of large random matrices. With standard sample covariance matrices, for example, one considers the ratio $\lim_{n \rightarrow \infty} p(n)/n$ and the e.s.d.\ often converges to a Dirac mass at a constant or to a Mar$\breve{\textnormal{c}}$enko-Pastur distribution in moderately high and in high dimensions, respectively (e.g., Chen and Pan \cite{chen:pan:2012}, Bai and Silverstein \cite{bai:silverstein:2010}, Tao and Vu \cite{tao:vu:2012}).

In addition, statement \eqref{e:main_theorem_intro} suggests a new form of \textit{universality} that is analogous to the law of large numbers (Tao \cite{tao:2012}). In other words, regardless of the underlying $\pi(dH)$, the wavelet eigenspectrum reveals the true distribution $\pi(dH)$ of the Hurst modes in the 
moderately high-dimensional limit \eqref{e:three-fold_lim}.

From the standpoint of probability theory, to the best of our knowledge this paper provides the first mathematical study of the asymptotic behavior of the e.s.d.\ of wavelet random matrices starting from measurements containing $p(n) \rightarrow \infty$ scaling laws.

From the standpoint of fractal analysis, this paper takes a decisive step in the expansion, to the high-dimensional context, of the study of scale-invariant and non-Markovian phenomena started by Kolmogorov \cite{kolmogorov:1940} and Mandelbrot and Van Ness \cite{mandelbrot:vanness:1968}, and later taken up by the likes of Flandrin \cite{flandrin:1992}, Wornell and Oppenheim \cite{wornell:oppenheim:1992}, Meyer et al.\ \cite{meyer:sellan:taqqu:1999}, among many others (see Pipiras and Taqqu \cite{pipiras:taqqu:2017} and references therein).

This paper is organized as follows. In Section \ref{s:framework}, we provide the basic wavelet framework, definitions and assumptions used throughout the paper. In Section \ref{s:main}, we state and establish the main results on the behavior of the rescaled logarithmic e.s.d.\ of wavelet random matrices in the three-way limit \eqref{e:three-fold_lim}. The proofs of Theorems \ref{t:main_theorem} and \ref{t:main_theorem_discrete} are written so as to provide readers with immediate access to the core ideas behind the argument, whereas the auxiliary technical results can be found in the Appendices. In Section \ref{s:conclusion}, we lay out conclusions. We also briefly discuss several open problems that this work leads to in the theory and applications of wavelet random matrices. With a view to universality claims, these include potential ways of expanding the mathematical framework of the paper. 

\section{Preliminaries and assumptions}\label{s:framework}

\subsection{Notation}

Throughout the paper, we use the following notation. All with respect to the field $\bbR$, ${\mathcal M}(m_1,m_2,\bbR)$ and ${\mathcal M}(n)={\mathcal M}(n,\bbR)$ are the vector spaces, respectively, of all $m_1 \times m_2$ matrices and $n \times n$ matrices, $GL(n,\bbR)$ is the general linear group (invertible matrices), and $O(n)$ is the orthogonal group (i.e., matrices $O \in {\mathcal M}(n)$ such that $OO^* = I$). Also, ${\mathcal S}(n,\bbR)$ and ${\mathcal S}_{\geq 0}(n,\bbR)$ denote, respectively, the sets of $n \times n$ symmetric and symmetric positive semidefinite matrices, whereas $I = I_{n}$ denotes the $n\times n$ identity matrix. The $(n-1)$--dimensional unit sphere is represented by the symbol ${\bbS}^{n-1}$. For any ${\mathbf A} \in {\mathcal S}(n,\bbR)$, the ordered eigenvalues of $\mathbf{A}$ are denoted $\lambda_1({\mathbf A}) \leq \hdots \leq \lambda_n({\mathbf A})$. More generally, for ${\mathbf M} \in {\mathcal M}(n)$, the ordered singular values of ${\mathbf M}$ are written
\begin{equation}\label{e:sigma1(M)=<...=<sigman(M)}
\sigma_1({\mathbf M}) \leq \hdots \leq \sigma_n({\mathbf M}).
\end{equation}
For any ${\mathbf A}\in {\mathcal M}(n)$, $\|{\mathbf A}\|_{\textnormal{op}}:= \sigma_n({\mathbf A})$ is the operator norm of ${\mathbf A}$. The symbol $\textnormal{Haar}(O(n))$ represents the Haar probability measure on $O(n)$. The symbol $o_{\bbP}(1)$ denotes a random variable that vanishes in probability as $n\to \infty$. Given a collection of scalars $x_1,\hdots,x_n$,
\begin{equation}\label{e:x(1)=<...=<x(n)}
x_{(1)} \leq \hdots \leq x_{(n)}
\end{equation}
denotes the associated ordered $n$-tuple.

\subsection{Measurements}

In regard to the underlying stochastic framework, first consider the following definition.
\begin{definition}
Let $\pi(dH)$ be a probability measure such that $\pi(0,1) = 1$. The univariate stochastic process
\begin{equation}\label{e:X_h(t)_def}
X_\mathcal{H} = \{X_{\mathcal{H}}(t)\}_{t \in \bbR}
\end{equation}
is called a \textit{random Hurst (exponent)--fractional Brownian motion} (rH--fBm) when, conditionally on some value $\mathcal{H} = H$ picked from $\pi(d H)$, $X_{\mathcal{H}}$ is a standard fBm with Hurst exponent
\begin{equation}\label{e:H_in_(0,1]}
H \in (0,1).
\end{equation}
\end{definition}
The basic properties of rH-fBm are provided in Lemma \ref{l:properties_of_Xh(t)}.\\

So, throughout this manuscript, we make use of the following assumptions on the measurements. For expository purposes, we first state the assumptions, and then provide some interpretation. In the assumptions, ${\mathcal T}$ denotes either ${\mathcal T} = \bbR$ or ${\mathcal T} = \{1,\hdots,n\}$, corresponding to the cases of continuous- or discrete-time measurements, respectively.

\medskip

\noindent  {\sc Assumption $(A1)$}: Let $\pi(d H)$ be a distribution such that
\begin{equation}\label{e:pi(dh)}
\pi(0,1) = 1,  \quad \supp \ \pi(d H) = \{\breve{H}_1,...,\breve{H}_{r} \}, \quad \varpi :=\min \{\supp \ \pi (d H)\},
\end{equation}
where $r \in \Naturals$ and $\breve{H}_i \in (0,1)$ for each $i \in \{1,...,r\}$. For each $n \in \Naturals$,
\begin{equation}\label{e:mathds_hn}
\mathds{H}_n  = \textnormal{diag}(H_1,\hdots,H_{p(n)})\in {\mathcal M}(n)
\end{equation}
is a diagonal random matrix whose (main diagonal) entries are independently chosen from the distribution $\pi(dH)$.

\medskip

\noindent  {\sc Assumption $(A2)$}: For each $n \in \Naturals$ and given some $\mathds{H}_n$ as in \eqref{e:mathds_hn}, $p=p(n)$ independent rH--fBm sample paths are generated, each based on one of the $p$ main diagonal entries of $\mathds{H}_n$. In particular, when restricted to discrete time, a total of $p \times n$ entries is available. The associated $p$-variate (latent) stochastic process is denoted by $\{X(t)\}_{t \in {\mathcal T}}$.

\medskip

\noindent {\sc Assumption $(A3)$}: For $p = p(n)$ and $\{X(t)\}_{t \in {\mathcal T}}$ as in assumption $(A2)$, the measurements have the form
\begin{equation}\label{e:Y=PX}
\Reals^{p(n)} \ni Y(t) = {\mathbf P}(n) X(t), \quad t \in {\mathcal T}.
\end{equation}
In \eqref{e:Y=PX}, the so-named \textit{coordinates matrix} is a random matrix ${\mathbf P}(n) \in GL(p(n),\Reals)$ that is independent of $X(t)$.

\medskip

\noindent {\sc Assumption $(A4)$}: Fix $j \in \bbN \cup \{0\}$. The dimension $p(n)$ and the scaling factor $a(n)$ (cf.\ $n_{a,j}$ as in \eqref{e:W(a(n)2^j)} below) satisfy the relations
\begin{equation}\label{e:p(n),a(n)_conditions}
a(n) \leq \frac{n}{2^j}, \quad \frac{a(n)}{n} + \frac{n}{a(n)^{1+ 2\varpi}} \rightarrow 0, \quad p(n) < \frac{n}{a(n)2^j}, \quad \infty \leftarrow p(n) = o\Big(\sqrt{\frac{n}{a(n)}}\Big),
\end{equation}
as $n\to \infty$, where $\varpi$ is as in \eqref{e:pi(dh)}.

\medskip

\noindent {\sc Assumption $(A5)$}: For the random matrix ${\mathbf P}(n)$ as in \eqref{e:Y=PX},

\begin{equation}\label{e:constraints_on_coordinate_matrix}
    \frac{\log \sigma_1(\mathbf{P}(n))}{\log a(n)} \stackrel{\bbP}\to 0,\quad \frac{\log \sigma_p(\mathbf{P}(n))}{\log a(n)} \stackrel{\bbP}\to 0, \quad n \rightarrow \infty.
\end{equation}

\medskip

Assumption $(A1)$ defines the distribution of Hurst exponents. Assumption $(A2)$ postulates the latent $p(n)$-variate stochastic process $X$ as a collection of independent rH-fBms. Assumption $(A3)$ describes the observed process $Y$, where the \textit{unknown} random coordinates matrix ${\mathbf P}(n)$ determines the directions of the multiple (random) scaling relations stemming from the latent process $X$. In turn, assumption $(A4)$ controls the divergence rates among $n$, $a(n)$ and $p(n)$ in the three-way limit. In particular, it states that the scaling factor $a(n)$ and number of fractional processes, $p(n)$, must grow slower than $n$, and that the three-component ratio $\frac{p(n)a(n)}{n}$ must converge to 0 (see \eqref{e:three-fold_lim}). This establishes a moderately high-dimensional regime (cf.\ the traditional ratio $\lim_{n \rightarrow \infty} \frac{p(n)}{n}$ for sample covariance matrices). By assumption $(A5)$, the minimum and maximum singular values of ${\mathbf P}(n)$ neither vanish nor diverge too fast, respectively. Hence, ${\mathbf P}(n)$ cannot interfere too heavily on the scaling relations emanating from the latent process.\\

Throughout this manuscript, whenever convenient we write
\begin{equation}\label{e:a=a(n)_or_p=p(n)}
a = a(n)\textnormal{ or }p = p(n).
\end{equation}

\begin{remark}
    Due to the many uses of the letter ``$H$'' throughout the paper, for the readers' convenience we provide Table \ref{tab:notation_for_h} to help them keep track of the notation.
\begin{table}[]
    \centering
        \begin{tabular}{cclc}
        \textbf{notation} & \textbf{domain} & \textbf{description} & \textbf{defined in} \\ \hline
        $\mathbf{H}$ & ${\mathcal M}(p)$ & Hurst matrix &  \eqref{e:def_ss}    \\
        $\mathds{H}_n$  & ${\mathcal M}(p)$ & diagonal matrix whose main diagonal  & \eqref{e:mathds_hn} \\
                     &                   & entries are picked from $\pi(dH)$ &                \\
         $\mathbb{H}_n$  &  ${\mathcal M}(p)$ &  particular (deterministic) instance of $\mathds{H}_n$ & \eqref{e:diagonal_H_=_instance} \\
        $\mathcal{H}$ & $(0,1)$  & random Hurst exponent & \eqref{e:X_h(t)_def} \\
       $H$ & $(0,1)$  & particular (deterministic) instance of $\mathcal{H}$  & \eqref{e:H_in_(0,1]} \\
       $\breve{H}$ & $(0,1)$  & value in $\textnormal{supp}\hspace{0.5mm}\pi(dH)$  & \eqref{e:pi(dh)} \\\hline
    \end{tabular}
    \caption{Different uses of the letter ``$H$" throughout the paper.}
    \label{tab:notation_for_h}
\end{table}
\end{remark}

\subsection{Wavelet multiresolution analysis}

Recall that a \textit{wavelet} $\psi$ is a unit $L^2(\bbR)$-norm function that annihilates polynomials (see \eqref{e:N_psi}). Throughout the paper, we make use of a wavelet multiresolution analysis (MRA; see Mallat \cite{mallat:1999}, chapter 7). A wavelet MRA decomposes $L^2(\mathbb{R})$ into a sequence of \textit{approximation} (low-frequency) and \textit{detail} (high-frequency) subspaces $V_j$ and $W_j$, respectively, associated with different scales of analysis $2^{j}$, $j \in \bbZ$. In particular, given a wavelet $\psi$, there is a related \textit{scaling function} $\phi \in L^2(\bbR)$. Appropriate rescalings and shifts of $\phi$ and $\psi$ form bases for the subspaces $V_j$ and $W_j$, respectively (see Mallat \cite{mallat:1999}, Theorems 7.1 and 7.3).

In almost all mathematical statements, we make assumptions ($W1-W3$) on the underlying wavelet MRA. Such assumptions are standard in the wavelet literature and are accurately described in Section \ref{s:assumptions_on_the_MRA}. In particular, we make use of a compactly supported wavelet basis.

\subsection{Wavelet random matrices: continuous time}

For any $p \in \bbN$, suppose ${\mathcal T} = \bbR$ in \eqref{e:Y=PX}. Namely, assume we observe the continuous-time vector-valued stochastic process
\begin{equation}\label{e:infinite_sample_continuous}
\{Y(t)\}_{t \in \bbR} \subseteq \bbR^p.
\end{equation}
For a wavelet function $\psi$, the \textit{wavelet transform} vector of the stochastic process $Y$ is defined as the convolution
\begin{equation}\label{e:continuous_wavelet_detail}
\Reals^{p} \ni D(2^j,k) := 2^{-j/2} \int_\Reals \psi(2^{-j}t-k) Y(t) dt, \quad j \in \Naturals \cup \{0\}, \quad k \in \Integers,
\end{equation}
whenever such expression is meaningful. Likewise, let
\begin{equation}\label{e:DX(a(n)2^j,k)}
\Reals^{p(n)} \ni D_X(a(n)2^j,k) := 2^{-j/2} \int_\Reals \psi(2^{-j}t-k) X(t) dt
\end{equation}
be the wavelet transform of the latent process $X$ associated with $Y$. Note that, under ($A1-A5$), \eqref{e:DX(a(n)2^j,k)} is well defined in the mean squared sense by Lemma \ref{l:D(2^j,k)_is_well_defined}. Therefore, under the same conditions, \eqref{e:continuous_wavelet_detail} is also well defined in the mean squared sense.

So, let
\begin{equation}\label{e:n_in_N_wave_domain}
n \in \bbN
\end{equation}
be the wavelet-domain ``sample size", i.e., the number of wavelet coefficients available at scale $2^{0} = 1$. Also, let $j \in \bbN \cup \{0\}$ and let $a(n)$ be a dyadic scaling factor. Starting from \eqref{e:continuous_wavelet_detail}, the associated \textit{wavelet random matrices} will be denoted throughout this manuscript by
\begin{equation}\label{e:W(a(n)2^j)}
{\mathcal S}_{\geq 0}(p(n),\bbR) \ni \W(a(n)2^j) := \frac{1}{n_{a,j}} \sum^{n_{a,j}}_{k=1} D(a(n)2^j,k)D(a(n)2^j,k)^*, \quad n_{a,j} := \frac{n}{a(n)2^j}.
\end{equation}
In \eqref{e:W(a(n)2^j)}, $n_{a,j}$ may be understood as the \textit{effective sample size}, namely, the number of wavelet transform vectors $D(a(n)2^j,\cdot)$ available at scale $a(n)2^j$. For notational convenience, we assume $n_{a,j} \in \bbN$ (cf.\ the discussion around \eqref{e:nj=n/2^j}, for the case of discrete-time measurements).

Likewise, starting from \eqref{e:DX(a(n)2^j,k)}, we can define the associated wavelet random matrix
\begin{equation}\label{e:WX(a(n)2^j)}
{\mathcal S}_{\geq 0}(p(n),\bbR)  \ni \W_X(a(n)2^j) := \frac{1}{n_{a,j}} \sum^{n_{a,j}}_{k=1} D_X(a(n)2^j,k)D_X(a(n)2^j,k)^*.
\end{equation}
Note that both \eqref{e:W(a(n)2^j)} and \eqref{e:WX(a(n)2^j)} are well defined a.s. From expressions \eqref{e:continuous_wavelet_detail} and \eqref{e:WX(a(n)2^j)}, we can conveniently recast the wavelet random matrix \eqref{e:W(a(n)2^j)} for the measurements \eqref{e:infinite_sample_continuous} in the form
\begin{equation}\label{e:wavelet_transform_mixed_cont}
    \W(a(n)2^j) = {\mathbf P}(n)\W_X(a(n)2^j) {\mathbf P}(n)^*.
\end{equation}

\subsection{Wavelet random matrices: discrete time}

Now consider the realistic situation where ${\mathcal T} = \{1,\hdots,n\}$ in \eqref{e:Y=PX}. Namely, suppose we observe the discrete-time, vector-valued stochastic process
\begin{equation}\label{e:finite_sample}
\{Y(k)\}_{k=1,\hdots,n} \subseteq \bbR^p,
\end{equation}
associated with the starting scale $2^j = 1$ (or octave $j = 0$). In particular, $n \in \bbN$ in \eqref{e:finite_sample} denotes the sample size (cf.\ \eqref{e:n_in_N_wave_domain}). Starting from \eqref{e:finite_sample}, we suppose the wavelet transform vector $\widetilde{D}(2^j,k)$ of the high-dimensional process $Y$ stems from Mallat's pyramidal algorithm (Mallat \cite{mallat:1999}, chapter 7). It is given by the convolution
\begin{equation}\label{e:disc2}
\bbR^p \ni \widetilde{D}(2^j,k) := \sum_{\ell\in\bbZ} Y(\ell)h_{j,2^jk-\ell}, \quad j \in \bbN \cup \{0\},
\end{equation}
where we use the convention $Y(k) = 0$ for $k \notin \{1,\hdots,n\}$ and the filter terms are defined by
\begin{equation}\label{e:hj,l}
\bbR \ni h_{j,\ell} :=2^{-j/2}\int_\bbR\phi(t+\ell)\psi(2^{-j}t)dt, \quad \ell \in \bbZ.
\end{equation}
Due to the assumed compactness of the supports of $\psi$ and of the associated scaling function $\phi$ (see condition \eqref{e:supp_psi=compact}), only a finite number of filter terms is nonzero (Daubechies \cite{daubechies:1992}). It follows that \eqref{e:disc2} is well defined a.s. A more detailed description of Mallat's algorithm is provided in Section \ref{s:Mallats_algorithm}. For $n$, $j$, $a(n)$ and $n_{a,j}$ as in \eqref{e:continuous_wavelet_detail}, the associated \textit{wavelet random matrices} will be denoted throughout this manuscript by
\begin{equation}\label{e:W-tilde(a(n)2^j)}
{\mathcal S}_{\geq 0}(p(n),\bbR) \ni \widetilde{\W}(a(n)2^j) := \frac{1}{n_{a,j}} \sum^{n_{a,j}}_{k=1} \widetilde{D}(a(n)2^j,k)\widetilde{D}(a(n)2^j,k)^*.
\end{equation}
Likewise, let
\begin{equation}\label{e:discrete_wave_vec}
\bbR^{p(n)}\ni\widetilde{D}_X(a(n)2^j,k) := \sum_{\ell\in\bbZ} X(\ell)h_{j,2^jk-\ell}
\end{equation}
be the discrete-time wavelet transformation of the latent process $X$. As with \eqref{e:disc2}, \eqref{e:discrete_wave_vec} is well defined a.s.
Thus, we can naturally define the associated wavelet random matrix
\begin{equation*}
{\mathcal S}_{\geq 0}(p(n),\bbR)  \ni \widetilde{\W}_X(a(n)2^j) := \frac{1}{n_{a,j}} \sum^{n_{a,j}}_{k=1} \widetilde{D}_X(a(n)2^j,k)\widetilde{D}_X(a(n)2^j,k)^*
\end{equation*}
\begin{equation}\label{e:W-tilde_X(a(n)2^j)}
= \frac{1}{n_{a,j}}\widetilde{D}_X(a(n)2^j)\widetilde{D}_X(a(n)2^j)^*.
\end{equation}
In \eqref{e:W-tilde_X(a(n)2^j)}, each term $\widetilde{D}_X(a(n)2^j) \in {\mathcal M}(p(n),n_{a,j},\bbR)$ is a random matrix with $k$-th column given by \eqref{e:discrete_wave_vec}. From expressions \eqref{e:disc2} and \eqref{e:W-tilde_X(a(n)2^j)}, we can conveniently recast the wavelet random matrix \eqref{e:W-tilde(a(n)2^j)} for the measurements \eqref{e:finite_sample} in the form
\begin{equation}\label{e:wavelet_transform_mixed}
    \widetilde{\W}(a(n)2^j) = {\mathbf P}(n)\widetilde{\W}_X(a(n)2^j) {\mathbf P}(n)^*.
\end{equation}

\section{Main results} \label{s:main}

As anticipated in the Introduction, in this section we state and prove the main results, Theorems \ref{t:main_theorem} and \ref{t:main_theorem_discrete}, pertaining to continuous- and discrete-time measurements, respectively. For ease of understanding, the proofs of the theorems contain the bulk of the arguments. The auxiliary technical results can be found in the Appendices.

\subsection{Continuous time}\label{s:continuous_time}
In this section, we consider the framework in which measurements are made in continuous time (see \eqref{e:infinite_sample_continuous}). In particular, recall that, in this context, the wavelet transform is given by \eqref{e:continuous_wavelet_detail}.

Fix $j \in \bbN \cup \{0\}$ as in assumption ($A4$). For notational simplicity, write
\begin{equation}\label{e:Wn=W(a(n)2^j)}
{\mathbf W}_n = {\mathbf W}(a(n)2^j).
\end{equation}
In the first main result of this paper, we describe the asymptotic behavior of the rescaled log-eigenspectrum of the wavelet random matrix ${\mathbf W}_n$. So, to state the theorem, let
$$
\lambda_{1}\big({\mathbf W}_n\big) \leq \hdots \leq \lambda_{p(n)}\big({\mathbf W}_n\big)
$$
be the ordered eigenvalues of the wavelet random matrix ${\mathbf W}_n$. Also, let
\begin{equation}\label{e:empirical_specdist_log}
F_{{\mathbf W}_n}(d\upsilon) := \frac{1}{p}\sum_{\ell=1}^{p} \delta_{\Big\{\frac{\log \lambda_{\ell}(\mathbf{W}_n)}{\log a(n)} \leq \upsilon\Big\}}(d\upsilon)
\end{equation}
be the \textit{empirical spectral distribution} (e.s.d.) of the rescaled log-eigenvalues of ${\mathbf W}_n$. In Theorem \ref{t:main_theorem}, stated and proven next, we show that $F_{{\mathbf W}_n}(d\upsilon)$ converges weakly, in probability, to the distribution $\pi(dH)$ of the Hurst exponents up to the affine transformation $x\mapsto 2x+1$ of the latter.
\begin{theorem}\label{t:main_theorem}
Suppose assumptions ($A1-A5$) and ($W1-W3$) hold. Then, the e.s.d.\ $F_{{\mathbf W}_n}$ as in \eqref{e:empirical_specdist_log} converges weakly, in probability, to $F_{2\mathcal{H}+1}$. In other words, for any $g \in C_{b}(\bbR)$ and for any $\varepsilon > 0$,
\begin{equation}\label{e:main_theorem}
\lim_{n \rightarrow \infty}\bbP\Big(\Big|\int_{\bbR} g(\upsilon) F_{{\boldsymbol {\mathbf W}}_n}(d\upsilon) -\int_{\bbR} g(\upsilon) F_{2\mathcal{H}+1}(d\upsilon)\Big| > \varepsilon \Big)= 0.
\end{equation}
\end{theorem}
\begin{proof}
Let
\begin{equation}\label{e:empirical_cdf_log}
F_{{\mathbf W}_n}(\upsilon) := \frac{1}{p}\sum_{\ell=1}^{p} \indicator_{\Big\{\frac{\log \lambda_{\ell}({\mathbf W}_n)}{\log a(n)} \leq \upsilon\Big\}}, \quad \upsilon \in \bbR,
\end{equation}
be the \textit{empirical cumulative distribution function} (e.c.d.f.) associated with $F_{{\mathbf W}_n}(d\upsilon)$ as in \eqref{e:empirical_specdist_log}. By a natural adaptation of the proof of the Helly-Bray theorem (Shiryaev \cite{shiryaev:2015}, pp.\ 377--378), it suffices to show that, for any point of continuity $\upsilon_0$ of $F_{2\mathcal{H}+1}$,
\begin{equation}\label{e:consistency_log(eigen)/log(scale)}
F_{{\mathbf W}_n}(\upsilon_0) \stackrel{\bbP}\rightarrow F_{2\mathcal{H}+1}(\upsilon_0), \quad n \rightarrow \infty.
\end{equation}

So, towards establishing \eqref{e:consistency_log(eigen)/log(scale)}, recall expressions \eqref{e:WX(a(n)2^j)} and \eqref{e:wavelet_transform_mixed_cont} for $\W_X(a(n)2^j)$ and $\W(a(n)2^j)$, respectively. Now define
\begin{equation}
\label{e:waveletmatrix}
{\mathcal S}_{\geq 0}(p(n),\bbR)  \ni \mathbf{W}_X(2^j) := \frac{1}{n_{a,j}} \sum^{n_{a,j}}_{k=1} D_X(2^j,k)D_X(2^j,k)^* =\frac{1}{n_{a,j}} D_X(2^j)D_X(2^j)^*,
\end{equation}
where $ D_X(2^j) \in {\mathcal M}(p(n),n_{a,j},\bbR)$ is a random matrix with $k$-th column given by the vector $D_X(2^j,k) \in \Reals^{p(n)}$
(\textbf{n.b.}: the wavelet random matrix $\mathbf{W}_X(2^j)$ involves $n_{a,j} = \frac{n}{a(n)2^j}$ terms, which depends on the scaling factor $a(n)$). Starting from the wavelet random matrix $\mathbf{W}_X(2^j)$ as in \eqref{e:waveletmatrix}, let
\begin{equation}\label{e:W_prime_matrix}
    \breve{\mathbf{W}}(a(n)2^j) := \mathbf{P}(n) a(n)^{ \mathds{H}_n+\ (1/2)I}\mathbf{W}_X(2^j) a(n)^{ \mathds{H}_n+\ (1/2)I} \mathbf{P}(n)^*,
\end{equation}
where $\mathds{H}_n$ is given by \eqref{e:mathds_hn}. Then, as a consequence of the random operator self-similarity relation \eqref{e:W(a(n)2^j)_self_similarity} (see Lemma \ref{l:W(a(n)2^j)_D(a(n)2^j,k)_selfsimilar}, $(ii)$), as well as of expressions \eqref{e:wavelet_transform_mixed_cont} and \eqref{e:W_prime_matrix},
\begin{equation}\label{e:lambda-ell_=d_sigma-ell_mixed}
\big\{  \lambda_{\ell}\big(\mathbf{W}(a(n)2^j)\big)\big\}_{\ell = 1,\hdots, p(n)}
 \stackrel{d}= \big\{  \lambda_{\ell}\big(\breve{\mathbf{W}}(a(n)2^j)\big)\big\}_{\ell = 1,\hdots, p(n)}.
\end{equation}
For convenience of notation, write $ \breve{\mathbf{W}}_n = \breve{\mathbf{W}}(a(n)2^j)$. So, consider the e.c.d.f.\ of the rescaled log-eigenvalues coming from the right-hand side of \eqref{e:lambda-ell_=d_sigma-ell_mixed}, i.e.,
\begin{equation}\label{e:F_{W_n'}(upsilon)}
F_{{\breve{\mathbf{W}}}_{n}}(\upsilon):= \frac{1}{p}\sum_{\ell=1}^{p} \indicator_{\Big\{\frac{\log \lambda_{\ell}(\breve{\mathbf{W}}_n)}{\log a(n)} \leq \upsilon \Big\}}, \quad \upsilon \in \bbR.
\end{equation}
Note that, as a consequence of Lemma \ref{l:fixed_scale_log_limit}, $(i)$, and of the fact that ${\mathbf P}(n) \in GL(p(n),\bbR)$ (assumption $(A3)$), \eqref{e:F_{W_n'}(upsilon)} is well defined a.s. Let $\upsilon_0$ be a continuity point of $F_{2\mathcal{H}+1}$ and suppose, for the moment, that
\begin{equation}\label{e:F_{W_n'}(upsilon_0)_conv}
F_{{\breve{\mathbf{W}}}_{n}}(\upsilon_0) \stackrel{\bbP} \to F_{2\mathcal{H}+1}(\upsilon_0), \quad n \to \infty.
\end{equation}
Also, consider $F_{{\mathbf W}_{n}}$ as defined in \eqref{e:empirical_cdf_log}. We obtain
\begin{equation}\label{e:F_Wn(upsilon0)=F_W'n(upsilon0)}
F_{{\mathbf W}_{n}}(\upsilon_0)\stackrel{d}= F_{{\breve{\mathbf{W}}}_{n}}(\upsilon_0) \stackrel{d} \to F_{2\mathcal{H}+1}(\upsilon_0), \quad n \rightarrow \infty,
\end{equation}
where equality in distribution in \eqref{e:F_Wn(upsilon0)=F_W'n(upsilon0)} is a consequence of \eqref{e:lambda-ell_=d_sigma-ell_mixed}. Since, in addition, $F_{2\mathcal{H}+1}(\upsilon_0)$ is a constant, then $F_{{\mathbf W}_{n}}(\upsilon_0) \stackrel{\bbP}\to  F_{2\mathcal{H}+1}(\upsilon_0)$ as $n\to\infty$. This shows \eqref{e:consistency_log(eigen)/log(scale)}.

So, we need to establish the convergence \eqref{e:F_{W_n'}(upsilon_0)_conv}. For this purpose, we first construct \textit{almost sure }bounds for the eigenvalues $\lambda_{\ell}(\breve{\mathbf{W}}(a(n)2^j))$ as in \eqref{e:lambda-ell_=d_sigma-ell_mixed}. 
We do this by utilizing Weyl's inequalities for matrix products \eqref{e:Weyls_inequalities_products}, and also by applying twice the cyclic property of the eigenvalues of symmetric positive semidefinite matrices \eqref{e:cyclic_property}. In fact, recall the notation \eqref{e:sigma1(M)=<...=<sigman(M)} for singular values. Starting from expression \eqref{e:W_prime_matrix}, first we obtain
$$
     \lambda_{\ell}\big(a(n)^{\mathds{H}_n+(1/2)I} \mathbf{W}_X(2^j) a(n)^{\mathds{H}_n+(1/2)I} \big) \cdot \sigma^2_{1} \big( \mathbf{P}(n)\big)  \leq \lambda_{\ell}\big(\breve{\mathbf{W}}(a(n)2^j)\big)
$$
\begin{equation}\label{e:lambda_ell(W-breve)_bounds}
\leq  \lambda_{\ell}\big(a(n)^{\mathds{H}_n+(1/2)I} \mathbf{W}_X(2^j) a(n)^{\mathds{H}_n+(1/2)I} \big) \cdot \sigma^2_{p} \big( \mathbf{P}(n)\big), \quad \quad \ell = 1,\hdots, p(n).
\end{equation}
Second, based on \eqref{e:lambda_ell(W-breve)_bounds}, we arrive at
$$
     \lambda_{\ell}\big(a(n)^{2\mathds{H}_n+I} \big) \cdot \lambda_{1} \big( \mathbf{W}_X(2^j)\big)\cdot \sigma^2_{1} \big( \mathbf{P}(n)\big)  \leq \lambda_{\ell}\big(\breve{\mathbf{W}}(a(n)2^j)\big)
$$
\begin{equation}
\label{e:eigenvaluebounds}
\leq \lambda_{\ell}\big(a(n)^{ 2\mathds{H}_n+I} \big) \cdot \lambda_{p} \big( \mathbf{W}_X(2^j)\big) \cdot \sigma^2_{p} \big( \mathbf{P}(n)\big),
\quad \quad \ell = 1,\hdots, p(n).
\end{equation}
Recall that, by Lemma \ref{l:fixed_scale_log_limit}, $(i)$, $\log \lambda_\ell(\mathbf{W}_X(2^j))$ is well defined a.s.\ for $\ell = 1,\hdots,p$. Therefore, from expression \eqref{e:eigenvaluebounds}, taking logs while dividing by $\log a(n)$ yields
\begin{equation}
\label{e:log_eigenvalue_bound1}
\lambda_{\ell}\big(2 \mathds{H}_n+ I \big) +  G_{L}(n) \leq \frac{\log \lambda_{\ell}\big(\breve{\mathbf{W}}(a(n)2^j)\big) }{\log a(n)}
\leq \lambda_\ell \big(2 \mathds{H}_n+I \big) +  G_{U}(n),
\end{equation}
for $\ell=1,\hdots,p$. In \eqref{e:log_eigenvalue_bound1}, we define
\begin{equation}\label{e:G_lower}
 G_{L}(n) := \frac{\log \lambda_{1} \big( \mathbf{W}_X(2^j)\big)  }{\log a(n)}  +\frac{2 \log \sigma_{1}\big(\mathbf{P}(n)\big) }{\log a(n)}
\end{equation}
and
\begin{equation} \label{e:G_upper}
       G_{U}(n) :=  \frac{\log \lambda_{p} \big( \mathbf{W}_X(2^j)\big)  }{\log a(n)} +\frac{2 \log \sigma_{p}\big(\mathbf{P}(n)\big) }{\log a(n)}.
\end{equation}
Now observe that by making use of expression \eqref{e:constraints_on_coordinate_matrix} (see assumption $(A5)$) and of Lemma \ref{l:fixed_scale_log_limit}, $(ii)$, it follows that
\begin{equation}
\label{e:G_limits}
   G_{L}(n) \stackrel{\bbP}\to  0 \quad \text{ and } \quad  G_{U}(n) \stackrel{\bbP}\to 0,
\end{equation}
as $n \rightarrow \infty$. Further note that
\begin{equation}
    \label{e:eigenvalues_are_samples}
     \lambda_{\ell }\big( 2\mathds{H}_n+ I \big)  = 2 \mathcal{H}_{(\ell)} + 1, \quad \ell = 1,\hdots,p(n)
\end{equation}
(cf.\ \eqref{e:x(1)=<...=<x(n)}), which stems from the fact that $\mathds{H}_n$ as in \eqref{e:mathds_hn} is a diagonal matrix.

So, we can now rewrite the bounds \eqref{e:log_eigenvalue_bound1} in terms of the limits \eqref{e:G_limits} and the equalities \eqref{e:eigenvalues_are_samples}. That is, for every $n$ and each $\ell = 1,\hdots, p(n)$, recast
\begin{equation}
\label{e:bounds_for_log_wavelet_eigenvalues}
   (2\mathcal{H}_{(\ell)}+1) + o_{\bbP}(1)_L  \leq \frac{\log \lambda_{\ell}\big(\breve{\mathbf{W}}(a(n)2^j) \big)}{\log a(n)}  \leq (2 \mathcal{H}_{(\ell)}+1) +  o_{\bbP}(1)_U
\end{equation}
(\textbf{n.b.}: the subscripts $L$ and $U$ are used simply to make a distinction between the vanishing terms in the lower and upper bounds, respectively). Now fix $\upsilon \in \bbR$ and consider the e.c.d.f.s associated with the middle, right and left terms in \eqref{e:bounds_for_log_wavelet_eigenvalues}, respectively. In other words, for $\upsilon \in \bbR$, consider $F_{{\mathbf W}_{n}}(\upsilon)$ as in \eqref{e:F_{W_n'}(upsilon)}, as well as
\begin{equation}\label{e:F_Ln(upsilon)_F_Un(upsilon)}
F_{L_n}(\upsilon):= \frac{1}{p}\sum_{\ell=1}^{p} \indicator_{\big\{(2\mathcal{H}_{(\ell)}+1) +  o_{\bbP}(1)_L \hspace{0.5mm} \leq \upsilon \big\}} \hspace{2mm}\textnormal{and}\hspace{2mm} F_{U_n}(\upsilon):= \frac{1}{p}\sum_{\ell=1}^{p} \indicator_{\big\{(2 \mathcal{H}_{(\ell)}+1) + o_{\bbP}(1)_U  \hspace{0.5mm}\leq \upsilon \big\}}.
\end{equation}
Then, for each $\upsilon \in \bbR$, the bounds \eqref{e:bounds_for_log_wavelet_eigenvalues} imply the (random) stochastic dominance relation
\begin{equation}\label{e:cdfbounds}
    F_{U_{n}}(\upsilon) \leq F_{\breve{\mathbf{W}}_{n}}(\upsilon) \leq F_{L_{n}}(\upsilon) \quad \textnormal{a.s.}
\end{equation}
Now let $\upsilon_0$ be a point of continuity of $F_{2\mathcal{H}+1}$ as in \eqref{e:F_{W_n'}(upsilon_0)_conv}. By setting $E_n = o_{\bbP}(1)_U$ and $E_n = o_{\bbP}(1)_L$ in Lemma \ref{l:difference_cdfs}, ($ii$), we obtain, respectively,
\begin{equation}\label{e:F_Un->F_2H+1_F_Ln->F_2H+1}
F_{U_n}(\upsilon_0) \stackrel{\bbP} \to F_{2\mathcal{H}+1}(\upsilon_0) \hspace{3mm}\textnormal{and} \hspace{3mm}F_{L_n}(\upsilon_0) \stackrel{\bbP} \to F_{2\mathcal{H}+1}(\upsilon_0), \quad  n \to \infty.
\end{equation}
As a consequence, from the bounds \eqref{e:cdfbounds} it follows that \eqref{e:F_{W_n'}(upsilon_0)_conv} holds, as claimed. $\Box$\\
\end{proof}

\begin{remark}\label{r:forces}
It is informative to revisit the proof of Theorem \ref{t:main_theorem} in light of some basic concepts from classical random matrix theory, pertaining to \textit{fixed}-scale behavior.

So, consider the simplest framework of a single Hurst mode (i.e., $r=1$ in \eqref{e:pi(dh)}). Now recall the fact that the wavelet transform quasi-decorrelates the original measurements in the sense that the wavelet coefficients display fast-decaying correlations (e.g., Flandrin \cite{flandrin:1992}, Bardet \cite{bardet:2002}). As a consequence of this property, over fixed scales the wavelet random matrix $n_{a,j} {\mathbf W}(a(n)2^j)$ can be \textit{approximated} by a classical Wishart matrix up to some constant $C_1 = C_1(p,a(n),j,n) > 0$. In other words, we can write $C_1 \cdot n_{a,j} {\mathbf W}(a(n)2^j) \stackrel{d}\approx {\mathbf M}^*{\mathbf M}$, where the entries of the random matrix ${\mathbf M} \in {\mathcal M}(n_{a,j},p,\bbR)$ are i.i.d.\ ${\mathcal N}(0,1)$ random variables. In turn, for some constant $C_2= C_2(p,n_{a,j}) > 0$, the joint density $f$ of the ordered eigenvalues of ${\mathbf M}^*{\mathbf M}$ satisfies
\begin{equation}\label{e:samplecov_joint_eigenval_density}
f(\lambda_1,\hdots,\lambda_p) = C_2 \cdot \indicator_{\{0 \leq \lambda_1 \leq \hdots \leq \lambda_p\}} \cdot \prod^{p}_{\ell=1}\lambda^{(n_{a,j}-p-1)/2}_\ell\cdot \exp\Big\{ - \sum^{p}_{\ell=1}\lambda_\ell/2 \Big\} \cdot \underbrace{\prod_{i < \ell} (\lambda_{\ell} - \lambda_i)}_{\textnormal{eigenvalue repulsion}}
\end{equation}
(e.g., Anderson \cite{anderson:2003}, p.\ 539). Based on the right-hand side of \eqref{e:samplecov_joint_eigenval_density}, we can interpret that, over \textit{fixed} scales, there is a ``force" based on which (wavelet) eigenvalues mutually repel (cf.\ Menon and Trogdon \cite{menon:trogdon:2023}, chapter 8, or Tao \cite{tao:2012}, pp.\ 48--49). These conclusions naturally extend to log-eigenvalues.

The situation is markedly different over \textit{large} scales. The proof of Theorem \ref{t:main_theorem} (cf.\ relation \eqref{e:bounds_for_log_wavelet_eigenvalues}) allows us to interpret that two distinct ``fractal forces" -- i.e., both stemming from scaling effects -- act on the rescaled wavelet log-eigenvalues. One ``force" is repulsive, the other one, attractive (see Figure \ref{fig:forces} for a schematic representation).

The repulsive force appears when $\pi(dH)$ displays more than one Hurst mode (i.e., $r>1$ in \eqref{e:pi(dh)}). It splits log-eigenvalues according to the different scaling laws trickling into the system over high dimensions. It mainly results from a combination of: $(i)$ the scaling property of wavelet random matrices (see \eqref{e:lambda-ell_=d_sigma-ell_mixed}); $(ii)$ the variational characterization of eigenvalues, which appears implicitly in \eqref{e:lambda_ell(W-breve)_bounds}; $(iii)$ the strong law of large numbers, which ensures that Hurst exponents are sampled asymptotically with the proportions given by the distribution $\pi$ (see \eqref{e:F_Un->F_2H+1_F_Ln->F_2H+1} and its implicit connection with \eqref{e:eigenvalueboundsSigma}).

On the other hand, the attractive force eventually overcomes fixed-scale repulsion (cf.\ \eqref{e:samplecov_joint_eigenval_density}) and causes rescaled wavelet log-eigenvalues to collapse around each Hurst mode. This force stems from the vanishing small-scale features of the system under rescaling (see \eqref{e:G_limits}).

As a consequence of the action of these two ``forces", in the three-way limit \eqref{e:p(n),a(n)_conditions} (in particular, in the \textit{large} scale limit), we are left with the distribution of Hurst (scaling) exponents in the rescaled logarithmic e.s.d.\ of wavelet random matrices.
\end{remark}


\subsection{Discrete time}

Vis-{\` a}-vis Section \ref{s:continuous_time}, in this section we consider the more realistic situation in which measurements are made in discrete time (see \eqref{e:finite_sample}). In particular, recall that, in this context, the wavelet transform is given by \eqref{e:disc2}.

The following theorem, Theorem \ref{t:main_theorem_discrete}, is the second main result of the paper. It is the discrete-time counterpart of Theorem \ref{t:main_theorem}. The fundamental conclusions in both theorems are analogous, as well as their proofs. However, the proof of Theorem \ref{t:main_theorem_discrete} relies on developments for the continuous-time framework, since it involves controlling the difference between the spectra of wavelet random matrices for continuous- and discrete-time measurements (see Remark \ref{r:deviation_cont_discrete_time} for a short description).

Fix $j \in \bbN \cup \{0\}$ as in assumption ($A4$). As in Section \ref{s:continuous_time}, and, for notational simplicity, write
\begin{equation}\label{e:Wn=W(a(n)2^j)_discrete}
\widetilde{{\mathbf W}}_n = \widetilde{{\mathbf W}}(a(n)2^j).
\end{equation}
To state the theorem, let
$$
\lambda_{1}\big(\widetilde{{\mathbf W}}_n\big) \leq \hdots \leq \lambda_{p(n)}\big(\widetilde{{\mathbf W}}_n\big)
$$
be the ordered eigenvalues of the wavelet random matrix $\widetilde{{\mathbf W}}_n$. Also, let
\begin{equation}\label{e:empirical_specdist_log_discrete}
F_{\widetilde{{\mathbf W}}_n}(d\upsilon) := \frac{1}{p}\sum_{\ell=1}^{p} \delta_{\Big\{\frac{\log \lambda_{\ell}(\widetilde{{\mathbf W}}_n)}{\log a(n)} \leq \upsilon\Big\}}(d\upsilon)
\end{equation}
be the e.s.d.\ of the rescaled log-eigenvalues of $\widetilde{{\mathbf W}}_n$. In Theorem \ref{t:main_theorem_discrete}, stated and proven next, we show that $F_{\widetilde{{\mathbf W}}_n}(d\upsilon)$ converges weakly, in probability, to the distribution $\pi(dH)$ of the Hurst exponents up to an affine transformation of the latter.
\begin{theorem}\label{t:main_theorem_discrete}
Suppose assumptions ($A1-A5$) and ($W1-W3$) hold. Then, the e.s.d.\ $F_{{\widetilde{\mathbf W}}_n}$ as in \eqref{e:empirical_specdist_log_discrete} converges weakly, in probability, to $F_{2\mathcal{H}+1}$. In other words, for any $g \in C_{b}(\bbR)$ and for any $\varepsilon > 0$,
\begin{equation}\label{e:main_theorem_discrete}
\lim_{n \rightarrow \infty}\bbP\Big(\Big|\int_{\bbR} g(\upsilon) F_{{\boldsymbol {\widetilde{\mathbf W}}}_n}(d\upsilon) -\int_{\bbR} g(\upsilon) F_{2\mathcal{H}+1}(d\upsilon)\Big| > \varepsilon \Big)= 0.
\end{equation}
\end{theorem}
\begin{proof}
The proof follows a similar argument to the one for showing Theorem \ref{t:main_theorem}. Let
\begin{equation}\label{e:empirical_cdf_log_discrete}
F_{{\widetilde{\mathbf W}}_n}(\upsilon) := \frac{1}{p}\sum_{\ell=1}^{p} \indicator_{\Big\{\frac{\log \lambda_{\ell}(\widetilde{\mathbf{W}}_n)}{\log a(n)} \leq \upsilon\Big\}}, \quad \upsilon \in \bbR,
\end{equation}
be the e.c.d.f.\ associated with $F_{\widetilde{{\mathbf W}}_n}(d\upsilon)$ as in \eqref{e:empirical_specdist_log_discrete}. As in the proof of Theorem \ref{t:main_theorem}, it suffices to show that, for any point of continuity $\upsilon_0$ of $F_{2\mathcal{H}+1}$,
\begin{equation}\label{e:consistency_log(eigen)/log(scale)_discrete}
F_{{\widetilde{\mathbf W}}_n}(\upsilon_0) \stackrel{\bbP}\rightarrow F_{2\mathcal{H}+1}(\upsilon_0), \quad n \rightarrow \infty.
\end{equation}

So, recall relation \eqref{e:W-tilde_X(a(n)2^j)}. We define the \textit{auxiliary wavelet random matrix}
$$
{\mathcal S}_{\geq 0}(p(n),\bbR) \ni \widetilde{\mathbf{B}}_a(2^j) := a(n)^{-(\mathds{H}_n+(1/2)I) } \widetilde{{\mathbf W} }_X(a(n)2^j)
a(n)^{-(\mathds{H}_n+(1/2)I) }
$$
$$
= a(n)^{-(\mathds{H}_n+(1/2)I) }  \frac{1}{n_{a,j}}\sum^{n_{a,j}}_{k=1}\widetilde{D}_X(a(n)2^j,k)\widetilde{D}_X(a(n)2^j,k)^*
a(n)^{-(\mathds{H}_n+(1/2)I) }
$$
\begin{equation}\label{e:Btilde-nu}
= a(n)^{-(\mathds{H}_n+(1/2)I) }  \frac{1}{n_{a,j}}\widetilde{D}_X(a(n)2^j)\widetilde{D}_X(a(n)2^j)^*
a(n)^{-(\mathds{H}_n+(1/2)I) }.
\end{equation}
Now observe that, by \eqref{e:wavelet_transform_mixed}, we can express the discrete-time wavelet random matrix \eqref{e:Wn=W(a(n)2^j)_discrete} as
\begin{equation}\label{e:discrete_wavelet_in_distribution}
    \widetilde{\mathbf{W}}_n  =\mathbf{P}(n) a(n)^{\mathds{H}_n+(1/2)I}\widetilde{\mathbf{B}}_a(2^j) a(n)^{\mathds{H}_n+(1/2)I}\mathbf{P}(n)^*,
\end{equation}
where $\widetilde{\mathbf{B}}_a(2^j)$ is as in \eqref{e:Btilde-nu}.

As in the proof of Theorem \ref{t:main_theorem}, by Weyl's inequalities for matrix products \eqref{e:Weyls_inequalities_products} we obtain
$$
     \lambda_{\ell}\Big(a(n)^{\mathds{H}_n+(1/2)I}\widetilde{\mathbf{B}}_a(2^j) a(n)^{\mathds{H}_n+(1/2)I} \Big) \cdot \sigma^2_{1} \big( \mathbf{P}(n)\big)  \leq \lambda_{\ell}\big(\widetilde{\mathbf{W}}_n\big)
$$
$$
\leq  \lambda_{\ell}\Big(a(n)^{\mathds{H}_n+(1/2)I}\widetilde{\mathbf{B}}_a(2^j)  a(n)^{\mathds{H}_n+(1/2)I} \Big) \cdot \sigma^2_{p} \big( \mathbf{P}(n)\big), \quad \quad \ell = 1,\hdots, p(n).
$$
Then, an application of \eqref{e:cyclic_property} and \eqref{e:Weyls_inequalities_products} yields upper and lower bounds
$$
     \lambda_{\ell}\big(a(n)^{2\mathds{H}_n+I} \big) \cdot \lambda_{1} \big(\widetilde{\mathbf{B}}_a(2^j) \big)\cdot \sigma^2_{1} \big( \mathbf{P}(n)\big)  \leq \lambda_{\ell}\big(\widetilde{\mathbf{W}}_n \big)
$$
\begin{equation}
\label{e:eigenvaluebounds_discrete}
\leq \lambda_{\ell}\big(a(n)^{ 2\mathds{H}_n+I} \big) \cdot \lambda_{p} \big(\widetilde{\mathbf{B}}_a(2^j) \big) \cdot \sigma^2_{p} \big( \mathbf{P}(n)\big) , \quad \ell = 1,\hdots, p(n).
\end{equation}
By Lemma \ref{l:fixed_scale_log_limit_discrete}, $(i)$, $\log \lambda_\ell(\widetilde{\mathbf{B}}_a(2^j))$ is well defined a.s.\ for $\ell = 1,\hdots,p(n)$. Therefore, from expression \eqref{e:eigenvaluebounds_discrete}, taking logs while dividing by $\log a(n)$ yields
\begin{equation}
\label{e:log_eigenvalue_bound1_discrete}
\lambda_{\ell}\big( 2\mathds{H}_n+ I \big) +  \widetilde{G}_{L}(n) \leq \frac{\log \lambda_{\ell}\big( \widetilde{\mathbf{W}}_n \big) }{\log a(n)}
\leq  \lambda_\ell \big( 2\mathds{H}_n+I \big) +  \widetilde{G}_{U}(n),
\end{equation}
for $\ell=1,\hdots,p$. In \eqref{e:log_eigenvalue_bound1_discrete}, we define
\begin{equation*}
 \widetilde{G}_{L}(n) := \frac{\log  \lambda_{1} \big(\widetilde{\mathbf{B}}_a(2^j)\big)   }{\log a(n)} + \frac{2 \log \sigma_{1}\big(\mathbf{P}(n)\big) }{\log a(n)}, \quad \widetilde{G}_{U}(n) := \frac{\log  \lambda_{p} \big(\widetilde{\mathbf{B}}_a(2^j)\big) }{\log a(n)} + \frac{2 \log \sigma_{p}\big(\mathbf{P}(n)\big) }{\log a(n)}.
\end{equation*}
Moreover, by Lemma  \ref{l:fixed_scale_log_limit_discrete}, $(ii)$, and \eqref{e:constraints_on_coordinate_matrix} (see assumption ($A5$)),
$\widetilde{G}_{L}(n) \stackrel{\bbP}\to  0$ and $\widetilde{G}_{U}(n) \stackrel{\bbP}\to 0$, as $n \rightarrow \infty$. The remainder of the proof is identical to that of Theorem \ref{t:main_theorem}, \textit{mutatis mutandis}. This establishes \eqref{e:main_theorem_discrete}.
 $\Box$\\
\end{proof}

\begin{remark}\label{r:deviation_cont_discrete_time}
Proving Theorem \ref{t:main_theorem_discrete} involves developments for the continuous-time framework since it requires controlling the deviation of the auxiliary wavelet random matrix \eqref{e:Btilde-nu} from its continuous-time analogue \eqref{e:B^(j)}. In fact, this control in terms of the operator norm is provided in Lemma \ref{l:vecBtilde(2^j)_asympt}. In turn, this lemma is used in the proof of Lemma \ref{l:fixed_scale_log_limit_discrete}, $(ii)$, which the proof of Theorem \ref{t:main_theorem_discrete} relies on.
\end{remark}

\begin{figure}
    \centering
    \includegraphics[scale=0.5]{"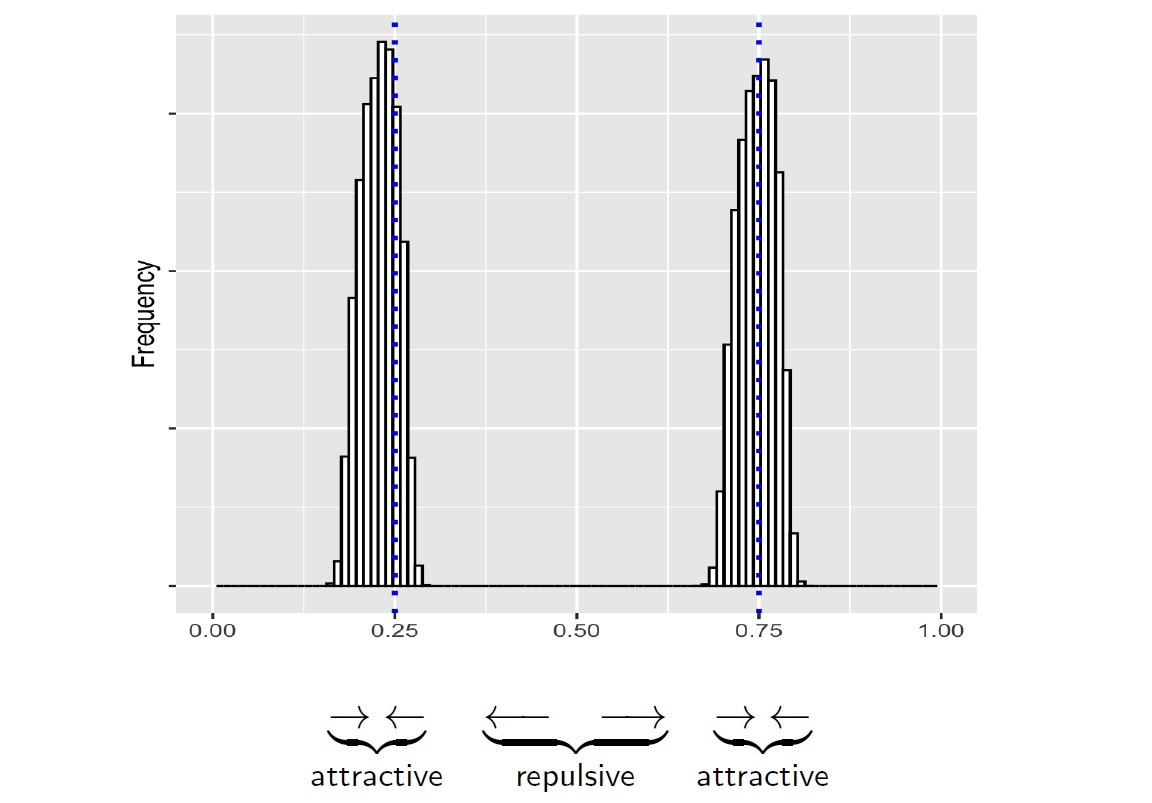"}
 \caption{ \textbf{Schematic representation of the ``forces" acting on the rescaled log-eigenvalues in the three-way limit \eqref{e:three-fold_lim}.  }  We can interpret that two distinct ``fractal forces" -- one repulsive, one attractive -- act on the rescaled wavelet log-eigenvalues. In the three-way limit \eqref{e:p(n),a(n)_conditions}, as a result of these forces only the large-scale features of the system remain. We are then left with the distribution of the Hurst exponents, as illustrated in Figure \ref{fig:histogram} (see Remark \ref{r:forces} for a discussion).   }
    \label{fig:forces}
\end{figure}

\section{Conclusion, discussion and open problems}\label{s:conclusion}

In this paper, we characterize the convergence of the (rescaled logarithmic) e.s.d.\ of wavelet random matrices. We assume a moderately high-dimensional framework where the sample size $n$, the dimension $p(n)$ and the scale $a(n)2^j$ go to infinity. We suppose the underlying measurement process is a random scrambling of a sample of size $n$ of a growing number $p(n)$ of fractional processes. Each of the latter processes is a fractional Brownian motion conditionally on a randomly chosen Hurst exponent. In our main results (Theorems \ref{t:main_theorem} and \ref{t:main_theorem_discrete}), we show that the (rescaled logarithmic) e.s.d.\ of the wavelet random matrices converges weakly, in probability, to the distribution of Hurst exponents.

This work leads to a number of open questions, which we now briefly discuss.

With a view to universality claims, it is of great interest to better understand the extent to which the main results of this paper hold while some assumptions get lifted.

The moderately high-dimensional condition
\begin{equation}\label{e:p=o(sqrt(n/a(n)))}
p(n) = o\Big(\sqrt{\frac{n}{a(n)}}\Big)
\end{equation}
(see \eqref{e:p(n),a(n)_conditions}) is crucially used only once in this paper, namely, in the convergence of the truncated Haar random matrix $\|U_{\textnormal{trunc}}\|_{\textnormal{op}} \stackrel{\bbP}\to 0$ (see \eqref{e:||U_trunc||=<||U_trunc,p||} and \eqref{e:||U_trunc,p||=oP(1)}). This convergence undergirds Proposition \ref{p:D_1_D_2_sum_bound} (see the bound \eqref{e:diagonal_bound}), which in turn is used in the construction of the fixed-scale bounds \eqref{e:eigenvalueboundsSigma} in Lemma \ref{l:fixed_scale_log_limit} and related auxiliary results. This leads to two open problems, namely, can the main conclusions of this paper be generalized from \eqref{e:p=o(sqrt(n/a(n)))} to $(i)$ the most general moderately high-dimensional framework $\lim_{n \rightarrow \infty} p(n)\cdot a(n)/n = 0$?; or $(ii)$ to the high-dimensional framework $\lim_{n \rightarrow \infty}p(n)\cdot a(n)/n = c > 0$?

Likewise, it is in Lemma \ref{l:fixed_scale_log_limit} (and related auxiliary results) that we use the assumptions of a discrete distribution $\pi(dH)$ and of the independence among rH-fBms making up the latent process in \eqref{e:Y=PX}. This leads to three other open problems, namely, can the main conclusions of this paper be generalized assuming $(iii)$ a general Borel measure $\pi(dH)$ on (0,1)?; or $(iv)$ a latent process $X$ with dependent rows? (here, preliminary simulation studies suggest a positive answer); or $(v)$ a non-Gaussian latent process?

The condition $N_{\psi}\geq 2$ (see assumption ($W1$)) is used in Lemma \ref{l:D(2^j,k)_is_stationary}, ($ii$). In turn, this lemma is used in Lemma \ref{l:fixed_scale_log_limit}, based on the bounds \eqref{e:gray_toeplitz_bound}. However, the case where
\begin{equation}\label{e:NPsi=1}
N_{\psi} = 1
\end{equation}
is still of interest because of the following reason. Heuristically, under the choice \eqref{e:NPsi=1}, the wavelet random matrix ${\boldsymbol{\mathcal W}}(a(n)2^j)$ behaves, approximately, like a \textit{multiresolution} sample covariance matrix (cf.\ Abry et al.\ \cite{abry:boniece:didier:wendt:2022}, Remark 3.4). This is so because taking the wavelet transform based on a Haar basis, for example -- for which \eqref{e:NPsi=1} holds --, roughly corresponds to a differencing procedure (e.g., Mallat \cite{mallat:1999}). Hence, whether or not the conclusions of Theorems \ref{t:main_theorem} and \ref{t:main_theorem_discrete} remain true for the case \eqref{e:NPsi=1} remains an open question.\vspace{2mm}

The research on the theory of wavelet random matrices is not only of mathematical interest. Rather, it naturally connects with a number of applications in the modeling of high-dimensional fractal systems, such as those emerging in neuroscience and fMRI imaging, network traffic, climate science, cointegration and high-dimensional time series. For those systems, statistically identifying Hurst multimodality from the available data demands the characterization of the fluctuations of the e.s.d.\ of wavelet random matrices, both in the bulk and around the edges (cf.\ Orejola et al.\ \cite{orejola:didier:wendt:abry:2022}). So far, though, this has only been done for measurements of the form of a high-dimensional signal-plus-noise model, with a fixed number of scaling exponents (Abry et al.\ \cite{abry:boniece:didier:wendt:2022,abry:boniece:didier:wendt:2023:regression}).

\appendix

\section{Wavelet framework}

\subsection{Assumptions on the wavelet multiresolution analysis}\label{s:assumptions_on_the_MRA}

In the main results of the paper, we make use of the following conditions on the underlying wavelet multiresolution analysis (MRA).

\medskip

\noindent {\sc Assumption $(W1)$}: $\psi \in L^2(\bbR)$ is a wavelet function, namely, it satisfies the relations
\begin{equation}\label{e:N_psi}
\int_{\bbR} \psi^2(t)dt = 1 , \quad \int_{\bbR} t^{p}\psi(t)dt = 0, \quad p = 0,1,\hdots, N_{\psi}-1, \quad \int_{\bbR} t^{N_{\psi}}\psi(t)dt \neq 0,
\end{equation}
for some integer (number of vanishing moments) $N_{\psi} \geq 2$.

\medskip

\noindent {\sc Assumption ($W2$)}: the scaling and wavelet functions
\begin{equation}\label{e:supp_psi=compact}
\textnormal{$\phi\in L^1(\bbR)$ and $\psi\in L^1(\bbR)$ are compactly supported }
\end{equation}
and $\widehat{\phi}(0)=1$.

\medskip

\noindent {\sc Assumption $(W3)$}: there exist positive constants $A <\infty$ and $\alpha > 1$ such that
\begin{equation}\label{e:psihat_is_slower_than_a_power_function}
\sup_{x \in \bbR} |\widehat{\psi}(x)| (1 + |x|)^{\alpha} < A.
\end{equation}

%

\medskip

Conditions \eqref{e:N_psi} and \eqref{e:supp_psi=compact} imply that $\widehat{\psi}(x)$ exists, is everywhere infinitely differentiable and its first $N_{\psi}-1$ derivatives are zero at $x = 0$. Condition \eqref{e:psihat_is_slower_than_a_power_function}, in turn, implies that $\psi$ is continuous (see Mallat \cite{mallat:1999}, Theorem 6.1) and, hence, bounded.

Note that assumptions ($W1-W3$) are closely related to the broad wavelet framework for the analysis of $\kappa$-th order ($\kappa \in \bbN \cup \{0\}$) stationary-increment stochastic processes laid out in Moulines et al.\ \cite{moulines:roueff:taqqu:2007:Fractals,moulines:roueff:taqqu:2007:JTSA,moulines:roueff:taqqu:2008} and Roueff and Taqqu~\cite{roueff:taqqu:2009}. The Daubechies scaling and wavelet functions generally satisfy ($W1-W3$) (see Moulines et al.\ \cite{moulines:roueff:taqqu:2008}, p.\ 1927, or Mallat \cite{mallat:1999}, p.\ 253). Usually, the parameter $\alpha$ increases to infinity as $N_{\psi}$ goes to infinity (see Moulines et al.\ \cite{moulines:roueff:taqqu:2008}, p.\ 1927, or Cohen \cite{cohen:2003}, Theorem 2.10.1). Also, under the orthogonality of the underlying wavelet and scaling function basis, ($W1-W3$) imply the so-called Strang-Fix condition (see Mallat \cite{mallat:1999}, Theorem 7.4, and Moulines et al.\ \cite{moulines:roueff:taqqu:2007:JTSA}, p.\ 159, condition (W-4)).

\subsection{Mallat's algorithm and discrete time measurements}\label{s:Mallats_algorithm}

Initially, suppose an infinite sequence of vector-valued measurements
\begin{equation}\label{e:infinite_sample_discrete}
\{Y(k)\}_{k \in \bbZ} \subseteq \bbR^p
\end{equation}
is available. Then, we can apply Mallat's algorithm to extract the so-named \textit{approximation} $(A(2^{j+1},\cdot))$ and \textit{detail} $(\widetilde{D}(2^{j+1},\cdot))$ coefficients at coarser scales $2^{j+1}$ by means of an iterative procedure. In fact, as commonly done in the wavelet literature, we initialize the algorithm with the process
\begin{equation}\label{e:Btilde}
\bbR^p \ni \widetilde{Y}(t) := \sum_{k \in \bbZ} Y(k)\phi(t-k), \quad  t\in \bbR.
\end{equation}
By the orthogonality of the shifted scaling functions $\{\phi(\cdot - k)\}_{k \in \bbZ}$,
\begin{equation}\label{e:a(0,k)}
\bbR^p \ni  A(2^0,k)= \int_\bbR \widetilde{Y}(t)\phi(t-k)dt= Y(k), \quad k \in \bbZ
\end{equation}
(see Stoev et al.\ \cite{stoev:pipiras:taqqu:2002}, proof of Lemma 6.1, or Moulines et al.\ \cite{moulines:roueff:taqqu:2007:JTSA}, p.\ 160; cf.\ Abry and Flandrin \cite{abry:flandrin:1994}, p.\ 33). In other words, the initial sequence, at octave $j = 0$, of approximation coefficients is given by the original sequence of random vectors. To obtain approximation and detail coefficients at coarser scales, we use Mallat's iterative procedure
\begin{equation}\label{e:Mallat}
A(2^{j+1},k) = \sum_{k'\in \mathbb{Z}} u_{k'-2k}A(2^j,k'),\quad \widetilde{D}(2^{j+1},k) =\sum_{k'\in \mathbb{Z}}v_{k'-2k} A(2^j,k'), \quad j \in \mathbb{N} \cup \{0\}, \quad k \in \mathbb{Z},
\end{equation}
where the (scalar) filter sequences $\{ u_k:=2^{-1/2}\int_\bbR \phi(t/2)\phi(t-k)dt \}_{k\in\bbZ}$, $\{v_k:=2^{-1/2}\int_\bbR\psi(t/2)\phi(t-k)dt\}_{k\in\bbZ}$ are called low- and high-pass MRA filters, respectively. Due to the assumed compactness of the supports of $\psi$ and of the associated scaling function $\phi$ (see condition \eqref{e:supp_psi=compact}), only a finite number of filter terms is nonzero, which is convenient for computational purposes (Daubechies \cite{daubechies:1992}). So, assume without loss of generality that $\text{supp}(\phi) = \text{supp}(\psi)=[0,T]$ (cf.\ Moulines et al.\ \cite{moulines:roueff:taqqu:2007:JTSA}, p.\ 160). Then, the wavelet (detail) coefficients $\widetilde{D}(2^j,k)$ of $Y$ can be expressed as \eqref{e:disc2}, where the filter terms are defined by \eqref{e:hj,l}.

If we replace \eqref{e:infinite_sample_discrete} with the realistic assumption that only a finite length series \eqref{e:finite_sample} is available, writing $\widetilde{Y}^{(n)}(t) := \sum_{k=1}^{n}Y(k)\phi(t-k)$, we have $\widetilde{Y}^{(n)}(t) = \widetilde{Y}(t)$ for all $t\in (T,n+1)$ (cf.\ Moulines et al.\ \cite{moulines:roueff:taqqu:2007:JTSA}). Noting that $\widetilde{D}(2^j,k) = \int_\bbR \widetilde Y(t) 2^{-j/2}\psi(2^{-j}t - k)dt$ and $\widetilde{D}^{(n)}(2^j,k) = \int_\bbR \widetilde Y^{(n)}(t)2^{-j/2}\psi(2^{-j}t - k)dt$, it follows that the finite-sample wavelet coefficients $\widetilde{D}^{(n)}(2^j,k)$ of $\widetilde{Y}^{(n)}(t) $ are equal to $\widetilde{D}(2^j,k)$ whenever $\textnormal{supp } \psi(2^{-j} \cdot - k) = (2^jk, 2^j (k+T)) \subseteq (T,n+1)$. In other words,
\begin{equation}\label{e:d^n=d}
\widetilde{D}^{(n)}(2^j,k) = \widetilde{D}(2^j,k),\qquad \forall (j,k)\in\{(j,k): 2^{-j}T\leq k\leq 2^{-j}(n+1)-T\}.
\end{equation}
Equivalently, such subset of finite-sample wavelet coefficients is not affected by the so-named \textit{border effect} (cf.\ Craigmile et al.\ \cite{craigmile:guttorp:Percival:2005}, Percival and Walden \cite{percival:walden:2006}). Moreover, by \eqref{e:d^n=d} the number of such coefficients at octave $j$ is given by $n_j = \lfloor 2^{-j}(n+1-T)-T\rfloor$. Hence, $n_j \sim 2^{-j}n$ for large $n$. Thus, for notational simplicity we suppose
 \begin{equation}\label{e:nj=n/2^j}
 n_{j} = \frac{n}{2^j} \in \bbN
 \end{equation}
 holds exactly and only work with wavelet coefficients unaffected by the border effect.

\section{Auxiliary results}\label{s:auxiliary}

In this section, we state and prove a number of lemmas and propositions used in the proofs of Theorems \ref{t:main_theorem} and \ref{t:main_theorem_discrete}. In several arguments, we make use of the so-called \textit{Weyl's inequalities} for both sums and products of Hermitian matrices. In particular, for $\mathbf{A}, \mathbf{B} \in \mathcal{M}(p,\bbR)$, and for $\ell = 1,...,p$, we have the following \textit{Weyl's inequalities} \begin{equation}\label{e:Weyls_inequalities_sum}
            \sigma_\ell(\mathbf{A}+\mathbf{B}) \leq \sigma_\ell(\mathbf{A})+\sigma_p(\mathbf{B})
        \end{equation}
and
\begin{equation}\label{e:Weyls_inequalities_products}
            \sigma_\ell(\mathbf{A})\cdot \sigma_1(\mathbf{B})\leq \sigma_\ell(\mathbf{A}\mathbf{B}) \leq \sigma_\ell(\mathbf{A})\cdot \sigma_p(\mathbf{B}).
        \end{equation}
Moreover, if $\mathbf{A}$ and $\mathbf{B}$ are Hermitian, then
        \begin{equation}\label{e:Weyls_inequalities_sum_symmetric}
            \lambda_\ell(\mathbf{A})+\lambda_1(\mathbf{B})\leq \lambda_\ell(\mathbf{A}+\mathbf{B}) \leq \lambda_\ell(\mathbf{A})+\lambda_p(\mathbf{B}), \quad \ell = 1,\hdots,p,
        \end{equation}
and
\begin{equation}\label{e:Weyls_inequalities_vershynin}
    \max_{1\leq \ell \leq p}| \lambda_\ell(\mathbf{A})- \lambda_\ell(\mathbf{B})| \leq \| \mathbf{A}-\mathbf{B}\|_{\textnormal{op}}.
\end{equation}
Proofs of the Weyl's inequalities \eqref{e:Weyls_inequalities_sum}--\eqref{e:Weyls_inequalities_vershynin} can be found in Gray \cite{gray:2006} and Vershynin \cite{vershynin:2018}.
In addition, we also make use of the cyclic property of Hermitian matrices, for $\ell = 1,...,p$,
\begin{equation}\label{e:cyclic_property}
    \lambda_\ell(\mathbf{A}\mathbf{B}) = \lambda_\ell(\mathbf{B}\mathbf{A}),
\end{equation}
which is a direct consequence of the fact that $\mathbf{A}\mathbf{B}$ and $\mathbf{B}\mathbf{A}$ have equal characteristic polynomials.

Throughout this and the next sections, further recall that we write $a=a(n)$ or $p=p(n)$ as in \eqref{e:a=a(n)_or_p=p(n)} whenever convenient.

\subsection{Continuous time}\label{s:Continuous time}

In this section, we provide some lemmas and propositions used in the proof of Theorem \ref{t:main_theorem}. First, we recap some notions associated with Toeplitz matrices.
\begin{definition}\label{d:Teoplitz}
A \textit{Toeplitz matrix} $T \in {\mathcal M}(n)$ is a matrix of the form
$$
T =
    \begin{bmatrix}
        \tau(0) & \tau(-1) &\tau(-2) & \hdots & \tau(-(n-1))\\
        \tau(1) & \tau(0) & \tau(-1) \\
        \tau(2) & \tau(1) & \tau(0) & \\
        \vdots & & & \ddots \\
        \tau(n-1) & & &  \hdots & \tau(0)
    \end{bmatrix}.
$$
We say a sequence of Toeplitz matrices $T_n = \big\{\tau(\ell_1-\ell_2)\big\}_{\ell_1,\ell_2=1,\hdots,n}$, $n \in \bbN$, is \textit{generated by an integrable
function} $f: [-\pi,\pi]\mapsto (0,\infty)$ when
\begin{equation}\label{e:toeplitz}
\tau(\ell) = \frac{1}{2\pi}\int_{-\pi}^{\pi}f(\lambda) e^{-i\ell\lambda}d\lambda, \quad \ell \in \bbZ.
\end{equation}
\end{definition}
Moreover, if $\{\tau(\ell)\}_{\ell \in \bbZ}$ is absolutely summable, then the sequence of Toeplitz matrices $\{T_n\}_{n \in \bbN}$ is said to be in the \textit{Wiener class}.

Let $\{\Sigma_n\}_{n \in \bbN}$ be a sequence of symmetric Toeplitz matrices in the Wiener class, generated by a function $f$. By Gray \cite{gray:2006}, Lemma 4.1, for every $n\in \mathbb{N}$ and $\ell = 1,...,n$, the eigenvalues of $\Sigma_n$ satisfy
\begin{equation}\label{e:gray_toeplitz_bound}
    m_{f} \leq \lambda_{\ell}(\Sigma_n) \leq M_{f},
\end{equation}
where
\begin{equation}\label{e:eigenvalue_bounds_toeplitz}
m_{f} := \textnormal{ess} \inf_{x\in [-\pi,\pi]} f(x)\quad \textnormal{and} \quad M_{f} := \textnormal{ess} \sup_{x\in [-\pi,\pi]} f(x).
\end{equation}\\

We are now in a position to state and prove Lemma \ref{l:fixed_scale_log_limit}. In this lemma, which is used in the proof of Theorem \ref{t:main_theorem}, we show that $\mathbf{W}_X(2^j)$ is a full rank matrix a.s. Consequently, $\log \lambda_\ell (\mathbf{W}_X(2^j))$ is well defined a.s.\ for $\ell = 1,...,p(n)$. In addition, we establish the limiting behavior of
$$
\log \lambda_\ell (\mathbf{W}_X(2^j))/\log a(n).
$$
To do so, we construct almost sure bounds on the eigenvalues of the fixed-scale wavelet random matrix $\mathbf{W}_X(2^j)$ as in \eqref{e:waveletmatrix}.

\begin{lemma}\label{l:fixed_scale_log_limit}
Suppose assumptions ($A1-A5$) and ($W1-W3$) hold. Let $\mathbf{W}_X(2^j)$ be the auxiliary wavelet random matrix as in \eqref{e:waveletmatrix}. Then, for $\ell = 1,...,p(n)$,
\begin{itemize}
\item [$(i)$]  $\log \lambda_\ell (\mathbf{W}_X(2^j))$ is well defined a.s.;
\item [$(ii)$] and
\begin{equation}
\label{e:fixed_scale_log_to_zero}
    \frac{\log \lambda_{\ell}(\mathbf{W}_X(2^j))}{\log a(n) }\stackrel{\bbP}\to 0, \quad n \rightarrow \infty.
\end{equation}
\end{itemize}
\end{lemma}
\begin{proof}
First, we show $(i)$. From expression \eqref{e:waveletmatrix}, $\mathbf{W}_X(2^j) \in {\mathcal S}_{\geq 0}(n,\bbR)$ a.s. Hence, $\lambda_1(\mathbf{W}_X(2^j))\geq 0$ a.s. Now recall that, by assumption ($A4$),
\begin{equation}\label{e:p=p(n)<n_a,j_WX}
p = p(n) < n_{a,j}.
\end{equation}
Since $\textnormal{rank}(\mathbf{W}_X(2^j)) = \textnormal{rank}(D_X(2^j))$, then in light of \eqref{e:p=p(n)<n_a,j_WX} it suffices to show that $\textnormal{rank}(D_X(2^j)) = p(n)$ a.s. Indeed, by expression \eqref{e:mixed_gaussian_representation}, $D_X(2^j)^* \stackrel{d}= \Big( \Gamma_{{\mathcal H}_1} {\mathbf Z}_1, \hdots,\Gamma_{{\mathcal H}_p} {\mathbf Z}_p \Big)$,
where $\bbR^{n_{a,j}} \supseteq {\mathbf Z}_1,\hdots,{\mathbf Z}_p \stackrel{\textnormal{i.i.d.}}\sim {\mathcal N}({\mathbf 0},I)$ and $\Gamma_{{\mathcal H}_\ell}$, $\ell = 1,\hdots,p$, are as in \eqref{e:Gamma-mathcal_H}. Consequently, for any Hurst (matrix) exponent $\mathbb{H}_n$ (see \eqref{e:diagonal_H_=_instance}), again by \eqref{e:p=p(n)<n_a,j_WX} and by Lemma \ref{l:gaussian_rv_are_LI},
$$
\bbP\big(\textnormal{dim}(\textnormal{span}\{\Gamma_{\mathcal{H}_1}\mathbf{Z}_1,\hdots,\Gamma_{\mathcal{H}_{p}}\mathbf{Z}_{p} \}) = p \big| \mathds{H}_n = \mathbb{H}_n\big) = 1.
$$
Since this holds for every $\mathbb{H}_n$, it follows that the rows of $D_X(2^j)$ are linearly independent a.s. Therefore, $\textnormal{rank}(D_X(2^j)) = p(n)$ a.s. This establishes $(i)$.

We now turn to $(ii)$. Recall the simple fact that, for any matrix $\mathbf{M}\in {\mathcal M}(p,n,\bbR)$ with $p=p(n)\leq n$,
\begin{equation}\label{e:matrix_transpose_eigen_trick}
    \lambda_\ell(\mathbf{M}\mathbf{M}^*) =\lambda_{n+\ell-p}(\mathbf{M}^*\mathbf{M}), \quad \ell = 1,\hdots, p.
\end{equation}
Therefore, by expression \eqref{e:waveletmatrix},
$$
\lambda_{\ell} \big( \mathbf{W}_X(2^j)\big) = \lambda_{\ell} \Big(\frac{1}{n_{a,j}} D_X(2^j)D_X(2^j)^*\Big) = \lambda_{n_{a,j}  - p + \ell} \Big(\frac{1}{n_{a,j}} D_X(2^j)^*D_X(2^j)\Big)
$$
\begin{equation}\label{e:eigenvalues_wavelete_decorelated}
 \stackrel{d}= \lambda_{n_{a,j}  - p + \ell} \Big(\frac{1}{n_{a,j}} \sum_{i=1}^r \Gamma_{\breve{H}_i}\mathcal{Z}_i\mathcal{Z}_i^*\Gamma_{\breve{H}_i}\Big), \quad \ell = 1,\hdots, p,
\end{equation}
where the equality in distribution follows from Lemma \ref{l:decorellation_sum_p}, $(iv)$, and the random matrices $\mathcal{Z}_i$, $i = 1,\hdots,r$, are defined in \eqref{e:decorrelation_sum_p_claim}.
  Consider the random variables $p_i = p_i(n)$, $i = 1,\hdots,r$, as in \eqref{e:p-tilde_i(n)}. By Kolmogorov's strong law of large numbers, $\lim_{n \rightarrow \infty}p_i(n)/p(n) = \bbP(\mathcal{H} =\breve{H}_i)$ a.s., for $i = 1,\hdots,r$. Thus, from assumption ($A4$) and Lemma \ref{l:D(2^j,k)_is_stationary}, ($ii$), an application of Proposition \ref{p:k_sum_of_deformed_wisharts} to the right-hand side of \eqref{e:eigenvalues_wavelete_decorelated} establishes the existence of strictly positive constants $m$ and $M$ such that, for $\ell = 1,\hdots, p$,
\begin{equation}\label{e:eigenvalueboundsSigma}
 m+o_\bbP(1) \leq \lambda_{n_{a,j}  - p + \ell} \Big(\frac{1}{n_{a,j}} \sum_{i=1}^r \Gamma_{\breve{H}_i}\mathcal{Z}_i\mathcal{Z}_i^*\Gamma_{\breve{H}_i}\Big) \leq M+o_\bbP(1).
\end{equation}
Note that the eigenvalues in \eqref{e:eigenvalueboundsSigma} are strictly positive a.s. Therefore, by Lemma \ref{l:sandwich_bounds_to_log_convergence}, again for each $\ell = 1 , \hdots, p$,
\begin{equation}\label{e:decorellated_log_convergence}
\frac{\log \lambda_{n_{a,j}  - p + \ell} \Big(\frac{1}{n_{a,j}} \sum_{i=1}^r \Gamma_{\breve{H}_i}\mathcal{Z}_i\mathcal{Z}_i^*\Gamma_{\breve{H}_i}\Big) }{\log a(n)} \stackrel{\bbP}\to 0.
\end{equation}
As a consequence of expressions \eqref{e:eigenvalues_wavelete_decorelated} and \eqref{e:decorellated_log_convergence}, the limit \eqref{e:fixed_scale_log_to_zero} follows. This completes the proof of $(ii)$. $\Box$\\
\end{proof}

Proposition \ref{p:k_sum_of_deformed_wisharts} is used in the proof of Lemma \ref{l:fixed_scale_log_limit}. However, Proposition \ref{p:Sum_of_deformed_wisharts} is needed in establishing Proposition \ref{p:k_sum_of_deformed_wisharts}. We state and prove these two propositions next.

\begin{proposition}\label{p:Sum_of_deformed_wisharts} Let $\{p_i \}_{n \in \bbN} = \{p_i(n)\}_{n \in \bbN} \subseteq \bbN \cup \{0\}$, $i=1,2$, be two sequences of random variables such that
\begin{equation}\label{e:moderately_high_condition}
 p_1+p_2=p \leq n \quad and  \quad p_1,\hspace{0.5mm} p_2 \stackrel{\textnormal{a.s.}}= o(\sqrt{n})
 \end{equation}
(\textbf{n.b.:} in \eqref{e:moderately_high_condition}, $p = p(n)$ is deterministic).
      For $i=1,2$, if $p_i \geq 1$ let $\mathcal{Z}_i \in \mathcal{M}(n,p_i,\bbR)$ be independent random matrices with independent, standard normal entries, and let $\mathcal{Z}_i = {\mathbf 0}$ if $p_i = 0$. Also, for $i=1,2$, let $\Sigma_i \in {\mathcal M}(n)$ be symmetric Toeplitz matrices generated by functions $f_i$ such that for $M_{f_i}$ and $m_{f_i}$ defined as in \eqref{e:eigenvalue_bounds_toeplitz}, $M_{f_i}\geq m_{f_i} > 0$. Let $\Gamma_i := \Sigma_\ell^{1/2},\ i = 1,2$.  Then, almost surely,
$$
    \lambda_{n-p+1} \Big(
\Gamma_1\frac{\mathcal{Z}_1 \mathcal{Z}_1^*}{n}\Gamma_1 + \Gamma_2\frac{\mathcal{Z}_2 \mathcal{Z}_2^*}{n}\Gamma_2 \Big)
$$
\begin{equation}\label{e:sum_wishart_lower_bound}
\geq \max\Big\{ \min\Big\{\frac{m_{f_2}m_{f_1}}{M_{f_2}},m_{f_2}\Big\} , \min\Big\{\frac{m_{f_2}m_{f_1}}{M_{f_1}},m_{f_1}\Big\} \Big\} + o_\bbP(1).
\end{equation}
\end{proposition}
\begin{proof}
Observe that, by Weyl's inequalities \eqref{e:Weyls_inequalities_products}, we obtain the two bounds
  \begin{equation}\label{e:Weyl's_on_the_sum_1}
    \lambda_{n-p+1} \Big(
\Gamma_1\frac{\mathcal{Z}_1 \mathcal{Z}_1^*}{n} \Gamma_1+ \Gamma_2\frac{\mathcal{Z}_2 \mathcal{Z}_2^*}{n} \Gamma_2\Big) \geq \lambda_{1}(\Sigma_2) \cdot \lambda_{n-p+1} \Big(
\Gamma_2^{-1}\Gamma_1\frac{\mathcal{Z}_1 \mathcal{Z}_1^*}{n} \Gamma_1\Gamma_2^{-1}+\frac{\mathcal{Z}_2 \mathcal{Z}_2^*}{n}  \Big)
\end{equation}
and
\begin{equation}\label{e:Weyl's_on_the_sum}
     \lambda_{n-p+1} \Big(
\Gamma_1\frac{\mathcal{Z}_1 \mathcal{Z}_1^*}{n} \Gamma_1+ \Gamma_2\frac{\mathcal{Z}_2 \mathcal{Z}_2^*}{n} \Gamma_2\Big) \geq \lambda_{1}(\Sigma_1) \cdot \lambda_{n-p+1} \Big(
\frac{\mathcal{Z}_1 \mathcal{Z}_1^*}{n} +\Gamma_1^{-1}\Gamma_2\frac{\mathcal{Z}_2 \mathcal{Z}_2^*}{n} \Gamma_2\Gamma_1^{-1} \Big).
    \end{equation}
Note that the right-hand side of inequalities \eqref{e:Weyl's_on_the_sum_1} and \eqref{e:Weyl's_on_the_sum} are of identical form.  Moreover, the method by which we bound \eqref{e:Weyl's_on_the_sum} will analogously apply to \eqref{e:Weyl's_on_the_sum_1}. So, to establish \eqref{e:sum_wishart_lower_bound} we simply take the max of the two lower bounds. Therefore, without loss of generality we focus on \eqref{e:Weyl's_on_the_sum}. That is, we want to show that
\begin{equation}\label{e:Wishart_and_dirty_wishart}
    \lambda_{n-p+1} \Big(
\frac{\mathcal{Z}_1 \mathcal{Z}_1^*}{n} +\Gamma_1^{-1}\Gamma_2\frac{\mathcal{Z}_2\mathcal{Z}_2^*}{n} \Gamma_2\Gamma_1^{-1} \Big)\geq \min\Big\{\frac{m_{f_2}}{M_{f_1}},1\Big\} + o_\bbP(1) \quad \textnormal{a.s.}.
     \end{equation}
    To do so, consider the following decomposition of the two matrices in the sum,
\begin{equation}\label{e:diagonalization_substitution}
    \frac{\mathcal{Z}_1 \mathcal{Z}_1^*}{n} + \Gamma_1^{-1}\Gamma_2\frac{\mathcal{Z}_2 \mathcal{Z}_2^*}{n}\Gamma_2\Gamma_1^{-1}  = \mathbf{U}_1 \mathbf{D}_1 \mathbf{U}_1^* + \mathbf{U}_2 \mathbf{D}_2 \mathbf{U}_2^* \quad \textnormal{a.s.}
    \end{equation}
In \eqref{e:diagonalization_substitution}, $\mathbf{U}_1$, $\mathbf{U}_2 \in O(n)$ and $\mathbf{D}_1$, $\mathbf{D}_2$ are diagonal matrices satisfying
$$\lambda_\ell(\mathbf{D}_1) = \lambda_\ell\Big(\frac{\mathcal{Z}_1 \mathcal{Z}_1^*}{n} \Big) \quad \textnormal{and} \quad \lambda_\ell(\mathbf{D}_2) = \lambda_\ell\Big(\Gamma_1^{-1}\Gamma_2\frac{\mathcal{Z}_2 \mathcal{Z}_2}{n}\Gamma_2\Gamma_1^{-1} \Big) $$
for $\ell = 1,...,n$. Thus,
 $$
    \lambda_{n-p+1} \Big(
\frac{\mathcal{Z}_1 \mathcal{Z}_1^*}{n} + \Gamma_1^{-1}\Gamma_2\frac{\mathcal{Z}_2 \mathcal{Z}_2^*}{n}\Gamma_2\Gamma_1^{-1} \Big)=  \lambda_{n-p+1} ( \mathbf{D}_1  + \mathbf{U}_1^*\mathbf{U}_2 \mathbf{D}_2 \mathbf{U}_2^* \mathbf{U}_1 ) .
$$
Note that $\mathbf{U}_1$ and $\mathbf{U}_2$ are independent random matrices. Moreover, by Lemma \ref{l:haar_basis} we can assume that $\mathbf{U}_1 \sim \textnormal{Haar}(O(n))$. Hence, by Lemma \ref{l:conditional_haar}, $\mathbf{U}_1^*\mathbf{U}_2 = \mathbf{U} \sim \textnormal{Haar}(O(n))$. Thus,
\begin{equation}\label{e:decomposition_equality_indistribution}
\lambda_{n-p+1} ( \mathbf{D}_1  + \mathbf{U}_1^*\mathbf{U}_2 \mathbf{D}_2 \mathbf{U}_2^* \mathbf{U}_1 ) \stackrel{d}=\lambda_{n-p+1} ( \mathbf{D}_1  + \mathbf{U} \mathbf{D}_2 \mathbf{U}^* ).
\end{equation}
For now, suppose that, almost surely, the non-zero eigenvalues of $\mathbf{D}_1$ and $\mathbf{D}_2$ satisfy the boundedness relations
\begin{equation}\label{e:boundedness_relations_D_1}
1 + o_\mathbb{P}(1) \leq \lambda_{n-p_1+1}(\mathbf{D}_1) \leq \lambda_{n}(\mathbf{D}_1) \leq 1 + o_\mathbb{P}(1)
\end{equation}
and
\begin{equation}\label{e:boundedness_relations_D_2}
\frac{m_{f_2}}{M_{f_1}} + o_\mathbb{P}(1)\leq \lambda_{n-p_2+1}(\mathbf{D}_2)\leq \lambda_{n}(\mathbf{D}_2) \leq \frac{M_{f_2}}{m_{f_1}} + o_\mathbb{P}(1).
\end{equation}
Namely, $\mathbf{D}_1$ and $\mathbf{D}_2$ satisfy the boundedness relations \eqref{e:eigenvalues_lower_limit} and \eqref{e:eigenvalues_upper_limit}. Then, by Proposition \ref{p:D_1_D_2_sum_bound},
$$
\lambda_{n-p+1}(\mathbf{D}_1 + \mathbf{U}\mathbf{D}_2\mathbf{U}^*)  \geq \min\Big\{\frac{m_{f_2}}{M_{f_1}},1\Big\} + o_\mathbb{P}(1) \quad\textnormal{a.s.}
$$
Moreover, under conditions \eqref{e:moderately_high_condition}, by Lemma \ref{l:bounds_equal_in_distribution} and by the equality in distribution \eqref{e:decomposition_equality_indistribution}, we obtain
$$
  \lambda_{n-p+1} \Big(
\frac{\mathcal{Z}_1 \mathcal{Z}_1^*}{n} + \Gamma_1^{-1}\Gamma_2\frac{\mathcal{Z}_2 \mathcal{Z}_2^*}{n}\Gamma_1^{-1} \Gamma_2\Big)  \geq \min\Big\{\frac{m_{f_2}}{M_{f_1}},1\Big\} + o_\bbP(1).
$$
This establishes \eqref{e:sum_wishart_lower_bound}.

So, we now need to establish relations \eqref{e:boundedness_relations_D_1} and \eqref{e:boundedness_relations_D_2}. In fact, as a consequence of condition \eqref{e:moderately_high_condition}, by Proposition \ref{l:max_eigenvalue_of_white_wishart},
\begin{equation}\label{e:bai_yin}
\lambda_1\Big(\frac{\mathcal{Z}_1^*\mathcal{Z}_1}{n}\Big), \lambda_{p_1}\Big(\frac{\mathcal{Z}_1^*\mathcal{Z}_1}{n}\Big),
\lambda_1\Big(\frac{\mathcal{Z}_2^*\mathcal{Z}_2}{n}\Big), \lambda_{p_2}\Big(\frac{\mathcal{Z}_2^*\mathcal{Z}_2}{n}\Big) \stackrel{\bbP}\to 1.
\end{equation}
From relation \eqref{e:matrix_transpose_eigen_trick}, $\lambda_\ell(\mathcal{Z}_1^*\mathcal{Z}_1/n) = \lambda_{n-p_1+\ell}(\mathbf{D}_1)$ for $\ell=1,...,p_1$. It follows from expression \eqref{e:bai_yin} that $\mathbf{D}_1$ satisfies conditions \eqref{e:boundedness_relations_D_1}, as claimed.

Similarly, from relation \eqref{e:matrix_transpose_eigen_trick} and by Weyl's inequalites \eqref{e:Weyls_inequalities_products}, for $\ell =1,...,p_2$,
$$
\lambda_1(\Sigma_2)\cdot \lambda_1(\Sigma_1^{-1})\cdot\lambda_{1}\Big(\frac{\mathcal{Z}_2^* \mathcal{Z}_2}{n}\Big) = \lambda_1(\Sigma_2)\cdot \lambda_1(\Sigma_1^{-1})\cdot \lambda_{n-p_2+1}\Big(\frac{\mathcal{Z}_2\mathcal{Z}_2^*}{n}\Big) \leq \lambda_{n-p_2+\ell}\Big(\Gamma_1^{-1}\Gamma_2\frac{\mathcal{Z}_2 \mathcal{Z}_2^*}{n}\Gamma_2^*\Gamma_1^{*-1}\Big)
$$
$$
= \lambda_{n-p_2+\ell}(\mathbf{D}_2) \leq \lambda_n(\Sigma_2)\cdot \lambda_n(\Sigma_1^{-1})\cdot \lambda_{n-p_2+\ell}\Big(\frac{\mathcal{Z}_2\mathcal{Z}_2^*}{n}\Big) = \lambda_n(\Sigma_2)\cdot\lambda_{n}(\Sigma_1^{-1})\cdot\lambda_{p_2}\Big(\frac{\mathcal{Z}_2^* \mathcal{Z}_2}{n}\Big).
$$
Then, by the facts that $\lambda_1(\Sigma_1^{-1}) = \lambda_{n}(\Sigma_1)^{-1} $, $\lambda_n(\Sigma_1^{-1}) = \lambda_{1}(\Sigma_1)^{-1} $, and by expression \eqref{e:gray_toeplitz_bound}, we arrive at the bounds
\begin{equation}\label{e:d_2_bounds}
\frac{m_{f_2}}{M_{f_1}}\cdot \lambda_{1}\Big(\frac{\mathcal{Z}_2^* \mathcal{Z}_2}{n}\Big)\leq \lambda_{n-p_2+\ell}(\mathbf{D}_2) \leq \frac{M_{f_2}}{m_{f_1}}\cdot \lambda_{p_2}\Big(\frac{\mathcal{Z}_2^* \mathcal{Z}_2}{n}\Big),
\end{equation}
for $\ell =1,...,p_2$. Hence, from expression \eqref{e:bai_yin} and inequalities \eqref{e:d_2_bounds}, it follows that $\mathbf{D}_2$ also satisfies \eqref{e:boundedness_relations_D_2}, as claimed. $\Box$\\
\end{proof}

The following proposition is used in the proof of Lemma \ref{l:fixed_scale_log_limit}. It establishes upper and lower bounds on the eigenvalues appearing on the right-hand side of \eqref{e:eigenvalues_wavelete_decorelated}. Note that the lower bound \eqref{e:sum_bound_lower} is not sharp. Instead, it is one possible bound we can obtain based on the max bound in \eqref{e:sum_wishart_lower_bound}.

\begin{proposition}\label{p:k_sum_of_deformed_wisharts}
    Fix $r \in \Naturals$. For $i = 1,...,r$, let $p_i = p_i(n) \in \bbN \cup \{0\}$, $i=1,\hdots,r$, be a nondecreasing sequence of random variables such that
     \begin{equation}\label{e:moderatly_high_condition_sum}
      \sum_{i=1}^r p_i=p= o(\sqrt{n}) \quad \textnormal{and} \quad \frac{p_i}{p} \stackrel{\textnormal{a.s.}}\to c_i \in (0,1)
     \end{equation}
(\textbf{n.b.}: in \eqref{e:moderatly_high_condition_sum}, $p=p(n)$ is deterministic). Moreover, for $i = 1,\hdots,r$, let $\mathcal{Z}_i$ be independent random matrices as in Proposition \ref{p:Sum_of_deformed_wisharts}. In addition, for $i = 1,\hdots,r$, let $\Sigma_i \in {\mathcal M}(n)$ be symmetric Toeplitz matrices generated by a strictly positive function $f_i$ such that for $M_{f_i}$ and $m_{f_i}$ defined as in \eqref{e:eigenvalue_bounds_toeplitz}, $M_{f_i}\geq m_{f_i} > 0$. Also, define $\Gamma_i = \Sigma_i^{1/2}$.  Then,
    \begin{equation}\label{e:sum_bound_upper}
    \lambda_{n}\Big( \sum_{i=1}^r \Gamma_{i}\frac{\mathcal{Z}_i\mathcal{Z}_i^*}{n}\Gamma_{i}\Big)\leq \sum_{i=1}^r M_{f_i} + o_{\bbP}(1),
    \end{equation}
    whereas
    \begin{equation}\label{e:sum_bound_lower}
    \lambda_{n-p+1}\Big( \sum_{i=1}^r \Gamma_{i}\frac{\mathcal{Z}_i\mathcal{Z}_i^*}{n}\Gamma_{i}^*\Big)
\geq m_{f_r} \min\Big\{1,\frac{m_{f_{r-1}}}{M_{f_r}},\frac{m_{f_{r-1}}m_{f_{r-2}}}{M_{f_r}M_{f_{r-1}}},\hdots, \frac{\prod_{i=1}^{r-1}m_{f_i}}{\prod_{i=2}^{r}M_{f_i}}  \Big\} + o_\bbP(1).
    \end{equation}
\end{proposition}
\begin{proof}
We first establish \eqref{e:sum_bound_upper}. By Weyl's inequalities \eqref{e:Weyls_inequalities_products}-\eqref{e:Weyls_inequalities_sum_symmetric} as well as the bounds \eqref{e:gray_toeplitz_bound},
$$
\lambda_{n}\Big( \sum_{i=1}^r \Gamma_{i}\frac{\mathcal{Z}_i\mathcal{Z}_i^*}{n}\Gamma_{i}\Big)\leq  \sum_{i=1}^r \lambda_{n}\Big(\Gamma_{i}\frac{\mathcal{Z}_i\mathcal{Z}_i^*}{n}\Gamma_{i}\Big) \leq  \sum_{i=1}^r \lambda_{n}(\Sigma_i)\cdot \lambda_{n}\Big(\frac{\mathcal{Z}_i\mathcal{Z}_i^*}{n}\Big)\leq \sum_{i=1}^r M_{f_i}\cdot \lambda_{n}\Big(\frac{\mathcal{Z}_i\mathcal{Z}_i^*}{n}\Big).
$$
Then, under condition \eqref{e:moderatly_high_condition_sum}, by Proposition \ref{l:max_eigenvalue_of_white_wishart} we obtain \eqref{e:sum_bound_upper}.

In turn, to establish \eqref{e:sum_bound_lower}, we proceed in an inductive fashion. We only explicitly derive a bound for the case $r=3$, since the general case of $r \in \bbN$ can be established by a simple adaptation of the argument.

Let $\mathbf{U}_2\mathbf{D}_2\mathbf{U}_2^* =\Gamma_3^{-1} \Big( \Gamma_2\frac{\mathcal{Z}_2 \mathcal{Z}_2^*}{n} \Gamma_2 + \Gamma_1\frac{\mathcal{Z}_1 \mathcal{Z}_1^*}{n} \Gamma_1\Big)\Gamma_3^{-1}$, where $\mathbf{U}_2\in O(n)$ and $\lambda_{\ell}(\mathbf{D}_2) = \lambda_{\ell}\Big(\Gamma_3^{-1} \Big( \Gamma_2\frac{\mathcal{Z}_2 \mathcal{Z}_2^*}{n} \Gamma_2 + \Gamma_1\frac{\mathcal{Z}_1 \mathcal{Z}_1^*}{n} \Gamma_1\Big)\Gamma_3^{-1}\Big)$ for $\ell  =1,...,n$. Then,
$$
   \lambda_{n-(p_1+p_2)+1} (\mathbf{D}_2)  \geq\lambda_1(\Sigma_3^{-1})\cdot \lambda_{n-(p_1+p_2)+1} \Big(\Gamma_2\frac{\mathcal{Z}_2 \mathcal{Z}_2^*}{n} \Gamma_2 + \Gamma_1\frac{\mathcal{Z}_1 \mathcal{Z}_1^*}{n} \Gamma_1\Big)
$$
\begin{equation}\label{e:conjugated_bound_lower}
  \geq \frac{m_{f_2}}{M_{f_3}}\min\Big\{\frac{m_{f_1}}{M_{f_2}},1\Big\} + o_\bbP(1)\quad \textnormal{a.s.}
\end{equation}
where the first inequality is a consequence of Weyl's inequalities \eqref{e:Weyls_inequalities_products}, and the second follows from Proposition \ref{p:Sum_of_deformed_wisharts}, and expression \eqref{e:gray_toeplitz_bound}.
Moreover,
\begin{equation}\label{e:conjugated_bound_upper}
   \lambda_n(\mathbf{D}_2) \leq \frac{M_{f_1}}{m_{f_3}}\lambda_n \Big(\frac{\mathcal{Z}_1 \mathcal{Z}_1^*}{n} \Big) +  \frac{M_{f_2}}{m_{f_3}} \lambda_n \Big(\frac{\mathcal{Z}_2 \mathcal{Z}_2^*}{n} \Big) \leq \frac{M_{f_1}}{m_{f_3}} +  \frac{M_{f_2}}{m_{f_3}} + o_\bbP(1) \quad \textnormal{a.s.},
\end{equation}
where the first inequality follows from  Weyl's inequalities \eqref{e:Weyls_inequalities_products}-\eqref{e:Weyls_inequalities_sum_symmetric} and bounds \eqref{e:gray_toeplitz_bound}, and the second inequaility is a consequence of condition \eqref{e:moderatly_high_condition_sum} and Proposition \ref{l:max_eigenvalue_of_white_wishart}. Hence, $\mathbf{D}_2$ satisfies the bounds \eqref{e:eigenvalues_lower_limit} and \eqref{e:eigenvalues_upper_limit}. In addition, by Lemma \ref{l:haar_basis}, we can recast $\frac{\mathcal{Z}_3\mathcal{Z}_3^*}{n} \stackrel{d}= \mathbf{U}_1 \mathbf{D}_1\mathbf{U}_1^*$ where $\mathbf{U}_1\sim \textnormal{Haar}(O(n))$ and $\lambda_\ell(\mathbf{D}_1) = \lambda_\ell\big(\frac{\mathcal{Z}_3\mathcal{Z}_3^*}{n}\big)$ for $\ell = 1,...,n$. Then, as in the proof of Proposition \ref{p:Sum_of_deformed_wisharts}, by Proposition \ref{p:D_1_D_2_sum_bound} and Lemma \ref{l:bounds_equal_in_distribution},  it follows that
$$
\lambda_{n-(p_1+p_2+p_3)+1} \Big( \frac{\mathcal{Z}_3\mathcal{Z}_3^*}{n} + \Gamma_3^{-1} \Big(  \Gamma_2\frac{\mathcal{Z}_2 \mathcal{Z}_2^*}{n} \Gamma_2 + \Gamma_1\frac{\mathcal{Z}_1 \mathcal{Z}_1^*}{n} \Gamma_1\Big)\Gamma_3^{-1}\Big)
$$
$$
\geq \min\Big\{1, \frac{m_{f_2}}{M_{f_3}}\min\Big\{\frac{m_{f_1}}{M_{f_2}},1\Big\}\Big\}+ o_\bbP(1) = \min\Big\{1, \frac{m_{f_2}m_{f_1}}{M_{f_3}M_{f_2}},  \frac{m_{f_2}}{M_{f_3}} \Big\}+ o_\bbP(1)\quad \textnormal{a.s.}
$$
Therefore,
$$
\lambda_{n-(p_1+p_2+p_3)+1} \Big( \Gamma_3\frac{\mathcal{Z}_3\mathcal{Z}_3^*}{n}\Gamma_3 + \Gamma_2\frac{\mathcal{Z}_2 \mathcal{Z}_2^*}{n} \Gamma_2 + \Gamma_1\frac{\mathcal{Z}_1 \mathcal{Z}_1^*}{n} \Gamma_1\Big)
$$
$$
\geq m_{f_3} \min\Big\{1, \frac{m_{f_2}m_{f_1}}{M_{f_3}M_{f_2}},  \frac{m_{f_2}}{M_{f_3}} \Big\}+ o_\bbP(1).
$$
This shows \eqref{e:sum_bound_lower}. $\Box$\\
\end{proof}

\subsection{Discrete time}\label{s:discrete time}


The following lemma is the discrete-time analogue of Lemma \ref{l:fixed_scale_log_limit}. The lemma is used in the proof of Theorem \ref{t:main_theorem_discrete}. In the lemma, we make use of the continuous-time analogue of the auxiliary wavelet random matrix \eqref{e:Btilde-nu}. So, assuming continuous-time measurements, let
$$
{\mathcal S}_{\geq 0}(p(n),\bbR) \ni \widehat{\mathbf{B}}_a(2^j) :=a(n)^{-(\mathds{H}_n+(1/2)I) } \frac{1}{n_{a,j}}\sum^{n_{a,j}}_{k=1}D_X(a(n)2^j,k)D_X(a(n)2^j,k)^*
 a(n)^{-(\mathds{H}_n+(1/2)I)}
$$
\begin{equation}\label{e:B^(j)}
\stackrel{d}=  \frac{1}{n_{a,j}}\sum^{n_{a,j}}_{k=1}D_X(2^j,k)D_X(2^j,k)^*  = {\mathbf W}_X(2^j),
\end{equation}
where the equality in distribution follows from Lemma \ref{l:W(a(n)2^j)_D(a(n)2^j,k)_selfsimilar}, $(ii)$ (\textbf{n.b.}: in discrete time, there is no exact self-similarity-like relation such as \eqref{e:B^(j)}).

\begin{lemma}\label{l:fixed_scale_log_limit_discrete}
Suppose assumptions ($A1-A5$) and ($W1-W3$) hold. Let $\widetilde{\mathbf{B}}_a(2^j)$ be a fixed-scale wavelet random matrix for discrete-time
measurements as in \eqref{e:Btilde-nu}. Then, for $\ell = 1,...,p(n)$,
\begin{itemize}
\item [$(i)$]  $\log \lambda_\ell (\widetilde{\mathbf{B}}_a(2^j))$ is well defined a.s.;
\item [$(ii)$] and
\begin{equation}
\label{e:fixed_scale_log_to_zero_discrete}
    \frac{\log \lambda_{\ell}(\widetilde{\mathbf{B}}_a(2^j))}{\log a(n) }\stackrel{\bbP}\to 0, \quad n \rightarrow \infty.
\end{equation}
\end{itemize}
\end{lemma}
\begin{proof}
The proof of $(i)$ is analogous to the proof of Lemma \ref{l:fixed_scale_log_limit}, $(i)$. In particular, from expression \eqref{e:Btilde-nu} and under condition \eqref{e:p=p(n)<n_a,j_WX}, it suffices to show that $\textnormal{rank}(a(n)^{-\mathds{H}-(1/2)I}\widetilde{D}_X(a(n)2^j)) = p(n)$ a.s. Moreover, since $a(n)^{-\mathds{H}-(1/2)I}$ has full rank, we only need to show that $\textnormal{rank}(\widetilde{D}_X(a(n)2^j)) = p(n)$ a.s.  Indeed, by the independence of the $p$ rH-fBms that make up $\{X(t)\}_{t \in {\mathcal T}}$ (see assumption ($A2$) and expression \eqref{e:Y=PX}), and by Lemma \ref{l:covariance_of_univariate_discrete_wavelet_coeff}, conditionally on $\mathds{H}_n = \mathbb{H}_n$, the rows of $\widetilde{D}_X(a(n)2^j)$ are equal in distribution to independent Gaussian random vectors. That is, for $i=1,\hdots,p$, there exist a full rank matrix $\widetilde{\Gamma}_{H_i}$ and an independent standard normal vector $\mathbf{Z}_\ell\in\bbR^{n_{a,j}}$ such that $\widetilde{\Gamma}_{H_i}\mathbf{Z}_\ell\in\bbR^{n_{a,j}}$ is equal, in law, to the $i$--th row of $\widetilde{D}_X(a(n)2^j)$. Therefore, by Lemma \ref{l:gaussian_rv_are_LI},
$$
\bbP\big(\textnormal{dim}(\textnormal{span}\{\widetilde{\Gamma}_{H_1}\mathbf{Z}_1,\hdots,\widetilde{\Gamma}_{H_{p}}\mathbf{Z}_{p} \}=p \big| \mathds{H}_n = \mathbb{H}_n\big) = 1.
$$
 Since this holds for every $\mathbb{H}_n$, it follows that the rows of $D_X(2^j)$ are linearly independent a.s. Therefore, $\textnormal{rank}(D_X(2^j)) = p(n)$ a.s. This establishes $(i)$.

We now show $(ii)$. Observe that, by an application of Weyl's inequalities \eqref{e:Weyls_inequalities_sum_symmetric},
$$
 \lambda_{1} \big(\widetilde{\mathbf{B}}_a(2^j) - \widehat{\mathbf{B}}_a(2^j)\big)  + \lambda_1\big(\widehat{\mathbf{B}}_a(2^j)\big)  \leq \lambda_{\ell}\big(\widetilde{\mathbf{B}}_a(2^j)\big)
 $$
\begin{equation}\label{e:discrete_approx_eigenvalue_bounds}
\leq \lambda_{p} \big(\widetilde{\mathbf{B}}_a(2^j) - \widehat{\mathbf{B}}_a(2^j)\big)  +\lambda_p\big(\widehat{\mathbf{B}}_a(2^j)\big),\quad \ell = 1,\hdots, p(n),
\end{equation}
with $\widehat{\mathbf{B}}_a(2^j)$ as in \eqref{e:B^(j)}. In addition, as a consequence of Lemma \ref{l:vecBtilde(2^j)_asympt} and Weyl's inequality \eqref{e:Weyls_inequalities_vershynin}, it follows that
\begin{equation}\label{e:discrete_approximation_eigenvalue_convergence}
\big|\lambda_{1} \big(\widetilde{\mathbf{B}}_a(2^j) - \widehat{\mathbf{B}}_a(2^j)\big)\big|\stackrel{\bbP}\to 0,\quad \big|\lambda_{p} \big(\widetilde{\mathbf{B}}_a(2^j) - \widehat{\mathbf{B}}_a(2^j)\big)\big| \stackrel{\bbP}\to 0.
\end{equation}
Moreover, $ \widehat{\mathbf{B}}_a(2^j) \stackrel{d}= \mathbf{W}_X(2^j)$ (see expression \eqref{e:B^(j)}) implies that $\lambda_\ell\big(\widehat{\mathbf{B}}_a(2^j)  \big)  \stackrel{d}= \lambda_\ell\big(\mathbf{W}_X(2^j)\big)$ for $\ell = 1,...,p(n)$. Then, by expressions \eqref{e:eigenvalues_wavelete_decorelated}, \eqref{e:eigenvalueboundsSigma}, and by Lemma \ref{l:bounds_equal_in_distribution}, there exist strictly positive constants $m,\ M$ such that
\begin{equation}\label{e:oP(1)+m=<lambda_ell(B-hat)=<M+oP(1)}
o_\bbP(1)+ m\leq \lambda_\ell\big(\widehat{\mathbf{B}}_a(2^j)\big)  \leq M+ o_\bbP(1) \quad \textnormal{a.s.}, \quad \ell=1,\hdots,p(n).
\end{equation}
Consequently, making use of \eqref{e:discrete_approx_eigenvalue_bounds}, \eqref{e:discrete_approximation_eigenvalue_convergence} and \eqref{e:oP(1)+m=<lambda_ell(B-hat)=<M+oP(1)}, we obtain
$$
o_\bbP(1)+ m  \leq \lambda_{\ell}\big(\widetilde{\mathbf{B}}_a(2^j)\big) \leq M + o_\bbP(1)\quad \textnormal{a.s.}, \quad \ell=1,\hdots,p(n).
$$
Hence, by part $(i)$ of this lemma and by Lemma \ref{l:sandwich_bounds_to_log_convergence}, \eqref{e:fixed_scale_log_to_zero_discrete} holds. This completes the proof of $(ii)$. $\Box$\\
\end{proof}

The next proposition is analogous to Proposition C.1 in Abry and Didier \cite{abry:didier:2018:dim2}. It establishes that, at a given scale $2^j$, the root-mean squared difference  between the discrete- and continuous-time wavelet coefficients for rH-fBm is bounded, up to a constant factor, by $2^{j/2}$ times the square root of the dimension. The proposition is used in Lemma \ref{l:vecBtilde(2^j)_asympt}.
\begin{proposition}\label{p:D-tildeD_is_bounded}
Suppose assumptions $(A1-A5)$ and $(W1-W3)$ hold. For the fixed $j \in \bbN \cup \{0\}$ and for any $k \in \bbZ$, let $\widetilde{D}_X(2^j,k)$ be the discrete-time wavelet coefficient given by Mallat's algorithm (see \eqref{e:Mallat}). Then, there is a constant $C(\psi,\phi) \geq 0$, not depending on $j$ or $k$, such that
\begin{equation}\label{e:difference_discrete_continuous_coeff}
    \big(\bbE \|D_X(2^j,k)- \widetilde{D}_X(2^j,k)\|^2\big)^{1/2} \leq C(\psi,\phi) \hspace{1mm}2^{j/2} \hspace{1mm}p(n)^{1/2}.
\end{equation}
\end{proposition}
\begin{proof}

Conditionally on
\begin{equation}\label{e:diagonal_H_=_instance}
\mathds{H}_n = \mathbb{H}_n,
\end{equation}
$X$ is an ofBm. So, let $H_\ell$ denote $\ell$--th entry of the diagonal of the matrix $\mathbb{H}_n$. Then,
\begin{equation}\label{e:discrete_approx_sum}
\bbE \big[\|D_X(2^j,k)- \widetilde{D}_X(2^j,k)\|^2\big|\mathds{H}_n =  \mathbb{H}_n\big]  = \sum_{\ell = 1}^{p(n)} \bbE \big[ |d_{\mathcal{H}_\ell}(2^j,k) -\widetilde{d}_{\mathcal{H}_\ell}(2^j,k)|^2 \big| \mathcal{H}_\ell = H_\ell \big],
\end{equation}
where $\widetilde{d}_{\mathcal{H}_\ell}(2^j,k)$ is the univariate discrete wavelet transformation corresponding to the $\ell$-th entry of the vector $\widetilde{D}_X(2^j,k)$. Suppose for the moment that, for each $\breve{H} \in \textnormal{supp} \ \pi (d H)$,
\begin{equation}\label{e:univariate_discrete_approx}
    \bbE  |d_{\breve{H}}(2^j,k) -\widetilde{d}_{\breve{H}}(2^j,k)|^2  \leq C(\psi,\phi)2^j,
\end{equation}
where $C(\psi,\phi)$ is a constant only depending on $\psi$ and $\phi$. Then, \eqref{e:difference_discrete_continuous_coeff} is an immediate consequence of expressions \eqref{e:discrete_approx_sum} and \eqref{e:univariate_discrete_approx}.

So, it remains to show that \eqref{e:univariate_discrete_approx} holds. Indeed, by Taqqu \cite{taqqu:2003}, the univariate standard fBm admits the time-domain representation
\begin{equation}\label{e:fBm_time_representation}
\{B_H(t)\}_{t\in \bbR} \stackrel{\textnormal{f.d.d.}}=\Big\{\int_\bbR f_H(t,u) B(du)\Big\}_{t \in \bbR},
\end{equation}
where
\begin{equation}\label{e:fBm_time_kernel}
f_H(t,u) =\begin{cases}
\frac{\sqrt{\Gamma(2H+1)\sin( \pi H)}}{\Gamma(H+ \frac{1}{2})}  \{(t-u)^{H-1/2}_+-(-u)^{H-1/2}_+ \}, & H\in (0,1/2) \cup (1/2,1)\\
\indicator_{\{0\leq u \leq t\}} - \indicator_{\{t\leq u < 0\}}, & H = 1/2
\end{cases}.
\end{equation}
In \eqref{e:fBm_time_representation}, $B(du)$ is a real-valued, orthogonal-increment Gaussian random measure, and in \eqref{e:fBm_time_kernel}, $x_+ = x$ if $x\geq 0$ and $x_+ = 0$ otherwise. In light of \eqref{e:fBm_time_representation}, of the compact support of $\psi$ (see assumption ($W2$)) as well as of relation \eqref{e:CramerLeadbetter_condition_un_conditioned}, an adaptation of the proof of Theorem 11.4.1 in Samorodnitsky and Taqqu \cite{samorodnitsky:taqqu:1994} (a type of stochastic Fubini's theorem) yields
\begin{equation}\label{e:stoevProp2.1_analogue}
d_H(2^j,k) \stackrel{d}= 2^{-j/2}\int_\bbR \Big(\int_\bbR \psi(2^{-j}t-k) f_H(t,s)dt\Big) B(ds).
\end{equation}
Then, under assumptions $(W1-W3)$, the proofs of Lemma 6.1, Lemma 6.2 and Proposition 2.4 in Stoev et al.\ \cite{stoev:pipiras:taqqu:2002} can be directly adapted for the time-domain kernel \eqref{e:fBm_time_kernel}. However, for $H=1/2$, we note that in order to adapt the proofs of Lemma 6.1 and Proposition 2.4 in Stoev et al.\ \cite{stoev:pipiras:taqqu:2002}, the following identities are required, i.e.,
$$
\int_\bbR | f_{H=1/2}(t_1,u) -  f_{H=1/2}(t_2,u)|^2 du =  |t_1-t_2|
$$
and
\begin{equation}\label{e:independent_from_i}
f_{H=1/2}(i+2^jk,u) = f_{H=1/2}(i/2^j , u/2^j-k) + \indicator_{\{0\leq u/2^j < k\}} - \indicator_{\{k<u/2^j \leq 0\}}
\end{equation}
(see Stoev et al.\ \cite{stoev:pipiras:taqqu:2002}, p.\ 1895, first and second columns, respectively). This yields expression \eqref{e:univariate_discrete_approx}. $\Box$\\
\end{proof}

The following lemma is analogous to Lemma C.2 in Abry and Didier \cite{abry:didier:2018:dim2}. It is used in the proof of Lemma \ref{l:fixed_scale_log_limit_discrete}.
\begin{lemma}\label{l:vecBtilde(2^j)_asympt}
Suppose assumptions ($A1-A5$) and ($W1-W3$) hold. Let $\widetilde{\mathbf{B}}_{a}(2^j)$, $\widehat{\mathbf{B}}_a(2^j)$ be as in \eqref{e:Btilde-nu} and \eqref{e:B^(j)}, respectively. Then,
\begin{equation}\label{e:sqrt(Kaj)(Btilde-Bhat)}
\bbE \| \widetilde{\mathbf{B}}_{a}(2^j) - \widehat{\mathbf{B}}_a(2^j) \|_{\textnormal{op}} \rightarrow 0, \quad n \rightarrow \infty.
\end{equation}
\end{lemma}
\begin{proof}
By Minkowsky's inequality, the left-hand side of \eqref{e:sqrt(Kaj)(Btilde-Bhat)} is bounded by
$$
\frac{1}{n_{a,j}} \sum^{n_{a,j}}_{k=1} \bbE \big\|
a(n)^{-(\mathds{H}_n+(1/2)I) } \widetilde{D}_X(a(n)2^j,k)\widetilde{D}_X(a(n)2^j,k)^* a(n)^{-(\mathds{H}_n+(1/2)I)}
$$
\begin{equation}\label{e:bounding_normalized_distance_squared_wavecoef_vec}
- a(n)^{-(\mathds{H}_n+(1/2)I) } D_X(a(n)2^j,k)D_X(a(n)2^j,k)^*a(n)^{-(\mathds{H}_n+(1/2)I)} \big\|_{\textnormal{op}}.
\end{equation}
However, for $k = 1,\hdots,n_{a,j}$, the deviation between the discrete- and continuous-time (non-normalized) wavelet variance terms can be recast in the form
$$
\widetilde{D}_X(a(n)2^j,k)\widetilde{D}_X(a(n)2^j,k)^* - D_X(a(n)2^j,k)D_X(a(n)2^j,k)^*
$$
$$
= [\widetilde{D}_X(a(n)2^j,k)-D_X(a(n)2^j,k)][\widetilde{D}_X(a(n)2^j,k)-D_X(a(n)2^j,k)]^*
$$
$$
+ [\widetilde{D}_X(a(n)2^j,k)-D_X(a(n)2^j,k)]D_X(a(n)2^j,k)^*
+ D_X(a(n)2^j,k)[\widetilde{D}_X(a(n)2^j,k)-D_X(a(n)2^j,k)]^*.
$$
For notational simplicity, let
\begin{equation}\label{e:D-breve(a(n)2^j,k)}
\breve{D}_X(a(n)2^j,k) := a(n)^{-(\mathds{H}_n+(1/2)I)}D_X(a(n)2^j,k).
\end{equation}
Then, from the fact that the operator norm is sub-multiplicative, it follows that each term under the summation sign in \eqref{e:bounding_normalized_distance_squared_wavecoef_vec} can be bounded by
$$
\bbE\Big( \|a(n)^{-(\mathds{H}_n+(1/2)I)}\big\{[\widetilde{D}_X(a(n)2^j,k)-D(a(n)2^j,k)][\widetilde{D}_X(a(n)2^j,k)-D_X(a(n)2^j,k)]^*
$$
$$
+ [\widetilde{D}_X(a(n)2^j,k)-D_X(a(n)2^j,k)]D_X(a(n)2^j,k)^*
$$
$$
+ D_X(a(n)2^j,k)[\widetilde{D}_X(a(n)2^j,k)-D_X(a(n)2^j,k)]^*\big\}a(n)^{-(\mathds{H}_n+(1/2)I)}\|_{\textnormal{op}}\Big)
$$
$$
\leq \bbE \Big(\|a(n)^{-(\mathds{H}_n+(1/2)I)}\|^{2}_{\textnormal{op}}\cdot\|[\widetilde{D}_X(a(n)2^j,k)-D_X(a(n)2^j,k)][\widetilde{D}_X(a(n)2^j,k)-D_X(a(n)2^j,k)]^* \|_{\textnormal{op}}\Big)
$$
$$
+ \bbE\Big(\|a(n)^{-(\mathds{H}_n+(1/2)I)}\|_{\textnormal{op}} \cdot \| [\widetilde{D}_X(a(n)2^j,k)-D_X(a(n)2^j,k)] \breve{D}_X(a(n)2^j,k)^*  \|_{\textnormal{op}}\Big)
$$
$$
+ \bbE\Big(\| \breve{D}_X(a(n)2^j,k)  [\widetilde{D}_X(a(n)2^j,k)-D_X(a(n)2^j,k)]^*\|_{\textnormal{op}} \cdot \|a(n)^{-(\mathds{H}_n+(1/2)I)}\|_{\textnormal{op}} \Big)
$$
$$
= \bbE \big(\|a(n)^{-(\mathds{H}_n+(1/2)I)}\|^{2}_{\textnormal{op}}\cdot \|\widetilde{D}_X(a(n)2^j,k)-D_X(a(n)2^j,k)\|^2\big)
$$
\begin{equation}\label{e:two_term_bound_on_(Btilde-Bhat)}
+ 2 \bbE\big( \|a(n)^{-(\mathds{H}_n+(1/2)I)}\|_{\textnormal{op}}  \cdot\|\breve{D}_X(a(n)2^j,k)  \|\cdot\|\widetilde{D}_X(a(n)2^j,k)-D_X(a(n)2^j,k)\|\big).
\end{equation}
In \eqref{e:two_term_bound_on_(Btilde-Bhat)}, the last equality follows from the fact that, for a rank one matrix, ${\mathbf x}_1 {\mathbf x}_2^\top$, with ${\mathbf x}_1,{\mathbf x}_2\in\Reals^{p(n)}$, $\|{\mathbf x}_1{\mathbf x}_2^\top\|_{\textnormal{op}} = \|{\mathbf x}_1\|\|{\mathbf x}_2\|$. We now develop upper bounds for each term in the equality \eqref{e:two_term_bound_on_(Btilde-Bhat)}. Since $a(n)^{-(\mathds{H}_n+(1/2)I)}$ is a diagonal matrix, it follows that, for $\varpi$ as in \eqref{e:pi(dh)} (see assumption $(A1)$),
\begin{equation}\label{e:a_diagonal_matrix_bound}
    \|a(n)^{-(\mathds{H}_n+(1/2)I)}\|_{\textnormal{op}} \leq a(n)^{-(\varpi +1/2)}.
\end{equation}
 Hence, by Proposition \ref{p:D-tildeD_is_bounded} and expression \eqref{e:a_diagonal_matrix_bound},
$$
\bbE \big(\|a(n)^{-(\mathds{H}_n+(1/2)I)}\|^{2}_{\textnormal{op}}\cdot \|\widetilde{D}_X(a(n)2^j,k)-D_X(a(n)2^j,k)\|^2\big)
$$
$$
\leq a(n)^{-2\varpi-1} \bbE \big(\|\widetilde{D}_X(a(n)2^j,k)-D_X(a(n)2^j,k) \|^2 \big)
$$
\begin{equation}\label{e:(Btilde-Bhat)_term_one_bound}
\leq a(n)^{-2\varpi} C(\psi,\phi)^2 \hspace{0.5mm}2^j p(n) .
\end{equation}
In addition, again by \eqref{e:a_diagonal_matrix_bound},
$$
\bbE\big( \|a(n)^{-(\mathds{H}_n+(1/2)I)}\|_{\textnormal{op}}  \cdot \|\breve{D}_X(a(n)2^j,k)\|\cdot \|\widetilde{D}_X(a(n)2^j,k)-D_X(a(n)2^j,k)\|\big)
$$
\begin{equation}\label{e:(Btilde-Bhat)_term_two}
 \leq a(n)^{-(\varpi+1/2)} \bbE \big(\|\breve{D}_X(a(n)2^j,k) \| \cdot \|\widetilde{D}_X(a(n)2^j,k)-D_X(a(n)2^j,k)\|\big).
\end{equation}
Then, by the Cauchy-Schwarz inequality,
$$
\bbE \big(\|\breve{D}_X(a(n)2^j,k)\|\cdot \|\widetilde{D}_X(a(n)2^j,k)-D_X(a(n)2^j,k)\|\big)
$$
$$
\leq \Big(\bbE \|\breve{D}_X(a(n)2^j,k)\|^2\cdot \bbE \|\widetilde{D}_X(a(n)2^j,k)-D_X(a(n)2^j,k)\|^2 \Big)^{1/2}
$$
$$
= \Big(\bbE \|D_X(2^j,k)\|^2 \cdot \bbE \|\widetilde{D}_X(a(n)2^j,k)-D_X(a(n)2^j,k)\|^2 \Big)^{1/2},
$$
where the equality is a consequence of \eqref{e:D-breve(a(n)2^j,k)} and Lemma \ref{l:W(a(n)2^j)_D(a(n)2^j,k)_selfsimilar}, $(i)$. Observe that, since the $p(n)$ entries of $D_X(2^j,k)$ are i.i.d.,
$$
\bbE \|D_X(2^j,k)\|^2 =  \sum_{\ell=1}^{p(n)} \bbE d_{{\mathcal H}}(2^j,k)_\ell^2 = p(n)\cdot  \bbE d_{{\mathcal H}}(2^j,k)^2 = p(n)\cdot  \bbE d_{{\mathcal H}}(2^j,0)^2,
$$
where the last equality follows from the conditional stationarity of $ d_{\mathcal{H}}(2^j,k)$ (see Lemma  \ref{l:D(2^j,k)_is_stationary}, $(i)$). Thus, by Proposition \ref{p:D-tildeD_is_bounded}, the right-hand side of expression \eqref{e:(Btilde-Bhat)_term_two} can be bounded by
\begin{equation}\label{e:(Btilde-Bhat)_term_two_bound}
  a(n)^{-\varpi} p(n)\big(\bbE d_{{\mathcal H}}(2^j,0)_1^2 \big)^{1/2} C(\psi,\phi) (2^j)^{1/2}.
\end{equation}
Therefore, applying the bounds \eqref{e:(Btilde-Bhat)_term_one_bound} and \eqref{e:(Btilde-Bhat)_term_two_bound} to the terms in the sum \eqref{e:two_term_bound_on_(Btilde-Bhat)} yields
\begin{equation}
    \bbE \| \widetilde{\mathbf{B}}_{a}(2^j) - \widehat{\mathbf{B}}_a(2^j) \|_{\textnormal{op}} \leq A \cdot p(n)\cdot a(n)^{-2\varpi}  + B\cdot a(n)^{-\varpi} \cdot p(n),
\end{equation}
where $A$ and $B$ are positive constants that do not depend on $n$. By condition \eqref{e:p(n),a(n)_conditions} (see assumption $(A4)$), \eqref{e:sqrt(Kaj)(Btilde-Bhat)} holds. $\Box$\\
\end{proof}

\section{Additional results}
\label{s:additional results}

In this section, we establish auxiliary results used throughout Sections \ref{s:Continuous time} and \ref{s:discrete time}. In regard to the notation, throughout this section recall that $O(n)$ and $\textnormal{Haar}(O(n))$ denote, respectively, the orthogonal group on $\bbR^n$ and the Haar probability measure on $O(n)$.

The following lemma is used to establish Propositions \ref{p:Sum_of_deformed_wisharts} and \ref{p:k_sum_of_deformed_wisharts}. The lemma shows that a subset of $p_1 = p_1(n)$ Euclidean vectors and of $p_2 = p_2(n)$ columns obtained from a Haar-distributed random matrix are approximately orthogonal \textit{as long as} $p_1+p_2$ is sufficiently small with respect to the overall dimension $n \rightarrow \infty$.

\begin{lemma}\label{l:orthogonal_spaces}
For any $n \in \bbN$ and $1 \leq p = p(n) \leq n$, suppose that $p_i = p_i(n,\omega) \in \bbN$, $i=1,2$, are two sequences of random variables such that
\begin{equation}\label{e:moderatly_high_dimensions}
p_1 + p_2 = p = o(\sqrt{n}) 
\end{equation}
(\textbf{n.b.}: in \eqref{e:moderatly_high_dimensions}, $p=p(n)$ is deterministic). Let
\begin{equation}\label{e:mathcalE}
\mathcal{E} = \textnormal{span}\{{\mathbf e}_1,...,{\mathbf e}_{p_1} \}\subseteq \bbR^{n},
\end{equation}
where ${\mathbf e}_\ell$ denotes the $\ell$-th Euclidean vector. Also, let
\begin{equation}\label{e:mathcalU}
\mathcal{U} \equiv \mathcal{U}_n = \textnormal{span}\{{\mathbf u}_1,...,{\mathbf u}_{p_2} \}\subseteq \bbR^{n},
\end{equation}
where ${\mathbf u}_\ell = {\mathbf u}_{\ell,n}(\omega)$ denotes the $\ell$-th column of the random matrix
\begin{equation}\label{e:Un_sim_Haar}
{\mathbf U} \equiv {\mathbf U}_{n} \sim \textnormal{Haar}(O(n)).
\end{equation}
The following claims hold.
\begin{itemize}
\item [$(i)$]
\begin{equation}\label{e:Euclidean_and_Haar_vectors_are_LI}
\textnormal{dim}\big(\textnormal{span}\{ \mathbf{e}_1,\hdots,\mathbf{e}_{p_1}, \mathbf{u}_1,\hdots \mathbf{u}_{p_2}\}\big)=p \quad   a.s.
\end{equation}
\item [$(ii)$] For any (random) sequence of unit vectors $\{\boldsymbol{v}_{n}\}_{n \in \bbN} \subseteq  \textnormal{span}\{ \mathbf{e}_1,\hdots,\mathbf{e}_{p_1}, \mathbf{u}_1,\hdots \mathbf{u}_{p_2}\}$, almost surely there exist sequences of vectors $\boldsymbol{v}_{\mathcal{E},n} \in \mathcal{E}$, $\boldsymbol{v}_{\mathcal{U},n} \in \mathcal{U}$ based on which we can write the direct sum
    \begin{equation}\label{e:v=vE+vU}
    \boldsymbol{v}_n = \boldsymbol{v}_{\mathcal{E},n} + \boldsymbol{v}_{\mathcal{U},n},
    \end{equation}
such that
\begin{equation}\label{e:basis_lower_bound}
\|\boldsymbol{v}_{\mathcal{E},n} \|^2+ \| \boldsymbol{v}_{\mathcal{U},n}\|^2 \geq 1 - o_{\bbP}(1)
\end{equation}
and
\begin{equation}\label{e:basis_upper_bound}
\max\{\|\boldsymbol{v}_{\mathcal{E},n} \|^2, \| \boldsymbol{v}_{\mathcal{U},n}\|^2 \}\leq 1 + o_{\bbP}(1).
\end{equation}
\end{itemize}
\end{lemma}
\begin{proof}
We start by establishing $(i)$. First, we prove \eqref{e:Euclidean_and_Haar_vectors_are_LI} conditionally on fixed values of $p_1= p_1(n)$ and $p_2= p_2(n)$. For simplicity, we use the notation $\widetilde{p}_1= \widetilde{p}_1(n)$ and $\widetilde{p}_2=\widetilde{p}_2(n)$ for these deterministic values. Note that choosing any  (deterministic number) $\widetilde{p}_2$ of columns of ${\mathbf U}$ is equal, in law, to the following procedure: $(i)$ pick $\widetilde{p}_2$ vectors $\bbR^n \supseteq \{\mathbf{Z}_i \}_{i=1,\hdots, \widetilde{p}_2 }\stackrel{\textnormal{i.i.d.}}\sim \mathcal{N}(0,I_n)$
$(ii)$ next, orthonormalize these $\widetilde{p}_2$ vectors by means of the Gram-Schmidt algorithm (cf.\ Meckes \cite{meckes:2014:women}, p.\ 8). In particular, the resulting vectors $\widetilde{{\mathbf u}}_1,...,\widetilde{{\mathbf u}}_{\widetilde{p}_2}$ satisfy
\begin{equation}\label{e:Gram-Schmidt_Haar}
\widetilde{{\mathbf u}}_1,...,\widetilde{{\mathbf u}}_{\widetilde{p}_2}
\stackrel{d}= {\mathbf u}_1,...,{\mathbf u}_{\widetilde{p}_2}.
\end{equation}
Let $\mathbf{e}_1,...,\mathbf{e}_{\widetilde{p}_1}$ be the first $\widetilde{p}_1$ Euclidean vectors (\textbf{n.b.}: $\widetilde{p}_1$ is deterministic). Then,
$$
\bbP\big(\textnormal{dim}(\textnormal{span}\{ \mathbf{e}_1,\hdots,\mathbf{e}_{\widetilde{p}_1}, \mathbf{Z}_1,\hdots \mathbf{Z}_{\widetilde{p}_2}\})=p\big) = \bbP\Big( \bigcap_{j=1}^{\widetilde{p}_1}\{\mathbf{e}_j\notin \textnormal{span}\{\mathbf{Z}_1,\hdots,\mathbf{Z}_{\widetilde{p}_2}\}\Big)
$$
$$
= \bbP\Big( \bigcap_{j=1}^{\widetilde{p}_1}\{\mathbf{e}_j\notin \textnormal{span}\{\widetilde{{\mathbf u}}_1,...,\widetilde{{\mathbf u}}_{\widetilde{p}_2}\}\Big) =
\bbP\Big( \bigcap_{j=1}^{\widetilde{p}_1}\{\mathbf{e}_j\notin \textnormal{span}\{{\mathbf u}_1,...,{\mathbf u}_{\widetilde{p}_2}\}\Big)
$$
$$
=\bbP\big(\textnormal{dim}(\textnormal{span}\{ \mathbf{e}_1,\hdots,\mathbf{e}_{\widetilde{p}_1},{\mathbf u}_1,...,{\mathbf u}_{\widetilde{p}_2}\})=p\big),
$$
where the first and fourth equalities follow from Lemma \ref{l:two_spaces_LI_equivalence}, the second equality follows from the Gram-Schmidt algorithm, and the third equality is a consequence of \eqref{e:Gram-Schmidt_Haar}. Since $\textnormal{dim}(\textnormal{span}\{\mathbf{e}_1,..,\mathbf{e}_{\widetilde{p}_1}\}) = \widetilde{p}_1$ and by Lemma \ref{l:gaussian_rv_are_LI}, $\textnormal{dim}(\textnormal{span}\{ \mathbf{Z}_1,\hdots \mathbf{Z}_{\widetilde{p}_2}\})=\widetilde{p}_2$ a.s., we now consider the case where  $\widetilde{p}_2 \neq 0$. In particular, for $\widetilde{p}_1, \widetilde{p}_2 \neq 0 $, a simple adaptation of the proof of Lemma \ref{l:gaussian_rv_are_LI} yields $ \bbP\big(\textnormal{dim}(\textnormal{span}\{ \mathbf{e}_1,\hdots,\mathbf{e}_{\widetilde{p}_1}, \mathbf{Z}_1,\hdots \mathbf{Z}_{\widetilde{p}_2}\})=\widetilde{p}_1 + \widetilde{p}_2 = p\big) = 1$. Since this is true for any fixed pair $\widetilde{p}_1$, $\widetilde{p}_2$, then \eqref{e:Euclidean_and_Haar_vectors_are_LI} also holds unconditionally. Thus, $(i)$ is established.

We now turn to $(ii)$. The almost sure existence of the decomposition \eqref{e:v=vE+vU} is a consequence of \eqref{e:Euclidean_and_Haar_vectors_are_LI}. Now, for any $\boldsymbol{v}_{\mathcal{E},n}$ and $\boldsymbol{v}_{\mathcal{U},n}$ as in \eqref{e:v=vE+vU}, observe that $ 1 = \| \boldsymbol{v}_n\|^2 = \|\boldsymbol{v}_{\mathcal{E},n} \|^2+ \| \boldsymbol{v}_{\mathcal{U},n}\|^2  + 2\boldsymbol{v}_{\mathcal{E},n}^*\boldsymbol{v}_{\mathcal{U},n}.$ Hereinafter, we write $\boldsymbol{v}_{\mathcal{E}}= \boldsymbol{v}_{\mathcal{E},n}$ and $\boldsymbol{v}_{\mathcal{U}} = \boldsymbol{v}_{\mathcal{U},n}$ for notational simplicity. From this, we obtain the almost sure upper and lower bounds
$$
\|\boldsymbol{v}_{\mathcal{E}} \|^2+ \| \boldsymbol{v}_{\mathcal{U}}\|^2  + 2| \boldsymbol{v}_{\mathcal{E}}^*\boldsymbol{v}_{\mathcal{U}}| \geq 1 \geq \|\boldsymbol{v}_{\mathcal{E}} \|^2+ \| \boldsymbol{v}_{\mathcal{U}}\|^2  - 2| \boldsymbol{v}_{\mathcal{E}}^*\boldsymbol{v}_{\mathcal{U}}|.
$$
Hence,
\begin{equation}\label{e:basis_bounds}
    \|\boldsymbol{v}_{\mathcal{E}} \|^2+ \| \boldsymbol{v}_{\mathcal{U}}\|^2  \geq 1 - 2| \boldsymbol{v}_{\mathcal{E}}^*\boldsymbol{v}_{\mathcal{U}}| \quad \textnormal{and} \quad \max\{\|\boldsymbol{v}_{\mathcal{E}} \|^2, \| \boldsymbol{v}_{\mathcal{U}}\|^2\} \leq 1 +  2| \boldsymbol{v}_{\mathcal{E}}^*\boldsymbol{v}_{\mathcal{U}}|.
\end{equation}
Therefore, from \eqref{e:basis_bounds}, to establish  \eqref{e:basis_lower_bound} and \eqref{e:basis_upper_bound}, it suffices to show that
\begin{equation}\label{e:basis_angles_vanish}
\boldsymbol{v}_{\mathcal{E}}^*\boldsymbol{v}_{\mathcal{U}} = o_{\bbP}(1).
\end{equation}
So, let
\begin{equation}\label{e:U-trunc}
{\mathbf U}_{\textnormal{trunc}}\in {\mathcal M}(p_1,p_2,\bbR)
\end{equation}
be the upper $p_1 \times p_2$ corner of the random matrix ${\mathbf U}_{n}$ as in \eqref{e:Un_sim_Haar}, and consider the basis expansions
\begin{equation}\label{e:basis_expansions}
{\boldsymbol v}_{\mathcal{E}} =\sum_{i=1}^{p_1} \alpha_i {\mathbf e}_i, \quad {\boldsymbol v}_{\mathcal{U}} =\sum_{\ell=1}^{p_2} \beta_{\ell} {\mathbf u}_{\ell}
\end{equation}
(\textbf{n.b.}: the coefficients in \eqref{e:basis_expansions} are random and depend on $n$). Then,
\begin{equation}\label{e:v_E*v_U}
\boldsymbol{v}_{\mathcal{E}}^* \boldsymbol{v}_{\mathcal{U}} =\sum_{i=1}^{p_1} \sum_{\ell=1}^{p_2} \alpha_i \mathbf{e}_i^*\mathbf{u}_{\ell} \beta_{\ell} =
\begin{bmatrix}
\alpha_1& \hdots & \alpha_{p_1}
\end{bmatrix}
\begin{bmatrix}
\mathbf{e}_1^*\mathbf{u}_1 & \hdots & \mathbf{e}_1^*\mathbf{u}_{p_2}\\
\vdots\\
\mathbf{e}_{p_1}^*\mathbf{u}_1 & \hdots & \mathbf{e}_{p_1}^*\mathbf{u}_{p_2}
\end{bmatrix}
\begin{bmatrix}
\beta_1\\
\vdots \\
\beta_{p_2}
\end{bmatrix}
= \boldsymbol{\alpha}^* {\mathbf U}_{\textnormal{trunc}} \boldsymbol{\beta},
\end{equation}
where $\boldsymbol\alpha \in \bbR^{p_1}$ and $\boldsymbol\beta \in \bbR^{p_2}$ are vectors of coefficients. Note that, by \eqref{e:basis_expansions},
\begin{equation}\label{e:||alpha||,||beta||}
\|\boldsymbol\alpha\| =\| \boldsymbol{v}_{\mathcal{E}}\|, \quad \|\boldsymbol\beta\| =\| \boldsymbol{v}_{\mathcal{U}}\|.
\end{equation}
Now let
\begin{equation}\label{e:U-trunc,p}
{\mathbf U}_{\textnormal{trunc},p} \in {\mathcal M}(p,\bbR)
\end{equation}
be the upper $p \times p$ corner  ($p \leq n$) of the random matrix ${\mathbf U}_{n}$ as in \eqref{e:Un_sim_Haar}. In particular, since $\max\{p_1,p_2\}\leq p$, the matrix ${\mathbf U}_{\textnormal{trunc}} \in {\mathcal M}(p_1,p_2,\bbR)$ is embedded in the matrix ${\mathbf U}_{\textnormal{trunc},p} \in {\mathcal M}(p,\bbR)$. Hence,
\begin{equation}\label{e:||U_trunc||=<||U_trunc,p||}
\|{\mathbf U}_{\textnormal{trunc}}\|_{\op} \leq \|{\mathbf U}_{\textnormal{trunc},p}\|_{\op}.
\end{equation}
Then, by relation \eqref{e:v_E*v_U} and the Cauchy-Schwarz inequality, as well as by relations \eqref{e:basis_bounds} and \eqref{e:||U_trunc||=<||U_trunc,p||},
$$
| \boldsymbol{v}_{\mathcal{E}}^*\boldsymbol{v}_{\mathcal{U}}|\leq \|\boldsymbol{v}_{\mathcal{E}}\|\cdot \|\boldsymbol{v}_{\mathcal{U}}\| \cdot \|{\mathbf U}_{\textnormal{trunc}}\|_{\textnormal{op}}
$$
\begin{equation}\label{e:angle_bound}
\leq \|{\mathbf U}_{\textnormal{trunc}}\|_{\textnormal{op}} \cdot (1 + 2| \boldsymbol{v}_{\mathcal{E}}^*\boldsymbol{v}_{\mathcal{U}}| )
 \leq \|{\mathbf U}_{\textnormal{trunc,p}}\|_{\textnormal{op}} \cdot (1 + 2| \boldsymbol{v}_{\mathcal{E}}^*\boldsymbol{v}_{\mathcal{U}}| ).
\end{equation}
Consequently, we can write $|\boldsymbol{v}_{\mathcal{E}}^*\boldsymbol{v}_{\mathcal{U}}| (1 - 2 \|{\mathbf U}_{\textnormal{trunc,p}}\|_{\textnormal{op}}) \leq \|{\mathbf U}_{\textnormal{trunc},p}\|_{\textnormal{op}}$. However, by Jiang \cite{jiang:2008}, Theorem 3, under $p = o(\sqrt{n})$ (see condition \eqref{e:moderatly_high_dimensions}),
\begin{equation}\label{e:||U_trunc,p||=oP(1)}
\|{\mathbf U}_{\textnormal{trunc},p}\|_{\textnormal{op}} = o_{\bbP}(1).
\end{equation}
So, fix $0 < \delta < 1$ and let $\Omega_{n,\delta}= \{\omega: \|{\mathbf U}_{\textnormal{trunc},p}\|_{\textnormal{op}} < \delta /2\}$. Then, for any $\varepsilon > 0$, relations \eqref{e:angle_bound} and \eqref{e:||U_trunc,p||=oP(1)} imply that
$$
\bbP\big( | \boldsymbol{v}_{\mathcal{E}}^*\boldsymbol{v}_{\mathcal{U}}| > \varepsilon \big)
= \bbP\big( | \boldsymbol{v}_{\mathcal{E}}^*\boldsymbol{v}_{\mathcal{U}}| > \varepsilon \cap \Omega_{n,\delta} \big)
+ \bbP\big( | \boldsymbol{v}_{\mathcal{E}}^*\boldsymbol{v}_{\mathcal{U}}| > \varepsilon \cap A^c_{n,\delta} \big)
$$
$$
\leq \bbP\big( \|{\mathbf U}_{\textnormal{trunc},p}\|_{\textnormal{op}} > \varepsilon (1-\delta)\big) + o(1) = o(1).
$$
This shows \eqref{e:basis_angles_vanish}. Hence, $(ii)$ is established.  $\Box$\\
 \end{proof}

The next proposition is used in the proofs of Propositions \ref{p:Sum_of_deformed_wisharts} and \ref{p:k_sum_of_deformed_wisharts}. Note that, in the statement and proof, we make use of expressions appearing in Lemma \ref{l:orthogonal_spaces}. The proposition shows that, for random $p_1, p_2 \in \bbN$, a rank $p_1$ positive diagonal matrix plus a rank $p_2$ symmetric positive definite perturbation tends to have rank $p_1 + p_2$.
\begin{proposition}\label{p:D_1_D_2_sum_bound}
For any $n \in \bbN$ and $1 \leq p = p(n) \leq n$, suppose that $p_i = p_i(n,\omega) \in \bbN$, $i=1,2$, are two sequences of random variables
satisfying \eqref{e:moderatly_high_dimensions}. Let
\begin{equation}\label{e:D1_D2_def}
{\mathbf D}_1 \equiv {\mathbf D}_{1,n}=\textnormal{diag}( \zeta_{1,n},...,\zeta_{p_1,n},0,...,0),\quad  {\mathbf D}_2 \equiv {\mathbf D}_{2,n}=\textnormal{diag}( \eta_{1,n},...,\eta_{p_2,n},0,...0)
\end{equation}
be two independent sequences of random matrices in ${\mathcal S}_{\geq 0}(n,\bbR)$ with
$$
0 < \zeta_{1,n} \leq \hdots \leq \zeta_{p_1,n}\quad and \quad 0 < \eta_{1,n} \leq \hdots \leq \eta_{p_2,n} \quad a.s.
$$
Further assume that there exist positive constants $\zeta_{\min}$, $\zeta_{\max}$, $\eta_{\min}$ and $\eta_{\max}$ such that
 \begin{equation}\label{e:eigenvalues_lower_limit}
 \zeta_{1,n}  \geq \zeta_{\min} + o_{\bbP}(1)  \quad \textnormal{and} \quad \eta_{1,n} \geq \eta_{\min} + o_{\bbP}(1)
 \end{equation}
 and
 \begin{equation}\label{e:eigenvalues_upper_limit}
\zeta_{p_1,n}  \leq \zeta_{\max} + o_{\bbP}(1)   \quad \textnormal{and} \quad \eta_{p_2,n} \leq  \eta_{\max} + o_{\bbP}(1).
 \end{equation}
Now let ${\mathbf U} \equiv {\mathbf U}_n \sim \textnormal{Haar}(O(n))$ be a sequence of random matrices which are independent of ${\mathbf D}_1={\mathbf D}_{1,n}$ and ${\mathbf D}_2={\mathbf D}_{2,n}$. Then,
\begin{equation}\label{e:diagonal_bound}
\lambda_{n-p+1}\big({\mathbf D}_1 + {\mathbf U} {\mathbf D}_2 {\mathbf U}^* \big) \geq \textnormal{min}\{\eta_{\min},\zeta_{\min}\} + o_\bbP(1).
\end{equation}
\end{proposition}
\begin{proof}
Starting from the left-hand side of \eqref{e:diagonal_bound}, note that ${\mathbf D}_1 + {\mathbf U} {\mathbf D}_2 {\mathbf U}^* \in {\mathcal S}_{\geq 0}(n,\bbR)$. Now let ${\mathcal V}_p\subseteq \bbR^n$ be any $p$-dimensional subspace. By the Courant-Fischer min-max theorem (e.g., Horn and Johnson \cite{horn:johnson:2012}, p.\ 236),
\begin{equation}\label{e:courant_fischer}
    \lambda_{n-p+1}\big( {\mathbf D}_1 + {\mathbf U} {\mathbf D}_2 {\mathbf U}^* \big) \geq \inf_{\boldsymbol{v}\in {\mathcal V}_p \cap \bbS^{n-1}} \boldsymbol{v}^* ({\mathbf D}_1 + {\mathbf U} {\mathbf D}_2 {\mathbf U}^*)\boldsymbol{v}.
\end{equation}
Let ${\mathcal E}$ and ${\mathcal U}$ be as in \eqref{e:mathcalE} and \eqref{e:mathcalU}, respectively. Define ${\mathcal V}:= \mathcal{E} \oplus \mathcal{U}$. By Lemma \ref{l:orthogonal_spaces}, $\textnormal{dim}\hspace{0.5mm}{\mathcal V} = p$ a.s. Also, by the same lemma, for any $\boldsymbol{v} \in {\mathcal V}$ we can almost surely find $\boldsymbol{v}_\mathcal{E}  \in \mathcal{E}$ and $\boldsymbol{v}_\mathcal{U}  \in \mathcal{U}$ such that $\boldsymbol{v} =\boldsymbol{v}_\mathcal{E} +\boldsymbol{v}_\mathcal{U}$. Thus,
$$
\boldsymbol{v}^*({\mathbf D}_1 + {\mathbf U} {\mathbf D}_2 {\mathbf U}^*) \boldsymbol{v} = (\boldsymbol{v}_\mathcal{E}^* +\boldsymbol{v}_\mathcal{U}^*)({\mathbf D}_1 + {\mathbf U}{\mathbf D}_2{\mathbf U}^*)(\boldsymbol{v}_\mathcal{E} +\boldsymbol{v}_\mathcal{U})
$$
\begin{equation}\label{e:v^T(D1+UD2U*)v_basic_bound}
=\boldsymbol{v}_\mathcal{E}^* {\mathbf D}_1 \boldsymbol{v}_\mathcal{E} + 2 \boldsymbol{v}_\mathcal{E}^* {\mathbf D}_1 \boldsymbol{v}_\mathcal{U} + \boldsymbol{v}_\mathcal{E}^* {\mathbf U}{\mathbf D}_2{\mathbf U}^* \boldsymbol{v}_\mathcal{E} + 2 \boldsymbol{v}_\mathcal{E}^* {\mathbf U}{\mathbf D}_2{\mathbf U}^*\boldsymbol{v}_\mathcal{U} +
\boldsymbol{v}_\mathcal{U}^* {\mathbf D}_1 \boldsymbol{v}_\mathcal{U} +
\boldsymbol{v}_\mathcal{U}^* {\mathbf U}{\mathbf D}_2{\mathbf U}^*\boldsymbol{v}_\mathcal{U}.
\end{equation}
In light of \eqref{e:D1_D2_def}, since ${\mathbf D}_1, {\mathbf U}{\mathbf D}_2{\mathbf U}^* \in{\mathcal S}_{\geq 0}(n,\bbR)$, then
\begin{equation}\label{e:quad_D1_quad_UD2U*_bound}
\boldsymbol{v}_\mathcal{E}^* {\mathbf D}_1 \boldsymbol{v}_\mathcal{E} \geq \|\boldsymbol{v}_{\mathcal{E}}\|^2 \zeta_{1,n},\quad \boldsymbol{v}_\mathcal{U}^* {\mathbf U} {\mathbf D}_2 {\mathbf U}^* \boldsymbol{v}_\mathcal{U} \geq \|\boldsymbol{v}_{\mathcal{U}}\|^2 \eta_{1,n}.
\end{equation}
By \eqref{e:v^T(D1+UD2U*)v_basic_bound} and \eqref{e:quad_D1_quad_UD2U*_bound}, we obtain the bound
\begin{equation*}
    \boldsymbol{v}^*({\mathbf D}_1 + {\mathbf U} {\mathbf D}_2 {\mathbf U}^*) \boldsymbol{v}\geq \|\boldsymbol{v}_{\mathcal{E}}\|^2 \zeta_{1,n} + \|\boldsymbol{v}_{\mathcal{U}}\|^2 \eta_{1,n} + 2 \boldsymbol{v}_\mathcal{U}^* {\mathbf D}_1 \boldsymbol{v}_\mathcal{E}  + 2 \boldsymbol{v}_\mathcal{E}^* {\mathbf U}{\mathbf D}_2{\mathbf U}^*\boldsymbol{v}_\mathcal{U}.
\end{equation*}
\begin{equation}\label{e:quadratic_form_bounds}
\geq   \min\{\zeta_{1,n}, \eta_{1,n}\}(1 - o_{\bbP}(1)) + 2 \boldsymbol{v}_\mathcal{U}^* {\mathbf D}_1 \boldsymbol{v}_\mathcal{E}  + 2 \boldsymbol{v}_\mathcal{E}^* {\mathbf U}{\mathbf D}_2{\mathbf U}^*\boldsymbol{v}_\mathcal{U},
\end{equation}
where the last inequality follows from relation \eqref{e:basis_lower_bound} (see Lemma \ref{l:orthogonal_spaces}). Suppose, for the moment, that
\begin{equation}\label{e:vU*D1vE=oP(1)_vE*UD2U*vU=oP(1)}
\boldsymbol{v}_\mathcal{U}^* {\mathbf D}_1 \boldsymbol{v}_\mathcal{E} = o_{\bbP}(1), \quad \boldsymbol{v}_\mathcal{E}^* {\mathbf U}{\mathbf D}_2{\mathbf U}^*\boldsymbol{v}_\mathcal{U} = o_{\bbP}(1).
\end{equation}
Then, applying \eqref{e:vU*D1vE=oP(1)_vE*UD2U*vU=oP(1)} to expression \eqref{e:courant_fischer} with bound \eqref{e:quadratic_form_bounds} yields
$$
\lambda_{n-p+1}\big( {\mathbf D}_1 + {\mathbf U} {\mathbf D}_2 {\mathbf U}^* \big) \geq \min\{\zeta_{1,n}, \eta_{1,n}\}(1 - o_{\bbP}(1)) + o_{\bbP}(1).
$$
Thus, by condition \eqref{e:eigenvalues_lower_limit}, we establish \eqref{e:diagonal_bound}.

So, we need to show the two asymptotic relations in \eqref{e:vU*D1vE=oP(1)_vE*UD2U*vU=oP(1)}. We start off with the left one. Consider the basis expansions \eqref{e:basis_expansions}, in the proof of Lemma \ref{l:orthogonal_spaces}. Let ${\mathbf D}_{1,\textnormal{trunc}} := \textnormal{diag}(\zeta_{1,n},...,\zeta_{p_1,n}) \in {\mathcal S}_{\geq 0}(p_1,\bbR)$. Then,
$$
|\boldsymbol{v}_\mathcal{U}^* {\mathbf D}_1\boldsymbol{v}_\mathcal{E}| = \Big|\sum_{\ell=1}^{p_2} \sum_{i=1}^{p_1} \beta_{\ell} \alpha_i \zeta_ i {\mathbf e}_i^*{\mathbf u}_{\ell}\Big|
=\Big|
\begin{bmatrix}
\zeta_{1,n} \alpha_1& \hdots & \zeta_{p_1,n} \alpha_{p_1}
\end{bmatrix}
\begin{bmatrix}
{\mathbf e}_1^*{\mathbf u}_1 & \hdots & {\mathbf e}_1^*{\mathbf u}_{p_2}\\
\vdots\\
{\mathbf e}_{p_1}^*{\mathbf u}_1 & \hdots & {\mathbf e}_{p_1}^*{\mathbf u}_{p_2}
\end{bmatrix}
\begin{bmatrix}
\beta_1\\
\vdots \\
\beta_{p_2}
\end{bmatrix}\Big|
$$
\begin{equation}\label{e:D_1_cross_term_bound}
=|({\mathbf D}_{1,\textnormal{trunc}}\boldsymbol\alpha)^* {\mathbf U}_{\textnormal{trunc}}\boldsymbol\beta| \leq \|{\mathbf D}_{1,\textnormal{trunc}}\boldsymbol\alpha\|\cdot \|{\mathbf U}_{\textnormal{trunc}}\boldsymbol\beta\| \leq \zeta_{p_1,n} \|\boldsymbol\alpha\| \cdot  \|{\mathbf U}_{\textnormal{trunc},p}\|_{\textnormal{op}} \hspace{0.5mm}\|\boldsymbol\beta\|,
\end{equation}
where ${\mathbf U}_{\textnormal{trunc},p}$ and ${\mathbf U}_{\textnormal{trunc},p}$ are as in \eqref{e:U-trunc,p} and \eqref{e:U-trunc,p}, respectively.
In \eqref{e:D_1_cross_term_bound}, the first and second inequalities follow, respectively, from an application of the Cauchy-Schwarz inequality and of the bound \eqref{e:||U_trunc||=<||U_trunc,p||}. Therefore, by relations \eqref{e:eigenvalues_upper_limit}, \eqref{e:||alpha||,||beta||}, \eqref{e:basis_upper_bound} (see Lemma \ref{l:orthogonal_spaces}) and \eqref{e:||U_trunc,p||=oP(1)}, it follows that
\begin{equation}\label{e:cross_term_d_1}
|\boldsymbol{v}_\mathcal{U}^* {\mathbf D}_1\boldsymbol{v}_\mathcal{E}| \leq \zeta_{p_1,n} \cdot \max\{\| \boldsymbol{v}_{\mathcal{E}}\|^2,\| \boldsymbol{v}_{\mathcal{U}}\|^2\} \cdot \|{\mathbf U}_{\textnormal{trunc},p}\|_{\textnormal{op}} = o_{\bbP}(1).
\end{equation}
Hence, the left asymptotic relation in \eqref{e:vU*D1vE=oP(1)_vE*UD2U*vU=oP(1)} is proved.

We now consider the right asymptotic relation in \eqref{e:vU*D1vE=oP(1)_vE*UD2U*vU=oP(1)}. Note that, for any $\ell=1,\hdots,p_2$, ${\mathbf U}^* \mathbf{D}_2{\mathbf U} {\mathbf u}_{\ell} = \eta_j\mathbf{u}_{\ell}$. Then,
$$
|\boldsymbol{v}_\mathcal{E}^* {\mathbf U}{\mathbf D}_2{\mathbf U}^*\boldsymbol{v}_\mathcal{U}| =  \Big|\sum_{\ell=1}^{p_2} \sum_{i=1}^{p_1} \beta_{\ell} \alpha_i ({\mathbf e}_i^*{\mathbf U}){\mathbf D}_2({\mathbf U}^*{\mathbf u}_{\ell})\Big| =
\Big|\sum_{\ell=1}^{p_2} \sum_{i=1}^{p_1} \beta_{\ell} \alpha_i \eta_{\ell} {\mathbf e}_i^*{\mathbf u}_{\ell}\Big|
$$
$$
=\Big|
\begin{bmatrix}
\eta_1 \beta_1& \hdots & \eta_{p_2} \beta_{p_2}
\end{bmatrix}
\begin{bmatrix}
{\mathbf e}_1^*{\mathbf u}_1 & \hdots & {\mathbf e}_{p_1}^*{\mathbf u}_{1}\\
\vdots\\
{\mathbf e}_{1}^*{\mathbf u}_{p_2} & \hdots & {\mathbf e}_{p_1}^*{\mathbf u}_{p_2}
\end{bmatrix}
\begin{bmatrix}
\alpha_1\\
\vdots \\
\alpha_{p_1}
\end{bmatrix}\Big|
$$
$$
=|({\mathbf D}_{2,\textnormal{trunc}}\boldsymbol\beta)^* {\mathbf U}_{\textnormal{trunc}}^*\boldsymbol\alpha| \leq \|{\mathbf D}_{2,\textnormal{trunc}}\boldsymbol\beta\| \cdot \|{\mathbf U}_{\textnormal{trunc}}^*\boldsymbol\alpha\| \leq \eta_{p_2} \|\boldsymbol\alpha\| \cdot \|{\mathbf U}_{\textnormal{trunc},p}\|_{\textnormal{op}} \cdot  \|\boldsymbol\beta\|,
$$
where ${\mathbf D}_{2,\textnormal{trunc}} = \textnormal{diag}(\eta_{1,n},...,\eta_{p_2,n})$. Hence, again by relations \eqref{e:eigenvalues_upper_limit}, \eqref{e:||alpha||,||beta||}, \eqref{e:basis_upper_bound} (see Lemma \ref{l:orthogonal_spaces}) and \eqref{e:||U_trunc,p||=oP(1)}, it follows that
\begin{equation}\label{e:cross_term_d_2}
    \boldsymbol{v}_\mathcal{E}^* {\mathbf U}{\mathbf D}_2{\mathbf U}^*\boldsymbol{v}_\mathcal{U} = o_{\bbP}(1).
\end{equation}
This proves the right asymptotic relation in \eqref{e:vU*D1vE=oP(1)_vE*UD2U*vU=oP(1)}. Hence, the proposition is established.  $\Box$\\
\end{proof}

The following lemma is used in the proof of Proposition \ref{p:Sum_of_deformed_wisharts}.
\begin{lemma}\label{l:conditional_haar}
Let ${\mathbf U}_1, {\mathbf U}_2 \in O(n)$ be independent random matrices. Further assume that ${\mathbf U}_1 \sim \textnormal{Haar}(O(n))$. Then, ${\mathbf U}_1^*{\mathbf U}_2 \sim  \textnormal{Haar}(O(n)).$
\end{lemma}
\begin{proof}
Let $\mathcal{B}$ be a set from the Borel $\sigma$-algebra defined on $O(n)$. Then, for any $\widetilde{U}_2 \in O(n)$,
\begin{equation}\label{e:conditional_haar_argument}
\bbP(  {\mathbf U}_1^* {\mathbf U}_2\in \mathcal{B} | {\mathbf U}_2 = \widetilde{U}_2) = \bbP(  {\mathbf U}_1^*\widetilde{U}_2 \in \mathcal{B} | {\mathbf U}_2 =\widetilde{U}_2 ) =
\bbP( {\mathbf U}_1^* \in \mathcal{B} | {\mathbf U}_2 = \widetilde{U}_2) = \bbP(  {\mathbf U}_1^* \in \mathcal{B}).
\end{equation}
The second equality in \eqref{e:conditional_haar_argument} follows from the fact that ${\mathbf U}_1^*\widetilde{U}_2 \sim  \textnormal{Haar}(O(n)) $ for any fixed $\widetilde{U}_2 \in O(n)$. In turn, the third equality follows from the independence of ${\mathbf U}_1$ and ${\mathbf U}_2$. Therefore,
\begin{equation*}
    \bbP(  {\mathbf U}_1^*{\mathbf U}_2 \in \mathcal{B}) = \bbE[\bbP(  {\mathbf U}_1^* {\mathbf U}_2\in \mathcal{B} | {\mathbf U}_2 )]  = \bbP(  {\mathbf U}_1^* \in \mathcal{B}).
    \end{equation*}
Therefore, ${\mathbf U}_1^*{\mathbf U}_2 \sim  \textnormal{Haar}(O(n))$, as claimed. $\Box$\\
\end{proof}

The next lemma is used in the proofs of Propositions \ref{p:Sum_of_deformed_wisharts} and \ref{p:k_sum_of_deformed_wisharts}.
\begin{lemma}\label{l:haar_basis}
For $n \in \Naturals$, let $p\in \Naturals\cup\{0\}$ be a random variable such that $p \leq n$ a.s. If $p\geq 1$, let $\mathcal{Z} \in \mathcal{M}(n, p,\bbR)$ be a random matrix with independent standard normal random entries, otherwise $\mathcal{Z} = \mathbf{0}\in\bbR^n$.  Let $\mathcal{S}_{\geq0}(n,\bbR) \ni \mathbf{D} := \textnormal{diag}\big(\lambda_1(\mathcal{Z}\mathcal{Z}^*),...,\lambda_n(\mathcal{Z}\mathcal{Z}^*)\big)$. Then,  we can write
\begin{equation}\label{e:ZZ*=UDU*}
\mathcal{Z}\mathcal{Z}^* \stackrel{d}= \mathbf{U}\mathbf{D}\mathbf{U}^*,
\end{equation}
where $\mathbf{U} \sim \textnormal{Haar}(O(n))$.
\end{lemma}
\begin{proof}
Note that, by the singular value decomposition of $\mathcal{Z}\mathcal{Z}^*$, there exists a matrix $\mathbf{U}\in O(n)$ such that $\mathcal{Z}\mathcal{Z}^* = \mathbf{U}\mathbf{D}\mathbf{U}^*$ (\textbf{n.b}: when $p=0$, $\mathcal{Z}\mathcal{Z}^* = \mathbf{U}\mathbf{D}\mathbf{U}^*$ holds for any $\mathbf{U}\in O(n)$). So, it remains to show that $\mathbf{U}\sim \textnormal{Haar}(O(n))$.

Observe that, conditionally on $p=\widetilde{p}\neq0$, $\mathcal{Z} \in \mathcal{M}(n, \widetilde{p},\bbR)$ is a random matrix with independent standard normal entries. In addition, $\mathcal{Z}$ is conditionally independent of $p$. Namely, $\big(\hspace{0.5mm}\mathcal{Z} | p = \widetilde{p} \hspace{0.5mm}\big) \stackrel{d}= \widetilde{{\mathcal Z}}$, where $\widetilde{{\mathcal Z}} \in {\mathcal M}(n,\widetilde{p},\bbR)$ is a matrix with independent, standard normal entries. Then, by the rotational invariance of the standard multivariate Gaussian, we can assume that $\mathbf{U}\sim  \textnormal{Haar}(O(n))$ conditionally on $p=\widetilde{p}\neq 0$ (see Bai and Silverstein \cite{bai:silverstein:2010}, pp.\ 334). In addition, for $\widetilde{p}=0$, we can simply let $\mathbf{U} \sim \textnormal{Haar}(O(n))$. Now let $\mathcal{B}$ be a set from the Borel $\sigma$-algebra defined on $O(n)$. Then, for any $U\sim \textnormal{Haar}(O(n))$,
$$
\bbP(\mathbf{U}\in \mathcal{B}) = \sum_{\widetilde{p}=0}^n\bbP(\mathbf{U} \in \mathcal{B} | p=\widetilde{p} )\cdot \bbP(p=\widetilde{p}) = \sum_{\widetilde{p}=0}^n\bbP(U\in \mathcal{B})\cdot\bbP(p=\widetilde{p})  = \bbP(U\in \mathcal{B}).
$$
This shows \eqref{e:ZZ*=UDU*}. $\Box$\\
\end{proof}

The next lemma contains generic facts used in the proof of Lemma \ref{l:fixed_scale_log_limit}.
\begin{lemma}\label{l:sandwich_bounds_to_log_convergence}
    Consider three sequences of real valued random variables $\{X_n\}_{n\in \Naturals}$, $\{L_n\}_{n\in \Naturals}$, and $\{U_n\}_{n\in \Naturals}$ defined on a given probability space. Moreover, assume that for every $n \in \bbN$, $X_n >0$ a.s.\ and
    \begin{equation}\label{e:sandwichboundedsequences}
        m + L_n \leq X_n \leq M + U_n \quad \textnormal{a.s.}
    \end{equation}
    for constants $M\geq m>0$. Further assume that
    \begin{equation}\label{e:Ln,Un->0_in_prob}
    L_n,U_n \stackrel{\bbP}\rightarrow 0, \quad n \rightarrow \infty.
    \end{equation}
    Then,
    \begin{equation}\label{e:sandwich_bounds_to_log_convergence}
        \frac{\log X_n}{b_n}\stackrel{\bbP}\to 0, \quad n\to \infty,
    \end{equation}
    for any sequence $\{b_n\}_{n\in\Naturals}$ such that $b_n\to \infty$ as $n\to \infty$.
\end{lemma}
\begin{proof}
For any $n \in \bbN$, let $\Omega_{n} = \{\omega: L_n > -m\}$. Also, fix $\varepsilon > 0$. Then,
$$
\bbP\Big(\frac{|\log X_n|}{b_n} > \varepsilon \Big) \leq \bbP\Big(\frac{\log X_n}{b_n} > \varepsilon \Big)   +
\bbP\Big(\frac{\log X_n}{b_n} < - \varepsilon \Big).
$$
However,
$$
\bbP\Big(\frac{\log X_n}{b_n} > \varepsilon \Big)  \leq \bbP\Big(\frac{\log(M + U_n)}{b_n} > \varepsilon \Big) \rightarrow 0, \quad n \rightarrow \infty,
$$
where the limit is a consequence of condition \eqref{e:Ln,Un->0_in_prob}. On the other hand,
$$
\bbP\Big(\frac{\log X_n}{b_n} < - \varepsilon \Big) =  \bbP\Big(\Big\{\frac{\log X_n}{b_n} < - \varepsilon \Big\} \cap \Omega_n \Big)
+ \bbP\Big(\Big\{\frac{\log X_n}{b_n} < - \varepsilon \Big\} \cap \Omega^c_n \Big)
$$
$$
\leq \bbP\Big(\Big\{\frac{\log (m + L_n)}{b_n} < - \varepsilon \Big\} \cap \Omega_n \Big) + o(1) = o(1),
$$
where the last inequality and the last equality follow, again, from condition \eqref{e:Ln,Un->0_in_prob}. This establishes \eqref{e:sandwich_bounds_to_log_convergence}. $\Box$\\
\end{proof}

The next lemma is a generic fact used in the proofs of Proposition \ref{p:Sum_of_deformed_wisharts} and Lemma \ref{l:fixed_scale_log_limit_discrete}.

\begin{lemma}\label{l:bounds_equal_in_distribution}
Consider three sequences of random variables $\{X_n\}_{n\in\Naturals}$, $\{L_n\}_{n\in\Naturals}$ and $\{U_n\}_{n\in\Naturals}$ defined on a given probability space. Moreover, assume that for every $n\in \Naturals$, $X_n >0$ a.s.\ and that there exist constants $M\geq m>0$ such that condition \eqref{e:sandwichboundedsequences} holds. Further suppose that $L_n$, $U_n$ satisfy condition \eqref{e:Ln,Un->0_in_prob}. Let $\{Y_n\}_{n\in\Naturals}$ be another sequence of random variables, possibly defined on a different probability space, and assume that
\begin{equation}
    \label{e:two_sequences_equal_in_distribution}
    Y_n\stackrel{d}=X_n \hspace{2mm}\textnormal{\textit{for each} $n\in \Naturals$}.
\end{equation}
Then, there exist sequences of random variables $\{L'_n\}_{n\in\Naturals}$, $\{U'_n\}_{n\in\Naturals}$, such that
$$
m + L'_n \leq Y_n \leq M + U'_n \quad a.s.,
$$
where $L'_n=o_\bbP(1)$, $U'_n=o_\bbP(1)$.
\end{lemma}

Before we prove Lemma \ref{l:bounds_equal_in_distribution}, we state and show the following lemma, used in the proof of the former.
\begin{lemma}\label{e:equivalent_conditions_almost_sure_bounds_that_vanish}
    Let $\{X_n\}_{n\in\Naturals}$ be a sequence of random variables defined on a given probability space. Let $m$ and $M$ be two real-valued constants such that $m<M$. Then, the following statements are equivalent.
    \begin{enumerate}
        \item[$(i)$] There exist sequences of random variables $\{L_n\}_{n\in\Naturals}$ and $\{U_n\}_{n\in\Naturals}$ satisfying \eqref{e:Ln,Un->0_in_prob} and such that
        \begin{equation}\label{e:almost_sure_bound}
            m + L_n  \leq X_n \leq M + U_n \quad a.s.
        \end{equation}
        \item[$(ii)$] For every $\eta_L, \eta_U >0$,
    \begin{equation}\label{e:good_set_converges_to_1}
            \bbP\big(m - \eta_L < X_n < M + \eta_U \big)\to 1, \quad n\to \infty.
    \end{equation}
    \end{enumerate}
\end{lemma}
\begin{proof}
 First, we establish $(i) \Rightarrow (ii)$. Suppose $(i)$ holds, and fix $\eta_L, \eta_U > 0$. Let $\Omega_{L_n} = \{ \omega: |L_n|<\eta_L\}$, $\Omega_{U_n} = \{ \omega: |U_n|<\eta_U\}$. Then,
$$
\bbP\big(m - \eta_L < X_n < M + \eta_U\big) \geq \bbP\big(\{m - \eta_L < X_n < M + \eta_U\} \cap \Omega_{L_n}\cap \Omega_{U_n}\big)
$$
\begin{equation}\label{e:P(Omega-Ln_and_Omega-Un)}
\geq \bbP\big(\{m + L_n < X_n < M + U_n \} \cap \Omega_{L_n}\cap \Omega_{U_n}\big) = \bbP\big(\Omega_{L_n}\cap \Omega_{U_n}\big) \rightarrow 1, \quad n \rightarrow \infty.
\end{equation}
In \eqref{e:P(Omega-Ln_and_Omega-Un)}, the equality follows from condition \eqref{e:almost_sure_bound}. Hence, $(ii)$ holds. \\

We now establish that $(ii)\Rightarrow(i)$. So, suppose $(ii)$ holds. For $n \in \bbN$, let $L_n := -\max\{m-X_n,0\}$, $U_n := \max\{X_n-M,0\}$. Observe that
$$
m-\max\{m-X_n,0\} \leq X_n \leq M + \max\{X_n-M,0\}\quad \textnormal{a.s.},
$$
namely, relation \eqref{e:almost_sure_bound} holds for this choice of sequences $L_n$, $U_n$. We now show that these sequences satisfy \eqref{e:Ln,Un->0_in_prob}. Indeed, fix any $\varepsilon>0$. Let $\Omega_n^M=\{\omega:X_n-M < \varepsilon/2\}$ and let $\Omega_n^m=\{\omega:m-X_n\leq \varepsilon/2 \}$. Then,
$$
 \bbP\big(\max\{X_n-M,0\}>\varepsilon\big) = \bbP\big(\{\max\{X_n-M,0\}>\varepsilon\}\cap\Omega_n^M \big) +\bbP\big(\{\max\{X_n-M,0\}>\varepsilon\}\cap(\Omega_n^M)^c\big)
$$
\begin{equation}\label{e:P(max(Xn-M,0)>epsilon)}
= \bbP\big(\{\max\{X_n-M,0\}>\varepsilon\}\cap(\Omega_n^M)^c \big) \leq \bbP\big((\Omega_n^M)^c\big) = 1 - \bbP\big(X_n < M + \varepsilon/2\big) = o(1).
\end{equation}
In \eqref{e:P(max(Xn-M,0)>epsilon)}, the second equality follows from the fact that $\bbP\big(\{\max\{X_n-M,0\}>\varepsilon\}\cap\Omega_n^M \big)= 0$, and the last equality is a consequence of \eqref{e:good_set_converges_to_1}. By an analogous reasoning, $\bbP\big(\max\{m-X_n,0\} > \varepsilon\big) = o(1)$. This establishes relation \eqref{e:Ln,Un->0_in_prob}, as claimed. Hence, $(i)$ holds, as claimed. $\Box$\\
\end{proof}

We are now in a position to establish Lemma \ref{l:bounds_equal_in_distribution}.\\

\noindent {\sc Proof of Lemma \ref{l:bounds_equal_in_distribution}}: By condition \eqref{e:two_sequences_equal_in_distribution}, for every $\eta_L, \eta_U > 0$,
    $$
    \bbP\big(m-\eta_L <X_n<M+\eta_U \big) = \bbP\big(m-\eta_L <Y_n<M + \eta_U\big).
    $$
    Moreover, by condition \eqref{e:sandwichboundedsequences} and Lemma \ref{e:equivalent_conditions_almost_sure_bounds_that_vanish}, it follows that $\bbP\big(m-\eta_L <Y_n<M+\eta_U\big)\to 1$ as $n \rightarrow \infty$. Consequently, another application of Lemma \ref{e:equivalent_conditions_almost_sure_bounds_that_vanish} provides the existence of the sequences of random variables $\{L'_n\}_{n\in\Naturals}$, $\{U'_n\}_{n\in\Naturals}$. This completes the proof. $\Box$\\

The next lemma is a simple linear algebra fact used in the proof of Lemma \ref{l:orthogonal_spaces}. Its proof is omitted.
\begin{lemma}\label{l:two_spaces_LI_equivalence}
    Fix $2 \leq p \leq n$. Consider two fixed integers $p_1,p_2 \geq 1$ such that $p_1+p_2=p \leq n$. Let $\{\mathbf{v}_1,\hdots,\mathbf{v}_{p_1}\}\subseteq \bbR^n$ and $\{\mathbf{u}_1,\hdots,\mathbf{u}_{p_2}\}\subseteq \bbR^n$ each be a collection of $p$ linearly independent vectors. Then,
    \begin{equation}\label{e:dim_of_two_LI_spaces}
    \textnormal{dim}(\textnormal{span} \{ \mathbf{v}_1,\hdots,\mathbf{v}_{p_1},\mathbf{u}_1,\hdots,\mathbf{u}_{p_2} \}) = p
    \end{equation}
    if and only if
    \begin{equation}\label{e:each_vec_is_not_in_U}
     \mathbf{v}_j\notin \textnormal{span}\{\mathbf{u}_1,\hdots,\mathbf{u}_{p_2} \}, \quad j=1,\hdots,p_1.
     \end{equation}
\end{lemma}

In the next lemma, for the reader's convenience we state and prove the classical fact that a collection of independent multivariate Gaussian random vectors are linearly independent almost surely. The lemma is used in the proofs of Lemma \ref{l:fixed_scale_log_limit}, Lemma \ref{l:fixed_scale_log_limit_discrete}, and Lemma \ref{l:orthogonal_spaces}.
\begin{lemma}\label{l:gaussian_rv_are_LI}
    For a fixed $p\leq n$, consider the collection of random vectors $\Gamma_1 \mathbf{Z}_1, \hdots, \Gamma_p \mathbf{Z}_p \in \bbR^n$, where $\bbR^n \supseteq \{\mathbf{Z}_i \}_{i=1,\hdots,p} \stackrel{\textnormal{i.i.d}}\sim \mathcal{N}({\mathbf 0},I)$ and
    \begin{equation}\label{e:Gamma-i_is_full_rank}
    \textnormal{$\Gamma_i \in  \mathcal{S}_{\geq 0}(n,\mathbb{R})$ \textit{has full rank for each $i=1,...,p$.}}
    \end{equation}
    Then,
    \begin{equation}\label{e:dim(span(Gamma1Z1,...,GammapZp)=p)}
    \textnormal{dim}(\textnormal{span}\{\Gamma_1 \mathbf{Z}_1, \hdots, \Gamma_p \mathbf{Z}_p\})= p \quad a.s.
    \end{equation}
\end{lemma}
\begin{proof}
First, we show that, for $n\geq p$,
\begin{equation}\label{e:P(Gammap_Zp_in_span)=0}
 \bbP(\Gamma_p \mathbf{Z}_p \in \textnormal{span}\{\Gamma_1\mathbf{Z}_1,\hdots,\Gamma_{p-1}\mathbf{Z}_{p-1}\})=0.
\end{equation}
In fact,
    $$
    \bbP(\Gamma_p \mathbf{Z}_p \in \textnormal{span}\{\Gamma_1\mathbf{Z}_1,\hdots,\Gamma_{p-1}\mathbf{Z}_{p-1}\}) = \bbP(\{\mathbf{Z}_p \in \textnormal{span}\{\Gamma_p^{-1} \Gamma_1\mathbf{Z}_1,\hdots,\Gamma_p^{-1}\Gamma_{p-1}\mathbf{Z}_{p-1}\}\})
    $$
    $$
    = \int_{\bbR^n\times\hdots\times\bbR^n} \bbP(\mathbf{Z}_p \in \textnormal{span}\{\Gamma_p^{-1} \Gamma_1\mathbf{z}_1,\hdots,\Gamma_p^{-1}\Gamma_{p-1}\mathbf{z}_{p-1}\}\big| \{\mathbf{Z}_i = \mathbf{z}_i\}_{i=1,\hdots,p-1})\prod_{i=1,\hdots,p-1}dF_{\mathbf{Z_i}}(\mathbf{z}_i)
    $$
    \begin{equation}\label{e:P(Gamma-p*Zp_in_span)}
    = \int_{\bbR^n\times\hdots\times\bbR^n} \bbP(\mathbf{Z}_p \in \textnormal{span}\{\Gamma_p^{-1} \Gamma_1\mathbf{z}_1,\hdots,\Gamma_p^{-1}\Gamma_{p-1}\mathbf{z}_{p-1}\})\prod_{i=1,\hdots,p-1}dF_{\mathbf{Z_i}}(\mathbf{z}_i).
    \end{equation}
In \eqref{e:P(Gamma-p*Zp_in_span)}, the first equality follows from \eqref{e:Gamma-i_is_full_rank}, and the second and third equalities follow from the law of total probability and the independence of ${\mathbf Z}_i \in \bbR^n$, $i = 1,\hdots,p$. Since $n\geq p$ and $\textnormal{dim}(\textnormal{span}\{\Gamma_p^{-1} \Gamma_1\mathbf{z}_1,\hdots,\Gamma_p^{-1}\Gamma_{p-1}\mathbf{z}_{p-1}\}) \leq p-1$ , then the Lebesgue measure of $\textnormal{span}\{\Gamma_p^{-1} \Gamma_1\mathbf{z}_1,\hdots,\Gamma_p^{-1}\Gamma_{p-1}\mathbf{z}_{p-1}\})$ is zero. It follows that, in this case, $\bbP(\mathbf{Z}_p \in \textnormal{span}\{\Gamma_p^{-1} \Gamma_1\mathbf{z}_1,\hdots,\Gamma_p^{-1}\Gamma_{p-1}\mathbf{z}_{p-1}\})=0$. Therefore, \eqref{e:P(Gammap_Zp_in_span)=0} holds, as claimed.

    To show \eqref{e:dim(span(Gamma1Z1,...,GammapZp)=p)}, we proceed by induction on the collection of vectors. First, by \eqref{e:Gamma-i_is_full_rank}, observe that
    $$
    \bbP(\Gamma_1\mathbf{Z}_1 = {\mathbf 0}) = \bbP(\mathbf{Z}_1 ={\mathbf 0}) = 0.
    $$
    Equivalently, $\textnormal{dim}(\textnormal{span}\{\Gamma_1\mathbf{Z}_1\})=1$  a.s., establishing the base case. Now, for $n \geq p$, suppose that
    \begin{equation}\label{e:dimensional_inductive_hypothesis}
    \textnormal{dim}(\textnormal{span}\{\Gamma_1\mathbf{Z}_1,\hdots,\Gamma_{p-1}\mathbf{Z}_{p-1}\})=p-1\quad \textnormal{a.s.}
    \end{equation}
    Then, as a consequence of \eqref{e:P(Gammap_Zp_in_span)=0}, $\textnormal{dim}(\textnormal{span}\{\Gamma_1\mathbf{Z}_1,\hdots,\Gamma_p\mathbf{Z}_p\})= p$ a.s. This completes the proof. $\Box$\\
\end{proof}

In Lemma \ref{l:difference_cdfs}, stated and shown next, we show the generic fact that a perturbation of a sample does not affect the asymptotic behavior of the associated e.c.d.f.. Lemma \ref{l:difference_cdfs} is used in the proofs of Theorems \ref{t:main_theorem} and \ref{t:main_theorem_discrete}.

\begin{lemma}\label{l:difference_cdfs}
Let $\{X_n\}_{n \in \bbN}$ be i.i.d random variables with common c.d.f.\ $F_X$.
\begin{itemize}
    \item[$(i)$] Let $\{Y_n\}_{n \in \bbN}$ be a sequence of random variables such that
    \begin{equation}\label{e:Yn->0_a.s.}
    Y_n \stackrel{\textnormal{a.s.}}\to 0 \hspace{3mm}\textnormal{as $n\to \infty$}.
    \end{equation}
    Then, for every $\upsilon$ that is a point of continuity of $F_X$,
\begin{equation}\label{e:difference_of_cdfs_X_a.s}
      \frac{1}{n}\sum_{\ell = 1}^n \indicator_{\{X_{\ell}+Y_n\leq \upsilon\}}  \stackrel{\textnormal{a.s.}}\rightarrow F_X(\upsilon), \quad n \rightarrow \infty.
\end{equation}
    \item[$(ii)$] Let $\{E_n\}_{n \in \bbN}$ be a sequence of random variables such that $E_n \stackrel{\bbP}\to 0$ as $n\to \infty$. Then, for every $\upsilon$ that is a point of continuity of $F_X$,
\begin{equation}\label{e:difference_of_cdfs_X_prob}
      \frac{1}{n}\sum_{\ell = 1}^n \indicator_{\{X_{\ell}+E_n\leq \upsilon\}}  \stackrel{\bbP}\rightarrow F_X(\upsilon), \quad n \rightarrow \infty.
\end{equation}
\end{itemize}
\end{lemma}
\begin{proof}
Towards establishing $(i)$, fix any $\upsilon$ that is a point of continuity of $F_X$. Note that, by Kolmogorov's strong law of large numbers, as $n \rightarrow \infty$,
\begin{equation}\label{e:indicator_conv_K's_SLLN}
\frac{1}{n}\sum_{\ell = 1}^n \indicator_{\{X_{\ell} \leq \upsilon\}} \stackrel{\textnormal{a.s.}}\rightarrow \bbP(X_{\ell}\leq \upsilon).
\end{equation}
However, let $\varepsilon > 0$. Define $\Omega_0 = \big\{\omega: \textnormal{the limits in \eqref{e:Yn->0_a.s.} and \eqref{e:indicator_conv_K's_SLLN} hold}\big\}$. In particular,
\begin{equation}\label{e:P(Omega0)=1}
\bbP(\Omega_0) = 1.
\end{equation}
Then, for $\omega \in \Omega_0$, as a result of relation \eqref{e:Yn->0_a.s.} there exists $n_0(\omega)$ such that, for $\ell = 1,\hdots,n$ and $n \geq n_0(\omega)$,
\begin{equation}\label{e:indication_ineqs}
\indicator_{\{X_{\ell}+ \varepsilon \leq \upsilon\}} \leq \indicator_{\{X_{\ell}+ |Y_n| \leq \upsilon\}} \leq \indicator_{\{X_{\ell}+ Y_n \leq \upsilon\}} \leq \indicator_{\{X_{\ell}- |Y_n| \leq \upsilon\}}\leq \indicator_{\{X_{\ell}- \varepsilon \leq \upsilon\}}.
\end{equation}
After taking the sample average $\frac{1}{n} \sum^{n}_{\ell=1}$ across expression \eqref{e:indication_ineqs}, relations \eqref{e:indicator_conv_K's_SLLN} and \eqref{e:P(Omega0)=1} imply that, as $n\rightarrow \infty$,
\begin{equation}\label{e:indicator_conv_K's_SLLN_liminf_limsup}
F_X(\upsilon-\varepsilon) \leq \liminf_{n \rightarrow \infty}\frac{1}{n}\sum^{n}_{\ell=1}\indicator_{\{X_{\ell}+ Y_n \leq \upsilon\}} \leq \limsup_{n \rightarrow \infty}\frac{1}{n}\sum^{n}_{\ell=1}\indicator_{\{X_{\ell}+ Y_n \leq \upsilon\}}
\leq F_X(\upsilon+\varepsilon) \hspace{3mm}\textnormal{a.s.}
\end{equation}
By considering a sequence $\varepsilon_i \to  0$ as $i \rightarrow \infty$, the continuity of the c.d.f.\ $F_X$ at $\upsilon$ and relation \eqref{e:indicator_conv_K's_SLLN_liminf_limsup} imply that \eqref{e:difference_of_cdfs_X_a.s} holds. This establishes $(i)$.

We now turn to $(ii)$. Fix any $\upsilon$  that is a point of continuity of $F_X$. Consider any subsequence $\frac{1}{n'}\sum^{n'}_{\ell=1}\indicator_{\{X_{\ell}+ E_{n'} \leq \upsilon\}}$. We can pass to any sub-subsequence such that $E_{n''} \stackrel{\textnormal{a.s.}}\to 0$. Thus, by $(i)$,
\begin{equation}\label{e:difference_of_cdfs_X_a.s._for_prob}
      \frac{1}{n''}\sum_{\ell = 1}^{n''} \indicator_{\{X_{\ell}+E_{n''}\leq \upsilon\}}  \stackrel{\textnormal{a.s.}}\rightarrow F_X(\upsilon), \quad n'' \rightarrow \infty.
\end{equation}
That is, for any subsequence of $\frac{1}{n}\sum^{n}_{\ell=1}\indicator_{\{X_{\ell}+ E_n \leq \upsilon\}}$, there exists a sub-subsequence that almost surely converges to $F_X(\upsilon)$. This establishes ($ii$). $\Box$\\
\end{proof}

Consider a nondecreasing random sequence $p = p(n) \in \bbN$. Let $\mathcal{Z} \in \mathcal{M}(n,p,\mathbb{R})$ be a random matrix with independent standard normal entries, and consider the random matrix $\mathcal{Z}^*\mathcal{Z}/n$. In Proposition \ref{l:max_eigenvalue_of_white_wishart}, stated and proved next, we establish the almost sure convergence of $\|\mathcal{Z}^*\mathcal{Z}/n\|_{\op}$ in the moderately high-dimensional regime. Proposition \ref{l:max_eigenvalue_of_white_wishart} is used in the proofs of Propositions \ref{p:Sum_of_deformed_wisharts} and \ref{p:k_sum_of_deformed_wisharts}.
\begin{proposition}\label{l:max_eigenvalue_of_white_wishart}
    Let
    \begin{equation}\label{e:p=p(n),p=<n}
    p = p(n) \in \bbN, \quad p \leq n,
\end{equation}
    be a nondecreasing sequence of random variables such that
\begin{equation}\label{e:p/n->0_a.s.}
p/n \to 0 \hspace{2mm} a.s.
\end{equation}
     For each $p(n)$, let $\mathcal{Z} \in \mathcal{M}(n,p(n),\mathbb{R})$ be a random matrix with i.i.d.\ standard normal entries.  Then, as $n\to \infty$,
\begin{equation}\label{e:max_eigenvalue_difference_of_white_wishart}
        \Big\|\frac{\mathcal{Z}^*\mathcal{Z}}{n}-I_{p} \Big\|_{\textnormal{op}} \stackrel{\textnormal{a.s.}}\to 0.
    \end{equation}
Moreover, as $n\to \infty$,
\begin{equation}\label{e:max_eigenvalue_of_white_wishart}
\lambda_{1}\Big(\frac{\mathcal{Z}^*\mathcal{Z}}{n}\Big),\lambda_{p}\Big(\frac{\mathcal{Z}^*\mathcal{Z}}{n}\Big),\lambda_{n}\Big(\frac{\mathcal{Z}\mathcal{Z}^*}{n}\Big) \stackrel{\textnormal{a.s.}}\to 1.
\end{equation}
\end{proposition}
\begin{proof}
  For convenience of notation, we let $\mathcal{Z}_p := \mathcal{Z}$. By condition \eqref{e:p=p(n),p=<n}, the eigenvalues on the left-hand side of \eqref{e:max_eigenvalue_of_white_wishart} are well defined. Let $\Omega_0 = \big\{\{p(n)\}_{n \in \bbN} :  \textnormal{the limit \eqref{e:p/n->0_a.s.} holds} \big\}$ so that, by \eqref{e:p/n->0_a.s.}, $\bbP(\Omega_0)=1$. Write ${\mathbf p} := \{p(n)\}_{n \in \bbN} \in \bbN^{\bbN}$, and let $\widetilde{{\mathbf p}}= \{\widetilde{p}(n)\}_{n \in \bbN} $ be a particular realization of the sequence ${\mathbf p}$. Also, let $\bbP(d\widetilde{{\mathbf p}})$ be the probability distribution of ${\mathbf p}$. Then, by conditioning on ${\mathbf p}$,
  $$
    \bbP \Big( \Big\|\frac{\mathcal{Z}_p^*\mathcal{Z}_p}{n} - I_{p} \Big\|_{\textnormal{op}} \to 0 \Big) = \int_{\Omega_0} \bbP \Big( \Big\{\Big\|\frac{\mathcal{Z}_p^*\mathcal{Z}_p}{n} - I_{p} \Big\|_{\textnormal{op}} \to 0 \Big\} \Big| {\mathbf p} = \widetilde{{\mathbf p}}\Big)\bbP(d\widetilde{{\mathbf p}})
  $$
  $$
  = \int_{\Omega_0} \bbP \Big( \Big\{\Big\|\frac{\mathcal{Z}_{\widetilde{p}}^*\mathcal{Z}_{\widetilde{p}}}{n} - I_{\widetilde{p}(n)} \Big\|_{\textnormal{op}} \to 0 \Big\} \Big| {\mathbf p}= \widetilde{{\mathbf p}}\Big)\bbP(d\widetilde{{\mathbf p}})
  $$
  \begin{equation}\label{e:convergence_Wishart_conditioning}
  = \int_{\Omega_0} \bbP \Big( \Big\{\Big\|\frac{\mathcal{Z}_{\widetilde{p}}^*\mathcal{Z}_{\widetilde{p}}}{n} - I_{\widetilde{p}(n)} \Big\|_{\textnormal{op}} \to 0 \Big\} \Big)\bbP(d\widetilde{{\mathbf p}}) = 1.
  \end{equation}
 In \eqref{e:convergence_Wishart_conditioning}, the second equality follows from the substitution principle with $\mathcal{Z}_{\widetilde{p}} \in \mathcal{M}(n,\widetilde{p}(n))$, whose entries consist of i.i.d.\ standard normal random variables. In turn, the third equality results from the independence of ${\mathbf p}$ and the entries of $\mathcal{Z}_{\widetilde{p}}$. That is, by construction, for each $\mathbf{p}$ the sequence of matrices $\mathcal{Z}_{p}$ has independent standard normal entries. Also, the last equality is a consequence of the fact that, for any deterministic sequence $\widetilde{{\mathbf p}}$ such that $\widetilde{p}(n)/n\to 0$,
    $$
    \frac{1}{2}\sqrt{\frac{n}{\widetilde{p}}}\Big\|\frac{\mathcal{Z}_{\widetilde{p}}^*\mathcal{Z}_{\widetilde{p}}}{n} - I_{\widetilde{p}} \Big\|_{\textnormal{op}} \stackrel{\textnormal{a.s.}}\to 1\quad \textnormal{as} \quad n\to \infty
    $$
    (see Chen and Pan \cite{chen:pan:2012}, Theorem 1). This shows \eqref{e:max_eigenvalue_difference_of_white_wishart}.

    In regard to \eqref{e:max_eigenvalue_of_white_wishart}, note that, by Weyl's inequality \eqref{e:Weyls_inequalities_vershynin},
$$
\sup_{1\leq \ell \leq p}\Big| \lambda_{\ell}\Big(\frac{\mathcal{Z}^*\mathcal{Z}}{n}\Big) - 1 \Big| \leq \Big\|\frac{\mathcal{Z}^*\mathcal{Z}}{n} - I_{p} \Big\|_{\textnormal{op}}.
$$
Hence, expression \eqref{e:max_eigenvalue_of_white_wishart} is a direct consequence of \eqref{e:max_eigenvalue_difference_of_white_wishart} and \eqref{e:matrix_transpose_eigen_trick}. $\Box$ \\
\end{proof}

The subsequent lemmas are used to establish properties of the rH-fBm $X_\mathcal{H}(t)$ (see \eqref{e:X_h(t)_def}) and also of the univariate, continuous-time wavelet transform
\begin{equation}\label{e:wave_coefs_univar}
\bbR^{n_{a,j}} \ni {\mathbf d}_{{\mathcal H}}(2^j) := \big\{d_{\mathcal{H}}(2^j,k) \big\}_{k=1,\hdots,n_{a,j}}  := \Big\{2^{-j/2} \int_\Reals \psi(2^{-j}t-k) X_{\mathcal{H}}(t) dt \Big\}_{k=1,\hdots,n_{a,j}}
\end{equation}
(cf.\ \eqref{e:continuous_wavelet_detail}). We start off with the former. The following lemma is used in the proofs of Lemmas \ref{l:D(2^j,k)_is_stationary} and \ref{l:D(2^j,k)_is_well_defined}.
\begin{lemma}\label{l:properties_of_Xh(t)}
Let $X_\mathcal{H} = \{X_\mathcal{H}(t)\}_{t \in \bbR}$ be as in \eqref{e:X_h(t)_def}. Then,
\begin{itemize}
\item [$(i)$] for any $s,t \in \bbR$, the conditional covariance function of $X_\mathcal{H}$ is given by
\begin{equation}\label{e:univar_cov_function}
\bbE [X_\mathcal{H}(s)X_\mathcal{H}(t)\big |\mathcal{H}=H] =  \frac{|s|^{2H}+|t|^{2H}-|t-s|^{2H}}{2}.
\end{equation}
In addition, the conditional covariance function \eqref{e:univar_cov_function} is continuous as a function of $(s,t)$;
\item [$(ii)$] $X_\mathcal{H}$ has strictly stationary increments, i.e., for any $\tau \in \bbR$,
\begin{equation}\label{e:Xh_has_strictly_stationary_increments}
 \{X_\mathcal{H}(t + \tau)-X_\mathcal{H}(\tau)\}_{t \in \bbR} \stackrel{\textnormal{f.d.d.}}= \{X_\mathcal{H}(t)\}_{t \in \bbR}.
\end{equation}
Moreover, $X_\mathcal{H}$ is conditionally self-similar, i.e., for any $c > 0$,
\begin{equation}\label{e:cond_self-similarity}
\{X_\mathcal{H}(ct)\}_{t \in \bbR} \hspace{0.5mm}\big| \hspace{0.5mm}\mathcal{H} = H \stackrel{\textnormal{f.d.d.}}= \{c^{\mathcal{H}}X_\mathcal{H}(t)\}_{t \in \bbR} \hspace{0.5mm}\big| \hspace{0.5mm}\mathcal{H} = H
\end{equation}
\end{itemize}
\end{lemma}
\begin{proof}
By definition, conditionally on $\mathcal{H} = H$, $X_{\mathcal{H}}$ is a fBm. Hence, \eqref{e:univar_cov_function} holds, as well as the continuity of the conditional covariance function. This establishes $(i)$.

We now prove $(ii)$. In fact, for any $\tau \in \bbR$, any $r \in \bbN$, any $t_1,\hdots,t_r \in \bbR$ and any $B_1,\hdots,B_r \in \mathcal{B}(\Reals)$, we can write
$$
\bbP\Bigg( \bigcap^{r}_{\ell=1} \big\{X_\mathcal{H}(t_\ell + \tau)-X_\mathcal{H}(\tau)\in B_\ell \big\}\Bigg) = \int_0^1 \bbP \Big( \bigcap^{r}_{\ell=1} \big\{X_H(t_\ell+\tau)-X_\mathcal{H}(\tau) \in B_{\ell}\big\} \Big| \mathcal{H}=H\Big) \pi(d H)
$$
$$
= \int_0^1 \bbP \Big( \bigcap^{r}_{\ell=1} \big\{X_H(t_\ell) \in B_{\ell}\big\} \Big|\mathcal{H}=H\Big) \pi(d H)
= \bbP\Bigg( \bigcap^{r}_{\ell=1} \big\{X_\mathcal{H}(t_\ell)\in B_\ell \big\}\Bigg).
$$
Equivalently, \eqref{e:Xh_has_strictly_stationary_increments} holds.
The conditional self-similarity property \eqref{e:cond_self-similarity} promptly follows from the fact that, conditionally on $\mathcal{H} = H$, $X_{\mathcal{H}}$ is a fBm. This establishes $(ii)$. $\Box$\\
\end{proof}

We now turn to the conditional properties of \eqref{e:wave_coefs_univar}. The following lemma is used in the proofs of Lemma \ref{l:fixed_scale_log_limit} and Lemma \ref{l:decorellation_sum_p}. 
\begin{lemma}\label{l:D(2^j,k)_is_stationary}
Suppose assumptions ($A1-A5$) and ($W1-W3$) hold. Fix $H \in \textnormal{supp} \ \pi(dH)$.
\begin{itemize}
    \item[$(i)$]
    Conditionally on $\mathcal{H} = H \in \textnormal{supp}(dH)$, the wavelet transform $d_{\mathcal{H}}(2^j,k)$ is well defined in the mean squared sense for any $j \in \bbN \cup \{0\}$ and $k \in \bbZ$. In addition (still conditionally on $\mathcal{H} = H$),  for any fixed $j \in \bbN \cup \{0\}$, the sequence $\{d_{\mathcal{H}}(2^j,k)\}_{k \in \bbZ}$  is strictly stationary, and for any $\kappa \in \bbZ$, it satisfies
    \begin{equation}\label{e:wavelet_covariance}
    \bbE[ d_\mathcal{H}(2^j,k+\kappa)d_\mathcal{H}(2^j,k)\big |\mathcal{H}=H]=\int_{-\pi}^\pi e^{\imag x \kappa}(2^j)^{1+2H}\alpha(H)^2 \sum_{\ell = - \infty}^\infty  \frac{|\widehat{\psi}(x + 2 \pi \ell)|^2  }{|x + 2 \pi \ell|^{1+2H}}dx.
\end{equation}

In \eqref{e:wavelet_covariance}, $\widehat{\psi}$ denotes the Fourier transform of $\psi$ and \begin{equation}\label{e:coefficient_spectral_representation_fbm}
     \alpha(H)^2 := \frac{H\Gamma(2 H) \sin(H\pi) }{\pi}.
\end{equation}
 In particular, conditionally on $\mathcal{H}=H$, $\{d_\mathcal{H}(2^j,k)\}_{k \in \bbZ} $ is wide-sense stationary with spectral density
\begin{equation}\label{e:mixed_process_spectral_density}
     f_{H}(x) := (2^j)^{1+2H}\alpha(H)^2 \sum_{\ell = - \infty}^\infty  \frac{|\widehat{\psi}(x + 2 \pi \ell)|^2 }{|x + 2 \pi \ell|^{1+2H}}.
\end{equation}
    \item[$(ii)$]  The spectral density $f_H$ as in \eqref{e:mixed_process_spectral_density} is continuous and strictly positive on $[-\pi,\pi]$. In particular, there exist constants $M_H$ and $m_H$ such that $M_H \geq m_H > 0$ and
\begin{equation}\label{e:support_bounds_for_spectral_density}
    m_{H} := \textnormal{ess} \inf_{x\in [-\pi,\pi]} f_H(x)\quad \textnormal{and} \quad M_{H} := \textnormal{ess} \sup_{x\in [-\pi,\pi]} f_H(x).
     \end{equation}
    \item[$(iii)$] For any $H \in \textnormal{supp} \hspace{0.5mm}\pi(dH)$, let
\begin{equation}\label{e:Sigma-H}
\Sigma_{H} \equiv \Sigma_{H,n_{a,j}} = \Big\{\bbE \big[d_{{\mathcal H}}(2^j,k_1) d_{{\mathcal H}}(2^j,k_2)| {\mathcal H} = H \big] \Big\}_{k_1,k_2=1,\hdots,n_{a,j}}\in {\mathcal S}_{\geq 0}(n_{a,j},\bbR)
\end{equation}
be the conditional covariance matrix of the wavelet coefficients in \eqref{e:wave_coefs_univar}. Then, $\Sigma_{H}$ has full rank.
\end{itemize}
\end{lemma}
\begin{proof}
First, we show $(i)$. As in Lemma \ref{l:properties_of_Xh(t)}, conditionally on $\mathcal{H} = H$, $X_\mathcal{H}$ is a fBm with Hurst exponent $H$.  Therefore, $d_\mathcal{H}(2^j,k)$ is well defined and the (fixed-scale) conditional process is also strictly stationary (see, for instance, Abry and Didier \cite{abry:didier:2018:dim2}, Proposition 3.1, or Abry et al.\ \cite{abry:flandrin:taqqu:veitch:2003}, expression (39)). In addition, as a consequence of Flandrin \cite{flandrin:1992}, expression (34), we can compute the autocovariance of the wavelet coefficients as
\begin{equation}\label{e:integral_rep_wave_coeff}
    \bbE[ d_\mathcal{H}(2^j,k+\kappa)d_\mathcal{H}(2^j,k)\big| \mathcal{H} =H ]  =
 \int_{\Reals} e^{\imag x \kappa} (2^j)^{1+2H}\alpha(H)^2  \frac{|\widehat{\psi}(x)|^2}{|x |^{1+2H}}dx.
\end{equation}
After a change of variables, we obtain \eqref{e:wavelet_covariance}. This proves $(i)$.

We now establish $(ii)$. For notational convenience, define the function
\begin{equation}\label{e:g_DX(x)}
g_{D_X}(x) := (2^j)^{1+2H}\alpha(H)^2\frac{ |\widehat{\psi}(x)|^2  }{|x|^{1+2H}} \indicator_{\{x \neq 0\}}.
\end{equation}
By condition \eqref{e:supp_psi=compact} (see ($W2$)),
\begin{equation}\label{e:g_DX(x)_infinitely_diff_x_neq_0}
\textnormal{$g_{D_X}(x)$ is infinitely differentiable at any $x \neq 0$.}
\end{equation}
Then, the conditional spectral density of $\{d_{{\mathcal H}}(2^j,k)\}_{k \in \bbZ}$ can be written as
\begin{equation}\label{e:f_h(x)_cond_specdens}
f_{H}(x) = \sum_{\ell = -\infty}^\infty g_{D_X}(x + 2\pi \ell), \quad x \in [-\pi,\pi].
\end{equation}
Next, we aim to establish the continuity of $f_{H}$. For small $x$, by Mallat \cite{mallat:1999}, Theorem 7.4, we can write
\begin{equation}\label{e:taylor_expansion_wavelet}
    \widehat{\psi}(x) =  \widehat{\psi}^{N_\psi}(0)x^{N_\psi} + O( x^{N_\psi} ), \quad x \rightarrow 0.
\end{equation}
In addition, observe that, for all $x \in \Reals$,
\begin{equation}\label{e:upper_bound_for_expected_density}
   \frac{(2^j)^{1+2H}\alpha(H)^2}{|x |^{2H}}  \leq C \cdot \max \{ 1,|x|^{-2}\}.
\end{equation}
In \eqref{e:upper_bound_for_expected_density}, $C = (2^j)^3 \max_{H\in(0,1)}\alpha(H)^2$, where the maximum exists as a consequence of the continuity of the function $\alpha(H)$ (see \eqref{e:coefficient_spectral_representation_fbm}). Thus, from expressions \eqref{e:g_DX(x)}, \eqref{e:taylor_expansion_wavelet} and \eqref{e:upper_bound_for_expected_density}, we obtain the bound
$$
0\leq g_{D_X}(x) \leq O(x^{2 N_\psi - 3 })
$$
for small $x$. Since $N_{\psi} \geq 2$ (see $(W1$)), then the function $g_{D_X}(x)$ is continuous at $x = 0$. Hence, in view of \eqref{e:g_DX(x)_infinitely_diff_x_neq_0}, $g_{D_X}(x)$  is continuous for all $x \in \bbR$. Now, break up the sum in the spectral density as
$$
f_{H}(x) = g_{D_X}(x) + \sum_{\ell \neq 0}g_{D_X}(x+2\pi \ell).
$$
Observe that, for $\ell \neq 0$ and $x \in [-\pi,\pi]$, by assumption \eqref{e:psihat_is_slower_than_a_power_function} (see $(W3$)) and bound \eqref{e:upper_bound_for_expected_density},
$$
g_{D_X}(x+2\pi \ell) \leq C  \cdot \max\{1,|x+2\pi \ell|^{-2} \}\cdot\frac{|\widehat{\psi}(x+2\pi \ell)|^2 }{|x + 2 \pi \ell|} \leq \frac{ C  A }{|x + 2 \pi \ell| ( 1 + | x + 2\pi \ell|)^{2 \alpha}}.
$$
Thus, by the dominated convergence theorem, the function $x \to \sum_{\ell \neq 0}g_{D_X}(x+2\pi \ell)$ is continuous for all $x \in [-\pi, \pi]$. Therefore, it follows that the spectral density $f_{H}(x)$ is continuous and bounded from above. In particular, there exists a scalar $M_H$ such that $M_H \geq f_{H}(x) $ for all $x\in [-\pi,\pi]$  as in \eqref{e:support_bounds_for_spectral_density}.

We now show that $f_{H}(x) > 0$ for any $x \in \bbR$. Assume for the moment that, for any $\varepsilon > 0$, there exists $L = L(\varepsilon) \in \bbN$ such that
\begin{equation}\label{e:sup_sum_||ell|>=L|^2<varepsilon}
\sup_{x \in [-\pi,\pi]}\sum_{|\ell| > L} |\widehat{\psi}(x + 2 \pi \ell)|^2 < \varepsilon.
\end{equation}
Now note that, for $L$ as in \eqref{e:sup_sum_||ell|>=L|^2<varepsilon}, for any $\ell$ such that $|\ell| \leq L$ and for $x \in [-\pi,\pi]$,
\begin{equation}\label{e:|x+2*pi*ell|^(1+2h)=<|pi+2*pi*L|^(1+2h)}
|x + 2 \pi \ell|^{1+2H} \leq |\pi + 2 \pi |\ell||^{1+2H} \leq |\pi + 2 \pi L|^{1+2H}\Rightarrow \frac{1}{|\pi + 2 \pi L|^{1+2H}} \leq \frac{1}{|x + 2 \pi \ell|^{1+2H}}.
\end{equation}
So, for $x \neq 0$, relation \eqref{e:orthonormality_everywhere} (see Lemma \ref{l:orthonormality_everywhere}) and the bounds \eqref{e:|x+2*pi*ell|^(1+2h)=<|pi+2*pi*L|^(1+2h)} and \eqref{e:sup_sum_||ell|>=L|^2<varepsilon} imply that
$$
\sum^{\infty}_{\ell = - \infty} \frac{ |\widehat{\psi}(x + 2 \pi \ell)|^2 }{|x + 2 \pi \ell|^{1+2H}}
\geq \sum_{|\ell| \leq L} \frac{|\widehat{\psi}(x + 2 \pi \ell)|^2 }{|\pi + 2 \pi L|^{1+2H}}
\geq \frac{1}{|\pi + 2 \pi L|^{3}}  \sum_{|\ell| \leq L} |\widehat{\psi}(x + 2 \pi \ell)|^2
\geq \frac{1-\varepsilon}{|\pi + 2 \pi L|^{3}} .
$$
On the other hand, at $x = 0$, the same argument applies, \textit{mutatis mutandis}, starting from the expression
$$
\sum^{\infty}_{\ell = - \infty} \frac{|\widehat{\psi}(2 \pi \ell)|^2 }{|2 \pi \ell|^{1+2H}}
= \sum_{\ell \neq 0} \frac{|\widehat{\psi}(2 \pi \ell)|^2 }{|2 \pi \ell|^{1+2H}}.
$$
Since $\alpha(H)^2 > 0$ for all $H \in \textnormal{supp}\ \pi (d\mathcal{H})$, $f_{H}(x)$ as in \eqref{e:mixed_process_spectral_density} is strictly positive for all $x\in[-\pi,\pi]$, as claimed. In particular, with $f_H(x)$ continuous and strictly positive on a compact set, there exists $m_H>0$ as in \eqref{e:support_bounds_for_spectral_density}.

So, we now need to show \eqref{e:sup_sum_||ell|>=L|^2<varepsilon}. By way of contradiction, suppose there exists $\varepsilon_0 > 0$ and a sequence $\{x_{L} \}_{L \in \bbN} \subseteq [-\pi,\pi]$ such that
\begin{equation}\label{e:sum>=varepsilon0_contradiction}
\sum_{|\ell| > L} |\widehat{\psi}(x_L + 2 \pi \ell)|^2 \geq \varepsilon_0.
\end{equation}
Since $[-\pi,\pi]$ is compact, without loss of generality we can assume that $x_{L}\rightarrow x_0 \in [-\pi,\pi]$ as $L \rightarrow \infty$. Fix $L' \in \bbN$. Then, for large enough $L$,
$$
\sum_{|\ell| \leq L' } |\widehat{\psi}(x_L + 2 \pi \ell)|^2 \leq \sum_{|\ell| \leq L } |\widehat{\psi}(x_L + 2 \pi \ell)|^2 \leq 1,
$$
where the second inequality is a consequence of relation \eqref{e:orthonormality_everywhere} (see Lemma \ref{l:orthonormality_everywhere}). Since $\widehat{\psi}$ is continuous (see \eqref{e:supp_psi=compact} in ($W2$)), we can sequentially take the double limit $\lim_{L'\rightarrow \infty}\liminf_{L \rightarrow \infty}$ to obtain
\begin{equation}\label{e:liminf=1}
1 = \sum^{\infty}_{\ell - \infty} |\widehat{\psi}(x_0 + 2 \pi \ell)|^2 \leq \liminf_{L \rightarrow \infty}\sum_{|\ell| \leq L } |\widehat{\psi}(x_L + 2 \pi \ell)|^2 \leq 1,
\end{equation}
where the equality again follows from \eqref{e:orthonormality_everywhere} (see Lemma \ref{l:orthonormality_everywhere}). Therefore,
$$
\varepsilon_0 + 1 \leq \liminf_{L \rightarrow \infty}\sum_{|\ell| > L} |\widehat{\psi}(x_L + 2 \pi \ell)|^2 + \liminf_{L \rightarrow \infty}\sum_{|\ell| \leq L }|\widehat{\psi}(x_L + 2 \pi \ell)|^2 = 1,
$$
where the inequality stems from \eqref{e:sum>=varepsilon0_contradiction} and \eqref{e:liminf=1}, and the equality is again a consequence of \eqref{e:orthonormality_everywhere} (contradiction). This proves \eqref{e:sup_sum_||ell|>=L|^2<varepsilon}. Hence, ($ii$) is established.

We now turn to $(iii)$. Observe that $\Sigma_H$ as in \eqref{e:Sigma-H}  is a symmetric Toeplitz matrix (see Definition \ref{d:Teoplitz}). Therefore,  from part $(ii)$ of this lemma (see also expressions \eqref{e:gray_toeplitz_bound} and \eqref{e:eigenvalue_bounds_toeplitz}), it follows that  $\Sigma_H$ is nonsingular for every $n_{a,j}$. This establishes $(iii)$. $\Box$\\
\end{proof}

Next, we consider unconditional properties of the wavelet transform vectors. The following lemma proves that $d_\mathcal{H}(2^j,k)$ and $D_X(2^j,k)$ are (unconditionally) well defined in the mean squared sense (cf.\ Lemma \ref{l:D(2^j,k)_is_stationary}). It is used in the proof of Theorem \ref{t:main_theorem}.
\begin{lemma}\label{l:D(2^j,k)_is_well_defined}
Suppose assumptions ($A1-A5$) and ($W1-W3$) hold. Then, for any $j \in \bbN \cup \{0\}$ and $k \in \bbZ$ we have the following.
\begin{itemize}
    \item[$(i)$] The wavelet transform $d_\mathcal{H}(2^j,k)$ (see \eqref{e:wave_coefs_univar}) is well defined in the mean squared sense. 
    \item[$(ii)$]  The multivariate wavelet transform $D_X(2^j,k)$ (see \eqref{e:continuous_wavelet_detail}) is well defined in the mean squared sense. 
   \end{itemize}
\end{lemma}
\begin{proof}
First, we show $(i)$. Note that, by an application of the dominated convergence theorem, the mapping $\bbR \ni t \mapsto \bbE |t|^{{\mathcal H}}$ is continuous. Thus, in view of Lemma \ref{l:properties_of_Xh(t)}, $(i)$, the covariance function $\bbE X_\mathcal{H}(s)X_\mathcal{H}(t)$ is continuous. Therefore, by Cram\'{e}r and Leadbetter \cite{cramer:leadbetter:1967}, p.\ 86, it suffices to show that
\begin{equation}
\label{e:CramerLeadbetter_condition_un_conditioned}
    \int_{\bbR}\int_{\bbR} \big|\bbE X_\mathcal{H}(s)X_\mathcal{H}(t)| |\psi(s)||\psi(t)| ds dt < \infty.
\end{equation}
In fact, for any $s,t\in\Reals$, expression \eqref{e:univar_cov_function} (see Lemma \ref{l:properties_of_Xh(t)}) implies that
$$
|\bbE X_\mathcal{H}(s)X_\mathcal{H}(t)|= \big|\bbE \big[ \bbE \big(X_{\mathcal{H}}(s)X_{\mathcal{H}}(t) | \mathcal{H}\big) \big]\big|
$$
$$
\leq  \bbE\big( \sqrt{\bbE (X_{\mathcal{H}}(s)^2| \mathcal{H})} \sqrt{\bbE( X_{\mathcal{H}}(t)^2|\mathcal{H})} \big)  = \big| \bbE |s|^\mathcal{H} |t|^\mathcal{H} \big| \leq \max\{1,|s|\}\cdot \max\{1,|t|\} .
$$
Since $\psi$ is integrable with compact support (see assumptions $(W1)$ and $(W2)$), \eqref{e:CramerLeadbetter_condition_un_conditioned} holds. This proves $(i)$.

We now consider the multivariate setting. Recall that, by assumption $(A2)$, $X$ is a $p$-variate process composed of $p$ independent rH--fBms. Thus, conditionally on $\mathds{H}_n = \mathbb{H}_n$, $X$ is an ofBm with Hurst exponent (matrix) $\mathbb{H}_n$. Therefore, a natural adaptation of the proof of part $(i)$ of this lemma establishes $(ii)$. $\Box$\\
\end{proof}

The following lemma is used in the proof of Lemma \ref{l:fixed_scale_log_limit}. The lemma establishes properties of the (unconditional) dependence structure of ${\mathbf d}_{{\mathcal H}}(2^j)$ as in see \eqref{e:wave_coefs_univar}.
\begin{lemma}\label{l:decorellation_sum_p}
Suppose assumptions ($A1-A5$) and ($W1-W3$) hold. Fix $j \in \bbN \cup \{0\}$ and consider the random vector $ {\mathbf d}_{{\mathcal H}}(2^j)$ as in \eqref{e:wave_coefs_univar}. For $H \in \textnormal{supp} \hspace{0.5mm}\pi(dH)$ and $\Sigma_H$ as in \eqref{e:Sigma-H}, let
\begin{equation}\label{e:Gamma-H_cond}
\Gamma_{H} := \Sigma^{1/2}_{H} \in {\mathcal S}_{\geq 0}(n_{a,j},\bbR).
\end{equation}
Now define the random matrix
\begin{equation}\label{e:Gamma-mathcal_H}
\Gamma_{{\mathcal H}} \in {\mathcal S}_{\geq 0}(n_{a,j},\bbR)
\end{equation}
by replacing $H$ with ${\mathcal H}$ in \eqref{e:Gamma-H_cond}. The following claims hold.
\begin{itemize}
\item [$(i)$] The random matrix
\begin{equation}\label{e:Gamma^(-1)_H_is_well_defined}
\Gamma^{-1}_{{\mathcal H}} \textnormal{\textit{ is well defined a.s.}}
\end{equation}
and
\begin{equation}\label{e:Gamma^(-1)d-vec_H_N(0,I)}
\bbR^{n_{a,j}} \ni \Gamma^{-1}_{{\mathcal H}}{\mathbf d}_{{\mathcal H}}(2^j)  \stackrel{d}= {\mathcal N}({\mathbf 0},I).
\end{equation}
\item [$(ii)$] $\Gamma_{{\mathcal H}}$ and $\Gamma^{-1}_{{\mathcal H}} {\mathbf d}_{{\mathcal H}}(2^j)$ are conditionally independent. In other words, fix any (Borel) sets
    \begin{equation}\label{e:B1_in_Borel(S>=0)_B2_in_B(R^n)}
    B_1 \in {\mathcal B}({\mathcal S}_{\geq 0}(n_{a,j},\bbR)), \quad B_2 \in {\mathcal B}(\bbR^{n_{a,j}}).
    \end{equation}
    Then,
\begin{equation}\label{e:Gamma_Gamma^(-1)d_are_conditional_indep}
\bbP\big(  \Gamma_{{\mathcal H}} \in B_1 \hspace{0.5mm}\cap  \hspace{0.5mm}\Gamma^{-1}_{{\mathcal H}} {\mathbf d}_{{\mathcal H}}(2^j) \in B_2 \big| {\mathcal H} \big)
= \bbP\big(  \Gamma_{{\mathcal H}} \in B_1  \big| {\mathcal H}\big) \cdot \bbP\big( \Gamma^{-1}_{{\mathcal H}} {\mathbf d}_{{\mathcal H}}(2^j) \in B_2 \big| {\mathcal H}\big).
\end{equation}
\item [$(iii)$] $\Gamma_{{\mathcal H}}$ and $\Gamma^{-1}_{{\mathcal H}} {\mathbf d}_{{\mathcal H}}(2^j)$ are (unconditionally) independent.
\item [$(iv)$] Let $D_X(2^j)$ be the random matrix appearing in \eqref{e:DX(a(n)2^j,k)}. Let
\begin{equation}\label{e:p-tilde_i(n)}
p_i = p_i(n), \quad i = 1,\hdots,n,
\end{equation}
be the random variable which counts the number of random variables $\mathcal{H}_\ell$, $\ell = 1,\hdots,p$, equal to $\breve{H}_i$. Then, we can write
\begin{equation}\label{e:decorrelation_sum_p_claim}
\frac{1}{n_{a,j}} D_X(2^j)^*D_X(2^j) \stackrel{d} =
\frac{1}{n_{a,j}} \sum_{i=1}^r \Gamma_{\breve{H}_i}\mathcal{Z}_i\mathcal{Z}_i^*\Gamma_{\breve{H}_i}.
\end{equation}
In \eqref{e:decorrelation_sum_p_claim}, $\mathcal{Z}_i \in {\mathcal M}(n_{a,j},p_i,\bbR)$, $i = 1,\hdots,r$, are independent random matrices containing only independent standard normal entries if $p_i \geq 1$, and $\mathcal{Z}_i = {\mathbf 0}$ if $p_i = 0$.
\end{itemize}
\end{lemma}
\begin{proof}
To show $(i)$, first note that \eqref{e:Gamma^(-1)_H_is_well_defined} is a consequence of Lemma \ref{l:D(2^j,k)_is_stationary}, $(iii)$. Now let $B \in {\mathcal B}(\bbR^{n_{a,j}}) $. Recall that, by Lemma \ref{l:D(2^j,k)_is_stationary}, the zero-mean random vector ${\mathbf d}_{{\mathcal H}}(2^j) \in \bbR^{n_{a,j}}$ is Gaussian conditionally on ${\mathcal H} = H \in \textnormal{supp}\hspace{0.5mm}\pi(dH)$. Then,
\begin{equation}\label{e:P(Gamma^(-1)d-vec_in_B|H=H)=P(Z_in_B)}
\bbP\big(\Gamma^{-1}_{{\mathcal H}}{\mathbf d}_{{\mathcal H}}(2^j) \in B \big| {\mathcal H} = H\big)
= \bbP\big(\Gamma^{-1}_{H}{\mathbf d}_{H}(2^j) \in B \big| {\mathcal H} = H\big)
= \bbP\big({\mathbf Z} \in B \big| {\mathcal H} = H\big) = \bbP\big({\mathbf Z} \in B \big),
\end{equation}
where $\bbR^{n_{a,j}} \ni {\mathbf Z} \sim {\mathcal N}({\mathbf 0},I)$. Then,
$$
\bbP\big(\Gamma^{-1}_{{\mathcal H}}{\mathbf d}_{{\mathcal H}}(2^j) \in B\big) = \bbE \big[\bbP\big(\Gamma^{-1}_{{\mathcal H}}{\mathbf d}_{{\mathcal H}}(2^j) \in B \big| {\mathcal H}\big)\big] = \bbP\big({\mathbf Z} \in B\big).
$$
This establishes \eqref{e:Gamma^(-1)d-vec_H_N(0,I)}, and hence, $(i)$.

We now turn to $(ii)$. Fix $H \in \textnormal{supp}\hspace{0.5mm}\pi(dH)$ and consider $B_1$, $B_2$ as in \eqref{e:B1_in_Borel(S>=0)_B2_in_B(R^n)}. Then,
\begin{equation}\label{e:Gamma-H_Gamma^(-1)-H_cond_indep}
\bbP\big(  \Gamma_{{\mathcal H}} \in B_1 \hspace{0.5mm}\cap  \hspace{0.5mm}\Gamma^{-1}_{{\mathcal H}} {\mathbf d}_{{\mathcal H}}(2^j) \in B_2 \big| {\mathcal H} = H \big) =
\bbP\big(  \Gamma_{H} \in B_1 \hspace{0.5mm}\cap  \hspace{0.5mm}\Gamma^{-1}_{H} {\mathbf d}_{H}(2^j) \in B_2 \big| {\mathcal H} = H \big).
\end{equation}
We arrive at two possible cases. First, suppose $\Gamma_{H} \in B_1$. Then, the right-hand side of \eqref{e:Gamma-H_Gamma^(-1)-H_cond_indep} can be rewritten as
$$
\bbP\big( \Gamma^{-1}_{H} {\mathbf d}_{H}(2^j) \in B_2 \big| {\mathcal H} = H \big)= \bbP\big(  \Gamma_{H} \in B_1  \big| {\mathcal H} = H \big)\cdot \bbP\big( \Gamma^{-1}_{H} {\mathbf d}_{H}(2^j) \in B_2 \big| {\mathcal H} = H \big).
$$
Alternatively, suppose $\Gamma_{H} \notin B_1$. Then, the right-hand side of \eqref{e:Gamma-H_Gamma^(-1)-H_cond_indep} can be recast in the form
$$
\bbP\big(  \Gamma_{H} \in B_1 \hspace{0.5mm}\cap  \hspace{0.5mm}\Gamma^{-1}_{H} {\mathbf d}_{H}(2^j) \in B_2 \big| {\mathcal H} = H \big) = 0
$$
$$
= \bbP\big(  \Gamma_{H} \in B_1  \big| {\mathcal H} = H \big)\cdot \bbP\big( \Gamma^{-1}_{H} {\mathbf d}_{H}(2^j) \in B_2 \big| {\mathcal H} = H \big).
$$
Thus, in any case, the left-hand side of \eqref{e:Gamma-H_Gamma^(-1)-H_cond_indep} equals
$$
\bbP\big(  \Gamma_{{\mathcal H}} \in B_1  \big| {\mathcal H} = H \big)\cdot \bbP\big( \Gamma^{-1}_{{\mathcal H}} {\mathbf d}_{{\mathcal H}}(2^j) \in B_2 \big| {\mathcal H} = H \big).
$$
This establishes $(ii)$.

To show $(iii)$, consider any sets $B_1$, $B_2$ as in \eqref{e:B1_in_Borel(S>=0)_B2_in_B(R^n)}. Then, by relations \eqref{e:Gamma_Gamma^(-1)d_are_conditional_indep}, \eqref{e:P(Gamma^(-1)d-vec_in_B|H=H)=P(Z_in_B)} and \eqref{e:Gamma^(-1)d-vec_H_N(0,I)},
$$
\bbP\big(  \Gamma_{{\mathcal H}} \in B_1  \hspace{0.5mm}\cap \hspace{0.5mm} \Gamma^{-1}_{{\mathcal H}} {\mathbf d}_{{\mathcal H}}(2^j) \in B_2\big) =
\bbE \big[\bbP\big(  \Gamma_{{\mathcal H}} \in B_1  \hspace{0.5mm}\cap \hspace{0.5mm} \Gamma^{-1}_{{\mathcal H}} {\mathbf d}_{{\mathcal H}}(2^j) \in B_2 \big| {\mathcal H}\big)\big]
$$
$$
= \bbE \big[\bbP\big(  \Gamma_{{\mathcal H}} \in B_1  \big| {\mathcal H}\big)\cdot
\bbP\big(\Gamma^{-1}_{{\mathcal H}} {\mathbf d}_{{\mathcal H}}(2^j) \in B_2 \big| {\mathcal H}\big)\big]
= \bbE \big[\bbP\big(  \Gamma_{{\mathcal H}} \in B_1  \big| {\mathcal H}\big)\big] \cdot \bbP\big({\mathbf Z}\in B_2\big)
$$
$$
= \bbP\big(  \Gamma_{{\mathcal H}} \in B_1\big) \cdot \bbP\big({\mathbf Z}\in B_2\big)
= \bbP\big(  \Gamma_{{\mathcal H}} \in B_1\big)  \cdot \bbP\big(\Gamma^{-1}_{{\mathcal H}} {\mathbf d}_{{\mathcal H}}(2^j) \in B_2 \big).
$$
This establishes $(iii)$.

We now turn to $(iv)$. By the independence of the $p$ rH-fBms that make up $\{X(t)\}_{t \in {\mathcal T}}$ (see assumption (A2) and expression \eqref{e:Y=PX}) and by parts $(i)$ and $(iii)$ in this lemma, we can rewrite
\begin{equation}\label{e:mixed_gaussian_representation}
D_X(2^j)^* \stackrel{d}= \Big( \Gamma_{{\mathcal H}_1} {\mathbf Z}_1, \hdots,\Gamma_{{\mathcal H}_p} {\mathbf Z}_p \Big),
\end{equation}
where $\bbR^{n_{a,j}} \supseteq {\mathbf Z}_1,\hdots,{\mathbf Z}_p \stackrel{\textnormal{i.i.d.}}\sim {\mathcal N}({\mathbf 0},I)$ and $\Gamma_{{\mathcal H}_\ell}$, $\ell = 1,\hdots,p$, are as in \eqref{e:Gamma-mathcal_H}. Now, for each $i = 1,\hdots,r$, let $p_i = p_i(n)$  be as in \eqref{e:p-tilde_i(n)}. Note that $\sum^{r}_{i=1} p_i(n) = p(n)$ a.s. Then,
$$
{\mathcal S}_{\geq 0}(n_{a,j},\bbR) \ni \frac{1}{n_{a,j}}D_X(2^j)^*D_X(2^j) \stackrel{d}= \frac{1}{n_{a,j}}\sum^{p}_{\ell=1} \Gamma_{{\mathcal H}_\ell}{\mathbf Z}_{\ell}{\mathbf Z}^*_{\ell}\Gamma_{{\mathcal H}_\ell}
$$
\begin{equation}\label{e:decorrelation_sum_p}
  =  \frac{1}{n_{a,j}} \sum_{i=1}^r \sum_{\ell=1}^{p_i} \Gamma_{\breve{H}_i}\mathbf{Z}_{\ell,i}\mathbf{Z}_{\ell,i}^*\Gamma_{\breve{H}_i} = \frac{1}{n_{a,j}} \sum_{i=1}^r \Gamma_{\breve{H}_i}\mathcal{Z}_i\mathcal{Z}_i^*\Gamma_{\breve{H}_i},
\end{equation}
where $\mathcal{Z}_1,\hdots,\mathcal{Z}_r$ are as in \eqref{e:decorrelation_sum_p_claim}. In \eqref{e:decorrelation_sum_p}, for $i = 1,\hdots,r$ and $\ell = 1,\hdots,p_i$, $\mathbf{Z}_{\ell,i} \in \bbR^{n_{a,j}}$ are independent random vectors with independent standard normal entries if $p_i \geq 1$, and $\mathbf{Z}_{\ell,i} = {\mathbf 0}$, otherwise. Also, $\mathcal{Z}_i$ is as in \eqref{e:decorrelation_sum_p_claim}. This establishes \eqref{e:decorrelation_sum_p_claim}, and hence, $(iv)$. $\Box$\\
\end{proof}

The next lemma is used in the proof of Theorem \ref{t:main_theorem}. Accordingly, the lemma pertains to the continuous-time framework and establishes the scaling properties of $\{D_X(a(n)2^j,k)\}_{k\in\bbZ} $ and $\mathbf{W}_X(a(n)2^j)$. For the lemma, recall that ${\mathds H}_n$ is given by \eqref{e:mathds_hn}.
\begin{lemma}\label{l:W(a(n)2^j)_D(a(n)2^j,k)_selfsimilar}
Suppose assumptions ($A1-A5$) and ($W1-W3$) hold. Fix $j\in \Naturals \cup \{0\}$.
\begin{enumerate}
   \item[$(i)$] The wavelet transform  $D_X(a(n)2^j,k)$ as in \eqref{e:DX(a(n)2^j,k)} satisfies a random operator self-similarity relation. That is,
\begin{equation}\label{e:D(a(n)2^j,k)_self_similarity}
    \{D_X(a(n)2^j,k)\}_{k\in\bbZ} \stackrel{\textnormal{f.d.d.}}= \{a(n)^{\mathds{H}_n+1/2\cdot I}D_X(2^j,k) \}_{k\in\bbZ}.
\end{equation}
 \item[$(ii)$] The wavelet random matrix $\mathbf{W}_X(a(n)2^j)$ as in \eqref{e:WX(a(n)2^j)} satisfies a random operator self-similarity relation. That is,
\begin{equation}\label{e:W(a(n)2^j)_self_similarity}
    \mathbf{W}_X(a(n)2^j) \stackrel{d}= a(n)^{\mathds{H}_n+1/2\cdot I } \mathbf{W}_X(2^j) a(n)^{\mathds{H}_n+1/2\cdot I}.
\end{equation}


\end{enumerate}
\end{lemma}
\begin{proof}
By \eqref{e:continuous_wavelet_detail} and Lemma \ref{l:D(2^j,k)_is_well_defined}, $(ii)$, the wavelet transform of the stochastic process $X$ is given by
$$
D_X(a(n)2^j,k)
=  (a(n)2^{j})^{1/2} \int_{\Reals} \psi(u)X\big(a(n)2^{j}(u+k)\big)du ,
$$
where we make the change of variables $u = (a(n)2^{j})^{-1}t-k$. However, conditionally on $\mathds{H}_n = \mathbb{H}_n$, $X$ is an ofBm with Hurst (matrix) exponent $\mathbb{H}_n$. Therefore, making use of the operator self-similarity of $X$ (see \eqref{e:def_ss}), we obtain
$$
\big\{D_X(a(n)2^j,k)\big|\mathds{H}_n =  \mathbb{H}_n \big\}_{k \in \bbZ}  \stackrel{\textnormal{f.d.d.}}= \Big\{a(n)^{ \mathbb{H}_n+1/2\cdot I} 2^{j/2} \int_{\Reals} \psi(u)X\big (2^{j}(u+k)\big)du \Big\}_{k \in \bbZ}.
$$
Thus, by a second change of variables $t = 2^j(u+k)$, we obtain
\begin{equation}\label{e:conditional_D(a(n)2^j,k)_self_similar}
   \big\{ D_X(a(n)2^j,k)\big|\mathds{H}_n = \mathbb{H}_n \big\}_{k \in \bbZ} \hspace{2mm}\stackrel{\textnormal{f.d.d.}}=\hspace{2mm} \big\{ a(n)^{\mathds{H}_n +1/2 \cdot I} D_X(2^j,k)\big|\mathds{H}_n = \mathbb{H}_n \big\}_{k \in \bbZ} .
\end{equation}
Now, for any $m \in \bbN$, consider a collection of sets $B_1,\hdots,B_m \in \mathcal{B}(\bbR^p)$ and integers $k_1,\hdots,k_m$. Then,
$$
\bbP\Big( \bigcap^{m}_{\ell=1} \{D_X(a(n)2^j,k_\ell)\in  B_\ell \} \Big) = \int_{{\mathcal S}_{\geq 0}(p(n),\bbR)} \bbP \Big(  \bigcap^{m}_{\ell=1} \{D_X(a(n) 2^j,k_\ell)\in  B_\ell \} \Big|\mathds{H}_n = \mathbb{H}_n\Big) \bbP_{\mathds{H}_n}(d \hspace{0.5mm}\mathbb{H}_n)
$$
$$
=\int_{{\mathcal S}_{\geq 0}(p(n),\bbR)} \bbP \Big( \bigcap^{m}_{\ell=1} \{a(n)^{\mathbb{H}_n + 1/2 \cdot I }D_X(2^j,k_\ell)\in  B_\ell \}   \Big|\mathds{H}_n= \mathbb{H}_n\Big) \bbP_{\mathds{H}_n}(d \hspace{0.5mm}\mathbb{H}_n)
$$
$$
=\bbP\Big( \bigcap^{m}_{\ell=1} \{a(n)^{\mathds{H}_n+1/2 \cdot I} D_X(2^j,k_\ell)\in  B_\ell \} \Big),
$$
where the second equality follows from \eqref{e:conditional_D(a(n)2^j,k)_self_similar}. This establishes $(i)$.
Moreover, since
$$
\mathbf{W}_X(a(n)2^j) = \frac{1}{n_{a,j}} \sum^{n_{a,j}}_{k=1} D_X(a(n)2^j,k)D_X(a(n)2^j,k)^*,
$$
then $(ii)$ is a direct consequence of $(i)$. $\Box$\\
\end{proof}

 The following result is a technical lemma used in the proof of Lemma \ref{l:D(2^j,k)_is_stationary}.
\begin{lemma}\label{l:orthonormality_everywhere}
Under the assumptions of Lemma \ref{l:D(2^j,k)_is_stationary},
\begin{equation}\label{e:orthonormality_everywhere}
\sum^{\infty}_{\ell = -\infty} |\widehat{\psi}(x + 2 \pi \ell)|^2 = 1 \quad \textnormal{for all } x \in [-\pi,\pi).
\end{equation}
\end{lemma}
\begin{proof}
Recall that the orthonormality of $\{\psi(t-k)\}_{k \in \bbZ}$ is equivalent to the condition that $ \sum^{\infty}_{\ell = -\infty} |\widehat{\psi}(x + 2 \pi \ell)|^2 = 1$ a.e. (see Daubechies \cite{daubechies:1992}, p.\ 134). From condition \eqref{e:psihat_is_slower_than_a_power_function} (see $(W3)$), there exist constants $0 < A < \infty$ and $\alpha > 1$ such that, for each $x \in \Reals$ and $\ell \in \bbZ$, $|\widehat{\psi}(x + 2 \pi \ell)|^2 \leq A^2 / (1 + |x + 2 \pi \ell|)^{2 \alpha}.$ Since $\widehat{\psi}$ is continuous (see conditions \eqref{e:N_psi} and \eqref{e:supp_psi=compact} in ($W1$) and ($W2$)), then, by the dominated convergence theorem, the function $x \mapsto \sum^{\infty}_{\ell=-\infty} |\widehat{\psi}(x + 2 \pi \ell)|^2$ is continuous. Hence, \eqref{e:orthonormality_everywhere} holds. $\Box$\\
\end{proof}

The following lemma is used in the proof of Lemma \ref{l:fixed_scale_log_limit_discrete}. In particular, for $X_{{\mathcal H}}$ as in \eqref{e:X_h(t)_def}, it establishes properties of the conditional covariance matrix of the discrete-time wavelet coefficients as given by \eqref{e:disc2}, namely,
\begin{equation}\label{e:wave_coefs_univar_discrete}
\bbR^{n_{a,j}} \ni \widetilde{\mathbf d}_{{\mathcal H}}(2^j) := \big\{\widetilde{d}_{\mathcal{H}}(2^j,k) \big\}_{k=1,\hdots,n_{a,j}}  := \Big\{\sum_{\ell\in\bbZ} X_{\mathcal{H}}(\ell)h_{j,2^jk-\ell} \Big\}_{k=1,\hdots,n_{a,j}}.
\end{equation}
\begin{lemma}\label{l:covariance_of_univariate_discrete_wavelet_coeff}
    Suppose assumptions $(A1-A5)$ and $(W1-W3)$ hold. For $j \in \bbN \cup \{0\}$, let $\widetilde{\mathbf d}_{{\mathcal H}}(2^j,k)$ be the univariate discrete-time wavelet coefficient as in \eqref{e:wave_coefs_univar_discrete}
    In addition, for any $H \in \textnormal{supp} \hspace{0.5mm}\pi(dH)$, let
\begin{equation}\label{e:Sigma-H-discrete}
\widetilde{\Sigma}_{H} \equiv \widetilde{\Sigma}_{H,n_{a,j}} = \Big\{\bbE \big[\widetilde{d}_{{\mathcal H}}(2^j,k_1) \widetilde{d}_{{\mathcal H}}(2^j,k_2)| {\mathcal H} = H \big] \Big\}_{k_1,k_2=1,\hdots,n_{a,j}}\in {\mathcal S}_{\geq 0}(n_{a,j},\bbR)
\end{equation}
be the conditional covariance matrix of the discrete-time wavelet coefficients in \eqref{e:wave_coefs_univar}. Then, $\widetilde{\Sigma}_{H}$ has full rank, and conditionally on $\mathcal{H} = H$,
$$
\widetilde{\mathbf d}_{{\mathcal H}}(2^j) \big| \mathcal{H} = H\  \stackrel{\textnormal{a.s.}}= \widetilde{\Gamma}_{H}\mathbf{Z},
$$
 where $\mathbf{Z} \in \bbR^{n_{a,j}}$ is a standard normal random vector and $\widetilde{\Gamma}_{H}^2 =\widetilde{\Sigma}_{H} $.
\end{lemma}
\begin{proof} By definition, conditionally on $\mathcal{H} = H$, $X_\mathcal{H}$ is a fBm with Hurst exponent $H$. Hence, also conditionally, the random vector $\widetilde{\mathbf d}_{{\mathcal H}}(2^j)$ is Gaussian. By expressions (A.10) and (A.11) in Boniece et al.\  \cite{boniece:didier:sabzikar:2021}, $\bbE[ \widetilde{d}^2_\mathcal{H}(2^j,0)\big |\mathcal{H}=H] > 0$. Equivalently, the main diagonal entries of $\widetilde{\Sigma}_H$ are (equal and) strictly positive. Moreover, $\bbE[ \widetilde{d}_\mathcal{H}(2^j,\kappa)\widetilde{d}_\mathcal{H}(2^j,0)\big |\mathcal{H}=H] \rightarrow 0$ as $\kappa \rightarrow \infty$ (see, for instance, Boniece and Didier \cite{boniece:didier:sabzikar:2021}, Proposition A.3, $(ii)$). Therefore, by Proposition 5.1.1 in Brockwell and Davis \cite{brockwell:davis:1991}, the matrix $\widehat{\Sigma}_H$ as in \eqref{e:Sigma-H-discrete} is nonsingular for every $n_{a,j}$. This completes the proof. $\Box$\\
\end{proof}

\bibliography{highdim}

\end{document}